\documentclass[12pt, a4paper]{article}
\usepackage{amsmath,amsthm,amssymb}
\usepackage[top=100pt,bottom=70pt,left=58pt,right=66pt]{geometry}
\usepackage{enumerate}
\usepackage[all]{xy}
\usepackage{seqsplit}
\usepackage{xstring}
\usepackage[nottoc,numbib]{tocbibind}

\newcommand{\Ainfty}{$\mathrm{A}(\!\infty\!)$}
\makeatletter
\DeclareOldFontCommand{\rm}{\normalfont\rmfamily}{\mathrm}
\DeclareOldFontCommand{\sf}{\normalfont\sffamily}{\mathsf}
\DeclareOldFontCommand{\tt}{\normalfont\ttfamily}{\mathtt}
\DeclareOldFontCommand{\bf}{\normalfont\bfseries}{\mathbf}
\DeclareOldFontCommand{\it}{\normalfont\itshape}{\mathit}
\DeclareOldFontCommand{\sl}{\normalfont\slshape}{\@nomath\sl}
\DeclareOldFontCommand{\sc}{\normalfont\scshape}{\@nomath\sc}
\makeatother

\def\titlerunning#1{\gdef\titrun{#1}}
\makeatletter
\def\author#1{\gdef\autrun{\def\and{\unskip, }#1}\gdef\@author{#1}}
\def\address#1{{\def\and{\\\hspace*{18pt}}\renewcommand{\thefootnote}{}%
\footnote {#1}}%
\markboth{Marc Cabanes and Britta Sp\"ath}{\titrun}}
\makeatother
\def\email#1{e-mail: #1}
\def\subjclass#1{{\renewcommand{\thefootnote}{}%
\footnote{\emph{Mathematics Subject Classification (2010):} #1}}}

\newcommand{\bGzwei}{{{\mathbf G}_{2}}}

\newcommand{\otw}{\text{otherwise}}

\newtheorem{thm}{Theorem}
\newtheorem{lem}[thm]{Lemma}

\newtheorem{mainthm}{Theorem}
\newtheorem{theo}[thm]{Theorem}

\theoremstyle{definition}

\newtheorem{rem}[thm]{Remark}

\theoremstyle{plain}
 \newtheorem{prop}[thm]{Proposition}

 \newtheorem{cor}[thm]{Corollary}

\theoremstyle{definition}
 \newtheorem{defi}[thm]{Definition}{\bf}{}
 \newtheorem{notation}[thm]{Notation}
 
\numberwithin{equation}{thm}
\numberwithin{thm}{section}

\def\AA#1{{{\rm A}$(#1)$}}\def\BB#1{{{\rm B}$(#1)$}}

\def\norm#1#2{{\operatorname N}_{#1}(#2)}
\def\cent#1#2{{\operatorname C}_{#1}(#2)}

\newcommand{\Id}{\operatorname {Id}}

\newcommand{\w}{\widetilde}
\newcommand{\wt}{\widetilde}

\newcommand{\wh}{\widehat}
\newcommand{\bN}{{\mathbf N}}

\newcommand{\wbT}{{\widetilde {\mathbf T}}}
\newcommand{\Tr}{\operatorname{Tr}}

\newcommand{\wG}{{\w G}}

\newcommand{\wN}{{\w N}}

\newcommand{\bC}{{\mathbf C}}
\newcommand{\la}{\ensuremath{\lambda}}

\newcommand{\bT}{{\mathbf T}}

\newcommand{\bB}{{\mathbf B}}
\newcommand{\bG}{{{\mathbf G}_{\mathrm{sc}}}}
\newcommand{\GG}{{{\mathbf G}}}

\newcommand{\bS}{{\mathbf S}}
\newcommand{\bH}{{\mathbf H}}

\newcommand{\wbG}{\widetilde{\mathbf G}}

\newcommand{\Irr}{\mathrm{Irr}}

\newcommand{\bZ}{{\mathbf Z}}
\newcommand{\Dia}{\operatorname{Diag}}

\newcommand{\SL}{\operatorname{SL}}

\newcommand{\ZZ}{\ensuremath{\mathbb{Z}}}
\newcommand{\CC}{\ensuremath{\mathbb{C}}}

\newcommand{\ov}{\overline }

\def\oo#1{\overline{\overline{#1}}}

\newcommand{\xx}{\mathbf x }\def\ww{{\widetilde{\bf w}_0} }
\newcommand{\n}{\mathbf n }\newcommand{\nn}{\mathbf n }
\newcommand{\h}{\mathbf h }\newcommand{\hh}{\mathbf h }

\renewcommand{\o}{\overline}

\newcommand{\Cent}{\ensuremath{{\rm{C}}}}
\newcommand{\NNN}{\ensuremath{{\mathrm{N}}}}

\newcommand{\Sym}{{\operatorname{S}}}

\def\restr#1|#2{\left.#1\right\rceil_{#2}}

\newcommand{\FF}{\ensuremath{\mathbb{F}}}
\newcommand{\tD}{\ensuremath{\mathsf{D}}}
\newcommand{\Cy}{\mathrm C}
\newcommand{\tC}{\mathsf C}
\newcommand{\tB}{\mathsf B}
\newcommand{\tE}{\mathsf E}
\newcommand{\tA}{\mathsf A}
\newcommand{\cE}{\mathcal E}

\newcommand{\cZ}{\mathcal Z}
\newcommand{\calO}{\mathcal O}

\newcommand{\calL}{\mathcal L}
\newcommand{\calT}{\mathcal T}
\newcommand{\calC}{\mathcal C}
\newcommand{\calG}{\ensuremath{\mathcal G}}
\newcommand{\al}{{\alpha}}

\newcommand{\si}{{\sigma}}

\newcommand{\cO}{{\mathcal O}}

\makeatletter
\def\Set#1{\Set@h#1@}
\def\Lset#1{\Lset@h#1@}
\def\Set@h#1|#2@{\left\{\left.#1\vphantom{#2}\hskip.1em\,\right|\,\relax #2\right\}}
\def\Lset@h#1@{\left\{#1\right\}}
\def\CALC#1{\CALC@h#1@}
\def\CALC@h#1|#2@{\calC^{#1}(#2)}
\def\CALCrad#1{\CALCrad@h#1@}
\def\CALCrad@h#1|#2@{\calC_\radic^{#1}(#2)}

\def\CALCNC#1{\CALCNC@h#1@}
\def\CALCNC@h#1|#2@{\calC_{\radic,nc}^{#1}(#2)}

\def\restr#1|#2{\left.#1\right\rceil_{#2}}
\def\spann<#1>{{\left{\langle}}#1{\right{\rangle}}}

\def\Spann<#1>{\Spann@h#1@}
\def\Spann@h#1|#2@{\left\langle\left.#1\vphantom{#2}\hskip.1em\right.\mid\relax #2 \right\rangle}
\def\Set#1{\Set@h#1@}
\def\Set@h#1|#2@{\left\{\left.#1\vphantom{#2}\hskip.1em\,\right.
	\mid\relax #2\right\}}
\def\set#1{\set@h#1@}
\def\set@h#1@{\left\{#1\right\}}

\makeatother

\newcommand{\Aut}{\mathrm{Aut}}

\newcommand{\Out}{\ensuremath{\mathrm{Out}}}
\def\spann<#1>{\left\langle#1\right\rangle}

\newcommand{\Z}{\operatorname Z}\newcommand{\G}{\operatorname G}
\newcommand{\GL}{\operatorname {GL}}

\newcommand{\calR}{\mathcal R}

\newcommand{\und}{{\text{ and }}}

\newcommand{\wrt}{with respect to{ }}

\usepackage{cleveref}
\usepackage{enumitem}

\newlist{asslist}{enumerate}{2} 
\setlist[asslist]{label=(\roman*), ref=\thethm(\roman*)}

\crefalias{asslistenumi}{Assumption} 

\newlist{thmlist}{enumerate}{1} 
\setlist[thmlist]{label=(\alph*), ref=\thethm(\alph*)}

\newcommand{\ra}{\rightarrow}
\newcommand{\lra}{\longrightarrow}

\newcommand{\Norm}{\operatorname{N}}
\begin{document}
	
\titlerunning{McKay condition for types $\tB$ and $\tE$}
\title{Descent equalities and the inductive McKay condition for types $\tB$ and $\tE$}
\date{ }
\author{Marc Cabanes and Britta Sp\"ath}
\maketitle

\address{M.C.: CNRS, Institut de Math\'ematiques Jussieu-Paris Rive Gauche, Place Aur\'elie Nemours, 75013 Paris, France.
	\email{marc.cabanes@imj-prg.fr}\\
B. S.: School of Mathematics and Natural Sciences
University of Wuppertal, Gau\ss str. 20, 42119 Wuppertal, Germany, \email{bspaeth@uni-wuppertal.de}}

\abstract{We establish the inductive McKay condition introduced by Isaacs-Malle-Navarro \cite{IMN} for finite simple groups of Lie types $\tB_l$ ($l\geq 2$), $\tE_6$, $^2\tE_6$ and $\tE_7$, thus leaving open only the types $\tD$ and $^2\tD$. We bring to the methods previously used by the authors for type $\tC$ \cite{CS17C} some descent arguments using Shintani's norm map. This provides for types different from $ \tA, \tD, {}^2\tD$ a uniform proof of the so-called global requirement of the criterion given by the second author in \cite[2.12]{S12}. The local requirements from that criterion are verified through a detailed study of the normalizers of relevant Levi subgroups and their characters.}

\tableofcontents
\renewcommand{\thesubsection}{\thesection.\Alph{subsection}} 

\section{Introduction}

The celebrated 
McKay conjecture on character degrees asserts that for any finite group $G$ and prime $p$   
\begin{equation}\label{McKay}
\text{  $ |\Irr_{p'} (G)|=|\Irr_{p'} (\NNN_G(P))|   $ }
\end{equation} where $\Irr_{p'}$ denotes the set of irreducible characters of degree prime to $p$ and $P$ is a Sylow $p$-subgroup of $G$.

The reduction theorem proved by Isaacs-Malle-Navarro \cite{IMN} reduces this conjecture to the checking of the so-called \emph{inductive McKay condition} for each finite simple group $S$ and prime $p$ (see \cite[\S 10]{IMN}). The inductive McKay condition has been checked for many simple groups leaving open the cases of simple groups of Lie types $\tB$, $\tD$, $^2\tD$, $\tE_6$, $^2\tE_6$, $\tE_7$ for odd $p$ and different from the characteristic of the group, see \cite{ManonLie}, \cite{S12}, \cite{CS13}, \cite{MS16}, \cite{CS17A}, \cite{CS17C}. One of the main results of the present paper is the following.

\begin{mainthm}\label{thmA}
 The finite simple groups of Lie types $\tB$, $\tE_6$, $^2\tE_6$, and $\tE_7$ satisfy the inductive McKay condition for all prime numbers.
\end{mainthm}

The proof of this theorem, as most previous verifications of the inductive McKay condition, uses the criterion \cite[2.12]{S12}, for which we give a streamlined version in Theorem~\ref{Glo+Loc} below.

 For $G$  the universal covering of a simple group $S$ of Lie type, the outer automorphism group $\Out(G)$ decomposes as a semi-direct product Out$(G)=\Dia (G)\rtimes E$ where $\Dia (G)$ is the group of so-called outer diagonal automorphisms of $G$ (denoted by $Outdiag(G)$ in \cite[\S 2.5]{GLS3}) and $E$ is a group of field and graph automorphisms (denoted by $\Phi_G\Gamma_G$ in \cite{GLS3}). As a key step we prove

\begin{mainthm}\label{thmB}
Assume $S$ is of type $\tB$, $\tE_6$,  $^2\tE_6$ or  $\tE_7$. For every $\chi\in\Irr(G)$ there exists a $\Dia(G)$-conjugate of $\chi$ whose stabilizer in $\Out(G)$ is of the form $D 'E'$ where $D '\leq \Dia(G)$ and $E'\leq E$.
\end{mainthm}

The above property of $\Irr(G)$ as an $\Out(G)$-set, called \Ainfty\ in Definition~\ref{A(d)} below, is known in types $\tA$, $^2\tA$, $\tC$ (see \cite{CS17A}, \cite{CS17C}) and for characters $\chi$ either semi-simple or of odd degrees in all types (see \cite{S12}, \cite{MS16}).
This question about characters of groups of rational points of simply-connected reductive groups is not easily solved by the present knowledge about characters of groups of Lie type which focuses mostly on reductive groups with connected center. It has however gained interest lately, see for instance \cite{Ma17}, \cite{Taylor}. Whenever Out$(G)=\Dia(G)  E$ is abelian (types $\tB$, $\tC$, $\tE_7$), the above statement about stabilizers is just that 
\begin{equation}\label{Irrdelga}
\text{  $   |\Irr(G)^{\spann<\delta f>}|=|\Irr(G)^{\spann<\delta \, ,\, f>}| $ }
\end{equation}
 for any $\delta\in\Dia(G)$, $f\in E$, using the exponent notation for fixed points (see \cite[2.2]{CS17C}). In our case we prove this equality by showing that each number in the above equation is in fact equal to a number defined similarly for the smaller group $G^f$, whence the term of descent equality. In the proof the norm map introduced by Shintani \cite{Shintani76} is used in an elementary and effective way.

Those general statements are probably of independent interest with possible applications to the Jordan decomposition of characters, as in \cite[\S 8]{CS17A}. Given how central McKay's conjecture has proved to be in representation theory of finite groups, it is not surprising that the property of stabilizers \Ainfty\ has applications to the checking of other counting conjectures through the classification of finite simple groups (see for instance \cite{CS15}). Via the triangularity property of decomposition matrices this can also apply to conjectures about modular characters as shown in \cite[7.4, 7.6]{CS13}.

All of the above concerns the first ``global" part of the criterion given in Theorem~\ref{Glo+Loc} and applies uniformly to types $\tB$, $\tC$ (thus reproving \cite[3.1]{CS17C}), $^2\tE_6$, $\tE_7$, and with some adaptations to $\tE_6$ (see Sect.\,3.D).

The rest of the paper is more specific to McKay's conjecture, being about local subgroups and related properties of their characters.
Concretely, the ``local" assumptions of the criterion from \cite{S12} are implied by the condition \AA d  from Definition~\ref{A(d)} for integers $d\geq 1$ such that the cyclotomic polynomial $\Phi_d$ divides the polynomial order of our simple group. In Sect.\,\ref{sec_locB} we verify that type $\tB_{}$ satisfies \AA d while Sect.\,\ref{sec_locE} deals with types $\tE$. 

The checking splits quite naturally into two parts. The one arising from numbers $d$ that are \emph{regular} for the Weyl group involved (divisors of $2l$ in type $\tB_l$) correspond to prime numbers in the statement (\ref{McKay}) of McKay's conjecture such that any Sylow subgroup $P\leq G$ has abelian centralizer $\Cy_G(P)$. This is the most technical part of our proof (see Sect.\,\ref{sec_locB}.A-C). Probably owing to the particular nature of the spin groups and its consequences on the structure of $\norm GP$, the checking for type $\tB$ is quite different from the one made for type $\tC$ in \cite{CS17C} and several cases have to be treated separately. The part of the proof dealing with so-called relative Weyl groups
(first introduced by Brou\'e-Malle-Michel \cite{BMM93}) however coincides with type $\tC$ in the case of type
$\tB$, while in the cases of types $\tE$ we use a proof provided to the authors by G. Malle.

The case of non-regular integers $d$ such that $\Phi_d$ divides the polynomial order of $G$ involves simple groups of smaller cardinality than $|G/\Z (G)|$ for which condition \Ainfty\ is used in a crucial way.

\medskip\noindent{\bf Acknowledgements.} {This material is partly based upon work supported by the NSF under Grant DMS-1440140 while the authors were in residence at the MSRI, Berkeley CA.}
We also thank our home institutions for having allowed that stay. The first author thanks BU Wuppertal for its hospitality during several visits. We thank Gunter Malle for the proof of Proposition~\ref{Gunter} and his remarks on our manuscript.

\section{Quasi-simple groups of Lie type and 
	{inductive McKay condition}}

Let us recall some basic notation about normal inclusions and characters. Let $Y\leq X$ be an inclusion of finite groups. For $\theta\in\Irr(Y)$ one denotes by $X_{Y,\theta}$ , or just $X_\theta$ whenever $Y\unlhd X$, the inertia subgroup of $\theta$ in $\NNN_X(Y)$. If $\chi\in\Irr(X)$ we write $\Irr(Y\mid\chi)$ for the set of irreducible constituents of $\restr \chi|Y$, the restriction of $\chi$ to $Y$. For $\theta\in\Irr(Y)$, $\theta^X$ denotes the character of $X$ obtained by induction of $\theta$.

If $Y\unlhd X$ is a normal inclusion of finite groups, we say that \emph{maximal extendibility} holds with respect to $Y\unlhd X$ if for any $\theta\in\Irr(Y)$ there exists $\w\theta\in\Irr(X_{\theta})$ such that $\theta =\restr{\w\theta}|{Y}$. In this context, an \emph{extension map} is then any map$$\Lambda\ \colon\ \Irr(Y)\ \ \longrightarrow\coprod_{X'\ ;\ Y\unlhd X'\leq X}\Irr(X')$$ such that for every $\theta\in\Irr(Y)$, $\Lambda(\theta)\in\Irr(X_{\theta})$ and $\theta =\restr{\Lambda(\theta)}|{Y}$.

\subsection{Finite groups of Lie type and automorphisms}
We denote by boldface letters $\GG$ or $\bH$ connected reductive groups over an algebraic closure $\FF$ of the field with $p$ elements, for $p$ a prime number. We denote $\GG_{\mathrm{ad}}=\GG/\Z(\GG)$.
We denote by $\bG\to [\GG ,\GG]$ the unique simply connected covering. A \emph{regular embedding} $\GG\leq\w\GG$ is any closed embedding of reductive groups with connected $\Z(\w\GG)$ and $\w\GG =\Z(\w\GG)\GG$ (see \cite[15.1]{CE04}).

When $\GG$ is defined over a finite subfield $\FF_{q^{}} $ ($q$ a power of $p$), this yields a Frobenius endomorphism $F\colon \GG\to\GG$ (see \cite[\S 3]{DM}). For any cyclotomic polynomial $\Phi_d$ ($d\geq 1$), one has a notion due to Brou\'e-Malle of \emph{$d$-torus} which is a class of $F$-stable tori $\bS\leq \GG$ such that, among other properties, $\bS^F $ has order a power of $\Phi_d(q)$ (see \cite[25.6]{MT}). 

The choice of a pair $\bT\leq \bB$ where $\bT$ is a maximal torus and $\bB$ is a Borel subgroup of $\GG$ allows one to associate a root system $\Phi\subseteq \mathrm{Hom} (\bT,\FF^\times)$ along with the subsets $\Phi^+\supseteq\Delta$ of positive and simple roots. To each $\al\in\Phi$, there corresponds a one-parameter unipotent subgroup $$\FF\ni t\mapsto\xx_\al (t)\in\GG $$ with $\xx_\al(\FF)\leq \bB$ when $\al\in\Phi^+$. One also defines $\n_\al(t):=\xx_{\al} (t)\xx_{-\al} (-t^{-1})\xx_{\al} (t)\in\NNN_\GG(\bT)$ and $\h_\al(t)=\n_{\al} (t)\n_{\al} (1)^{-1}\in \bT$ for $t\in\FF^\times$. The $\xx_\al(\FF)$'s for $\al\in \pm\Delta$ generate $[\GG ,\GG]$. The \emph{type} of $\GG$ is by definition the one of $\Phi$ as a root system.

Assume now that $\bT\leq\bB\leq \GG$ is as above with $\GG=\bG$ of irreducible type. One defines a Frobenius endomorphism $F_0\colon \bG\to\bG$ by $F_0(\xx_\al(t))=\xx_\al (t^p)$ for all $t\in\FF$, $\al\in\Phi$. It is easy to construct a regular embedding $\bG\leq \w\GG$ with $\Z(\w\GG)$ a torus of rank 1 or 2 according to $\Z(\bG)$ being cyclic or not, and to extend $F_0$ to a map $F_0\colon \w\GG\to \w\GG$ defining $\w\GG$ over $\FF_p$. Indeed since $F_0(s)=s^p$ for any $s\in \bT$, all subtori of $\bT$ are $F_0$-stable, so one can define $\w\GG:=\bG\times_{\Z(\bG)}\bZ$ (central product) where $\Z(\bG)\leq\bZ\leq\bT$ with $\bZ$ a subtorus of minimal rank subject to $\Z(\bG)\leq\bZ$.

 Denote by $\Aut(\Delta)$ the group of graph automorphisms, i.e. permutations of $\Delta$ preserving lengths and pairwise angles of simple roots. Then any $\gamma\in \Aut(\Delta)$ induces an algebraic automorphism of $\bG$ defined by $\gamma(\xx_\al(t))=\xx_{\gamma(\al)}(t)$ for any $t\in\FF$, $\al\in \pm\Delta$. This makes a group $\Gamma_{\bT\leq\bB}\cong \Aut(\Delta)$ of automorphisms of $\bG$. This group can be made to act on $\w\GG$ by algebraic automorphisms commuting with $F_0$ (see for instance \cite[3.1]{S12}). Indeed for classical types $\bG=\SL_n(\FF), \dots$ one has well-known regular embeddings $\w\GG=\GL_n(\FF),\dots$ where this is clear. For exceptional types, only type $\tE_6$ needs some care and one chooses in the construction above $\bZ:=\{\hh_{\al_1}(t)\hh_{\al_2}(t^{-1})\hh_{\al_5}(t)\hh_{\al_6}(t^{-1})\mid t\in\FF^\times\}$ (notations of \cite[Table 1.12.6]{GLS3}) which is $\Gamma_{\bT\leq\bB}$-stable.

 The finite quasi-simple groups we consider are of the form $\GG_{\mathrm{sc}}^F$ where $F=F_0^m\gamma_0$ with $F_0$ as above, $m\geq 1$ and $\gamma_0\in\Gamma_{\bT\leq\bB}$. One then denotes $q:=p^m$. We keep $\bG\leq \w\GG$ a regular embedding as above.

\begin{defi}\label{DefE}
	Let $E$ be the subgroup of $\Aut (\GG_{\mathrm{sc}}^F)$ generated by $\restr{F_0}|{\GG_{\mathrm{sc}}^F}$ and the $\restr{\gamma}|{\GG_{\mathrm{sc}}^F}$ for $\gamma\in\Gamma_{\bT\leq\bB}$. 
\end{defi}

 As said before, the action of $F$ extends to $\w\GG$ and the action of $E$ also extends to $\w\GG^F$.

 \subsection{Inductive McKay conditions, global and local.}

As in most previous verifications for simple groups with a non-cyclic outer automorphism group (see \cite{S12}, \cite{MS16}, \cite{CS17A}, \cite{CS17C}) the inductive McKay condition will be checked here through \cite[2.12]{S12}. Theorem~\ref{Glo+Loc} below is a streamlined version of that criterion taking into account essentially \cite[\S 6]{CS17A} and the case of the defining prime \cite[1.1]{S12}. Note in particular that the reference to a prime number is replaced by the reference to a cyclotomic polynomial. We first define the following properties \Ainfty, \AA d, \BB d for $d\geq 1$ that represent the essential requirements of \cite[2.12]{S12}.

\begin{defi}\label{A(d)}  Keep $\bG$, $\wbG$, $F$, $E$ as above. We define the following properties  
	\begin{enumerate}
		\item[\Ainfty] \label{Ainfty} Every $\chi\in \Irr(\GG_{\mathrm{sc}}^F)$ has a $\wbG^F$-conjugate $\chi_0$ such that $(\wbG^FE)_{\chi_0}=\wbG^F_{\chi_0}E_{\chi_0}$ and $\chi_0$ extends to $\GG_{\mathrm{sc}}^FE_{\chi_0}$.
	\end{enumerate}
\noindent
Let $d\geq 1$, let $\bS$ be a Sylow (i.e. maximal) $d$-torus of $\bG$ and denote $N=\NNN_{\bG}(\bS)^F$, $\w N=\NNN_{\wbG}(\bS)^F$, $\w C=\Cy_{\wbG}(\bS)^F$. 
\begin{enumerate}[ref=\thethm.$A(d)$]
	\item[A($d$)]  \label{Ad} For every $\chi\in\Irr(N)$ there is an $\w N$-conjugate $\chi_0$ which extends to $(\GG_{\mathrm{sc}}^FE)_{\bS ,\chi_0}$ and such that $O=(\wbG^F\cap O)(E\cap O)$ for $O=(\wbG^FE)_{\bS ,\chi_0}\GG_{\mathrm{sc}}^F$.
	\item[B($d$)]  \label{Bd} Maximal extendibility holds with respect to the inclusions $N\unlhd\w N$ and $\w C\unlhd \w N$. There is an  extension map $\w\Lambda$ with respect to $\w C\unlhd \w N$ which is $(\wbG^FE)_{\bS }$-equivariant and satisfies $$\w\Lambda (\restr\epsilon|{\w C}  \theta) =\restr\epsilon|{\w N_\theta} \w\Lambda (\theta)$$ for any $\theta\in \Irr (\w C)$ and $\epsilon\in\Irr (\wbG^F\mid 1_{\GG_{\mathrm{sc}}^F})$.
\end{enumerate}
\end{defi}

\begin{rem}\label{GloInf} Due to the $\GG_{\mathrm{sc}}^F$-conjugacy of Sylow $d$-tori of $\bG$ (see \cite[25.11]{MT}), the conditions \AA d and \BB d are independent of the choice of $\bS$.
	
	Note that for $d\geq 3h$ (where $h$ is the Coxeter number of the root system of $\bG$) one has $\bS =\{ 1\}$, $N=\GG_{\mathrm{sc}}^F$, $\w C=\w N=\wbG^F$. Since maximal extendibility holds with respect to the inclusion $\GG_{\mathrm{sc}}^F\unlhd\wbG^F$ by a theorem of Lusztig (see \cite[10]{L88}), it is then clear that, for large $d$, \BB d is always satisfied while \AA d is equivalent to the above \Ainfty, hence our notation.\end{rem}
	
		The condition \Ainfty\ is the so-called \emph{global condition} introduced in \cite[2.12(v)]{S12}. It is also an important step in the checking of \AA d for any $d$. We informally call here
	\emph{local condition} the conjunction of \AA d and \BB d for $d$ such that $(\bG ,F)$ has a non-trivial $d$-torus. Among the latter	\emph{local cases}, we will also call \emph{regular} the cases where Sylow $d$-tori $\bS$ satisfy that $\Cy_\bG(\bS)$ is a (maximal) torus.
	
	Note that the local conditions  \AA d and \BB d  have been checked for $d=1,2$ and all types in \cite[\S~3]{MS16}, while \Ainfty\ is established there for characters of odd degree.

\begin{thm}\label{Glo+Loc}
Assume that \AA d and \BB d from Definition~\ref{A(d)} are satisfied by $\GG_{\mathrm{sc}}^F$ for any $d\geq 1$. Then if $\GG_{\mathrm{sc}}^F/\Z (\GG_{\mathrm{sc}}^F)$ is a simple group, it satisfies the inductive McKay condition of \cite{IMN} for any prime.

 More precisely, if one assumes only \Ainfty, \AA d and \BB d for a given $d\geq 1$ and $\GG_{\mathrm{sc}}^F/\Z (\GG_{\mathrm{sc}}^F)$ is a simple group, then the latter satisfies the inductive McKay condition for any prime $\ell\nmid 2q$ such that the multiplicative order of $q$ mod $\ell$ is $d$.
\end{thm}

\begin{proof}
	 Note that the cases where $\GG_{\mathrm{sc}}^F/\Z (\GG_{\mathrm{sc}}^F)$ is a simple group but $\Z (\GG_{\mathrm{sc}}^F)$ is smaller than its Schur multiplier reduce to the same verification as in the other cases, thus dealing only with $\GG_{\mathrm{sc}}^F$ , thanks to \cite[4.1]{ManonLie} and \cite[7.4]{CS17A}. So we will check the inductive McKay condition of \cite{IMN} with $\GG_{\mathrm{sc}}^F$ acting as the universal covering of  $\GG_{\mathrm{sc}}^F/\Z (\GG_{\mathrm{sc}}^F)$. 
	 
	 The condition is empty for primes not dividing $|\GG_{\mathrm{sc}}^F|$, so we assume $\ell$ is a prime divisor of $|\GG_{\mathrm{sc}}^F|$. Thanks to \cite[1.1]{S12} and \cite{MS16} we can assume that $\ell\not= p, 2$.  Let us take $d$ the order of $q$ mod $\ell$ and $\bS$, $N$, $\w N$, etc. as in our statement. First $\bS\neq 1$. Indeed, since $\ell$ divides $|\GG_{\mathrm{sc}}^F|$, we know that some cyclotomic polynomial $\Phi_{d\ell^a}$ for $a\geq 0$ divides the order of $|\GG_{\mathrm{sc}}^F|$ seen as a polynomial in $q$. But then $\Phi_d$ itself divides this polynomial order as can be seen on tables for $|\GG_{\mathrm{sc}}^F|$ (see for instance \cite[Table 24.1]{MT}). 
	 
	 Assuming now \AA d, \BB d and \Ainfty\ from Definition~\ref{A(d)}, we prove that $\GG_{\mathrm{sc}}^F/\Z (\GG_{\mathrm{sc}}^F)$ satisfies the inductive McKay condition for $\ell$ by reviewing the requirements given in \cite[2.12]{S12}. The group theoretical requirements that $\wbG^F/\GG_{\mathrm{sc}}^F$ is abelian, $\Cy_{\wbG^FE}(\GG_{\mathrm{sc}}^F)=\Z (\w\GG^F)$ and $\wbG^FE$ induces the whole $\Aut(\GG_{\mathrm{sc}}^F)$ are known as said in \cite[3.4(a)]{S12} (see also \cite[2.5.4]{GLS3}).
	
	One has $N\not=\GG_{\mathrm{sc}}^F$. Otherwise $\bS$ is normalized by $\GG_{\mathrm{sc}}^F$, hence $\Cy_\bG(\bS)^F:=C\unlhd \GG_{\mathrm{sc}}^F$ which implies that $C\in\{ \GG_{\mathrm{sc}}^F\, ,\, \Z (\GG_{\mathrm{sc}}^F)\}$ since $\GG_{\mathrm{sc}}^F$ is quasi-simple. If $C=\GG_{\mathrm{sc}}^F$ then $\bS\leq \Z (\GG_{\mathrm{sc}})$ by \cite[6.1]{B06}, which in turn contradicts $\bS\neq 1$ since $\Z (\GG_{\mathrm{sc}})$ is finite. If $C=\Z (\GG_{\mathrm{sc}}^F)$, arguing on the Sylow $p$-subgroup of the finite reductive group $C$, it is easy to see that $\Cy_\bG(\bS) $ would be a (maximal $F$-stable) torus $\bT$ with $\ell\mid  |\Z(\GG_{\mathrm{sc}}^F)|= |\bT^F|$. Discussing along the lines of \cite[Proof of 3.6.7]{Ca85}, one finds only types $\tA_1$ and $^2\tA_2$ for $(q,\ell)=(3,2)$ and $(2,3)$ respectively, but then $\GG_{\mathrm{sc}}^F$ is solvable which is again a contradiction. 
	
	By \cite[5.14, 5.19]{MaH0}, we know that there is a Sylow $\ell$-subgroup $Q$ of ${\GG_{\mathrm{sc}}^F}$ such that $\NNN_{\GG_{\mathrm{sc}}^F}(Q)\leq N$ except if ${\GG_{\mathrm{sc}}^F}$ is among a list of exceptions for which the inductive McKay condition has been checked in \cite[\S 3.3]{MaExt}. We have seen that any element of $\wbG^FE$ acts on $\GG_{\mathrm{sc}}^F$ as the restriction of a bijective endomorphism $\sigma$ of $\bG$ defined up to powers of $F$. This explains why the group $(\wbG^FE)_{\bS ,\psi_0}$ is well-defined. For such a $\sigma$ stabilizing $N$, one has by a Frattini argument $\sigma (Q)={Q}^g$ for some $g\in N$. Applying \cite[2.5]{CS13} to $g\sigma$, this implies that $\sigma (\bS)=\bS$. So our groups $O$ and $\w N$ are indeed the ones denoted the same in \cite[2.12]{S12}. 

Maximal extendibility holds with respect to $\GG_{\mathrm{sc}}^F\unlhd \w\GG^F$ as recalled in the remark above, while it is here assumed for the inclusion $N\unlhd \w N$ through condition \BB d. This completes the proof of points 2.12(i)-(iii) of \cite{S12}. Points 2.12(v) and 2.12(vi) are clearly contained in our assumptions \Ainfty\ and \AA d. 

The remaining requirement 2.12(iv) of \cite{S12} is a consequence of \cite[\S 6]{CS17A} assuming \BB d as explained in \cite[Proof of 7.3]{MS16}.
\end{proof}

\begin{rem} Let us say how assumption \Ainfty\ in fact determines $\Irr(\GG_{\mathrm{sc}}^F)$ as an $\Out(\GG_{\mathrm{sc}}^F)$-set.
	 We keep the notations $\GG_{\mathrm{sc}}$, $\w\GG$, $E$, and we abbreviate $\cZ =\w\GG^F/\GG_{\mathrm{sc}}^F$. Then the semi-direct product $\cZ\rtimes E$ acts on $\Irr(\GG_{\mathrm{sc}}^F)$. Assume that $\GG_{\mathrm{sc}}$ satisfies the stabilizer part of the \Ainfty\ condition, that is\hfill \break\noindent (a) {\it any $\cZ$-orbit in $\Irr(\GG_{\mathrm{sc}}^F)$ has an element $\chi$ such that $(\cZ E)_\chi =\cZ_\chi E_\chi$.}

Then as a $\cZ E$-set, $\Irr(\GG_{\mathrm{sc}}^F)
$ is a disjoint union of transitive $\cZ E$-sets of the type $\cZ E/\cZ_i E_i$ for $\cZ_i\leq\cZ$, $ E_i\leq E\cap\NNN_{\cZ E}(\cZ_i)$. 

One can then determine the stabilizers $(\cZ E)_\chi =\cZ_\chi E_\chi$ from the action of $E$ and $\cZ^*:=\Irr(\wbG^F\mid 1_{\GG_{\mathrm{sc}}^F})$ on $\Irr (\wbG^F)$, which is known by \cite[3.1]{CS13}. The latter gives the term $E_\chi$ since $E_\chi =E_{\cZ.\chi} =E_{\cZ^*.\w\chi}$ where $\w\chi\in\Irr(\wbG^F\mid\chi)$. On the other hand, \cite[9.(a)]{L88} tells us that $\cZ_\chi =(\cZ^*_{\w\chi})^\perp$ where orthogonality refers here to the perfect pairing between $\cZ$ and $\cZ^*\cong \mathrm{Hom} (\cZ,\Bbb C^\times)$.
\end{rem}

\section{The global condition \Ainfty\ for types $\tB$, $\tC$, and $\tE$} \label{sec4A}
The aim of this part is to check the property \Ainfty\ from Definition~\ref{A(d)} for the types listed. This is Theorem~\ref{gloBCE} below. It clearly implies Theorem~\ref{thmB} announced in the introduction. Many statements have however a broader range, see Theorems~\ref{pr_Cla_Sh}, \ref{thm_Char_count} and \ref{GtildeExt}.
\subsection{A descent theorem for invariant classes. }

\medskip

We show in this section a descent equality for the left-hand side of Equation (\ref{Irrdelga}). This is Theorem~\ref{pr_Cla_Sh}.
We have to recall the basic notions and notations for twisted conjugacy and norm (or Shintani) maps.

 Let $H$ be a group. For $\si\in\Aut(H)$ and $h\in H$ one denotes by $h\si\in\Aut(H)$ the composite of $\si$ with conjugation by $h$, namely $(h\si)(x)=h\cdot \si(x)\cdot h^{-1}$ for $x\in H$. One denotes by $\sim_H$ the relation on $H$ of conjugacy.

More generally, one denotes by $\sim_\si$ the equivalence relation on $H$ such that for $h',h''\in H$, $h'\sim_\si h''$ means that there exists $h\in H$ with $h''=h^{-1}h'\si(h)$. Seeing $H$ as a normal subgroup of a semi-direct product $H\rtimes \spann<c>$ where $c$ acts on $H$ by $\si$, one has a bijection $h\mapsto hc$ sending $H$ to the coset $Hc$ and $\sim_\si$ is transformed into $\sim_H$ or equivalently $\sim_{H\rtimes \spann<c>}$. 

When $H$ is finite, maps $H\to \Bbb C$ that are constant on $\sim_\sigma$-classes have a well-known basis defined in terms of $\Irr (H)^{\spann<\sigma>}$. 
	The following can be found in \cite[8.14]{Isa} or \cite[1.1-3]{Shintani76}.

\begin{prop}\label{onxY} Let $Y\unlhd X$ be finite groups and $x\in X$ with $\spann<Y,x>=X$.  Let $\Lambda$ be an extension map \wrt $Y\unlhd X$ (see beginning of Sect.\,2), which exists by \cite[11.22]{Isa} since $X/Y$ is cyclic. Then $(\restr{\Lambda(\chi)}|{xY})_{\chi\in\Irr(Y)^X}$ forms 
	a basis of the space of $\sim_X$-class functions $xY\to\Bbb C$, in particular $|xY/_{\sim_X}|=|xY/_{\sim_Y}|=|\Irr(Y)^X|$. 
\end{prop}

\medskip

The following norm map was first introduced by T. Shintani in \cite{Shintani76}.
 Let $\bH$ be a connected algebraic group and $F_1,F_2\colon \bH\to\bH$ be Frobenius morphisms (see \cite[\S~3]{DM}) such that $F_1F_2=F_2F_1$.
	Note that by Lang's theorem (see \cite[21.7]{MT}) every element of $\bH^{F_1}$ can be written as $y^{-1} F_2(y)$ for some $y\in\bH$.
	
\begin{theo}[{\cite[\S 1.1]{Shoji98}}]\label{not_Norm} 
The map $$\Norm_{F_1/F_2}^{(\bH)}: \bH^{F_1} /{\sim_{F_2}}\longrightarrow \bH^{F_2} /{\sim_{F_1}} \text{ given by }y^{-1} F_2(y) \mapsto F_1(y) y^{-1}$$ is well-defined and bijective. 
\end{theo}	
%
%
%
%

%
A classical consequence of Proposition \ref{onxY} and Theorem~\ref{not_Norm} is that $$|\Irr (\bH^{F_1^m})^{F_1}|=|\Irr (\bH^{F_1})|$$ for any $m\geq 1$, see \cite[2.7]{Shintani76}. %
We need a slight variant adding a diagonal automorphism to the picture.

\begin{thm} \label{pr_Cla_Sh} Let $F_1\colon \wbG\to\wbG$ a Frobenius endomorphism of a connected algebraic group, let $F:=F_1^m$ for some integer $m\geq 1$. Let $\GG\leq \wbG$ be a closed $F_1$-stable connected subgroup such that $[\wbG ,\wbG]\leq\GG$. Let $t\in \wbG^F$. Then 
 
\begin{equation}\label{Fixpoints1}
 |\Irr (\GG^F)^{\spann<t F_1>}|=|\Irr (\GG^{F_1})^{\spann<t_1>}|
\end{equation} for some $t_1\in \wbG^{F_1}$ corresponding to $t$ by $\Norm_{F/F_1}^{(\wbG)}$.

Moreover, if $\wbG^F=\spann<\GG^F,t>$, then $\wbG^{F_1}=\spann<\GG^{F_1},t_1>$. 
\end{thm}

\begin{proof}
Let us first consider the pair of commuting Frobenius endomorphisms $(F,tF_1)$ on $\GG$. Theorem~\ref{not_Norm} implies that $$|\GG^F/\sim_{tF_1}|=|\GG^{tF_1}/\sim_{F}|.$$ Working now in the semi-direct product $\w\GG\rtimes\spann<F_1>$, the considerations given at the start of the section allow to rewrite the above as $$
|\GG^F{tF_1}/\sim_{\GG^F}|=|\GG^{tF_1}F/\sim_{\GG^{tF_1}}|.$$

Lang's theorem ensures that there is $g\in \w\GG$ such that $t = g^{-1}F_1(g)$. Denote $t_1=F(g)g^{-1}\in \w\GG^{F_1}$, so that the $\sim_{F}$-class of $t_1$ is associated to $t$ by $\Norm_{F/F_1}^{(\wbG)}$. Now the interior automorphism $x\mapsto gxg^{-1}$ of $\w\GG\rtimes\spann<F_1>$ sends $\GG^{tF_1}$ to $\GG^{F_1}$ and $F$ to $t_1^{-1}F$. Applying that to the right hand side of the equation above, we deduce $$|\GG^F{tF_1}/\sim_{\GG^F}|=|\GG^{F_1}t_1^{-1}F/\sim_{\GG^{F_1}}|.$$ Using the last equation in Proposition~\ref{onxY} allows to write the above as \begin{equation}\label{Fixpoints2}|\Irr(\GG^F)^{\spann<tF_1>}|=|\Irr(\GG^{F_1})^{\spann<t_1^{-1}F>}|.\end{equation}  We know that $F=F_1^m$ acts trivially on $\GG^{F_1}$, so the right hand side above can be rewritten as $|\Irr(\GG^{F_1})^{\spann<t_1^{-1}>}|$ or equivalently $|\Irr(\GG^{F_1})^{\spann<t_1>}|$. This finishes the proof of (\ref{Fixpoints1}).

In order to check the last statement, denote $\bC =\wbG/\GG$. This is an abelian connected algebraic group where we can define $\Norm_{F/F_1}^{(\bC)}$. Since $t_1\in {\wbG^{F_1}}$ is such that its $\sim_{F}$-class is the image by $\Norm_{F/F_1}^{(\wbG)}$ of the $\sim_{F_1}$-class of $t$, the cosets mod $\GG$ satisfy  $t_1\GG =\Norm_{F/F_1}^{(\bC)}(t\GG)$ as is easily seen from the definition of Shintani norm maps.  In the abelian group $\bC$, one has $\bC^F/\sim_{F_1}=\bC^F/[\bC^F,F_1]$ while $\sim_F$ is trivial on $\bC^{F_1}$. Then $\Norm_{F/F_1}^{(\bC)}=\Norm_{F_1^m/F_1}^{(\bC)}$ is just the map $\bC^F/[\bC^F,F_1]\to \bC^{F_1}$ defined by $x\mapsto xF_1(x)\dots F_1^{m-1}(x)$ on $\bC^F$ (see also \cite[\S 1.2]{Shoji98}). This is a group morphism and it is a bijection by Theorem~\ref{not_Norm}, so the assumption that $t\GG\in \wbG/\GG$ generates $(\wbG/\GG)^F=\wbG^F/\GG^F$ implies that $t_1\GG$ generates $(\wbG/\GG)^{F_1}$.
\end{proof}

\subsection{Genericity of unipotent characters and automorphisms. }

In \cite{BMM93}, Brou\'e-Malle-Michel developed a language underlining the generic nature of unipotent characters. In the following we point out how this feature can take into account the action of automorphisms. This is a variation of the results given in \cite[\S~1B]{BMM93} and extends their proof. 

Let us recall that to each reductive group $\bH$ is associated its isomorphism type $\Gamma$ of (finite crystallographic) root system represented by a Dynkin diagram. Any Frobenius endomorphism $F\colon \bH\to \bH$ defines a unique $F_\Gamma\in \Aut(\Gamma)$, a consequence of the $\sim_\bH$-uniqueness of pairs $\bT\leq \bB$ (see \cite[\S~1A]{BMM93}, \cite[22.2]{MT}). We also deal with bijective endomorphisms $$f\colon \bH\to\bH $$ that commute with $F$. They are of the same form or algebraic automorphisms, both cases giving rise to an $$f_\Gamma \in \Aut(\Gamma)$$ commuting with $F_\Gamma$ (see \cite[2.5.5(c)]{GLS3}).

\begin{prop}\label{UGamma}
	Let $\Gamma$ be a (finite crystallographic) root system.
	Then there exist  
	a finite $\Aut(\Gamma)$-set $U(\Gamma)$  and, 
	for each reductive group $\bH$ with root system $\Gamma$ and Frobenius map $F\colon \bH\to\bH$,
	a bijection 
	$$ \tau_{\bH ,F}  \colon \cE (\bH^F,1)\xrightarrow{\ \ \sim\ \ } U(\Gamma)^{F_{\Gamma}} \eqno(1)$$ satisfying the following.
	
For any bijective algebraic morphism $f\colon \bH\to\bH$ with $f\circ F=F\circ f$ one has
 $$f_{\Gamma}(\tau_{\bH ,F}(\lambda\circ \restr{f}|{\bH^F}))=\tau_{\bH ,F}(\lambda)\text{ for any  }\ \lambda\in \cE (\bH^F,1).\eqno(2) $$
\end{prop}

\begin{proof} Note first that since one considers only unipotent characters, it is sufficient to define $\tau_{\bH ,F}$ in cases where $\bH =\bH_{\mathrm{ad}}$ (see \cite[13.20]{DM}). Note also that Equation (2) is just an equivariance condition, so that it is preserved by composition of endomorphisms satisfying it.
	
 One defines $U(\Gamma)$ as follows. For $\Gamma$ an isomorphism type of \emph{irreducible} root system, let $W(\Gamma)$ be the associated Weyl group, and let $\Bbb G_0$ be the adjoint generic group associated with $\Gamma$ and the trivial coset $W(\Gamma)$ inside the automorphism group of the coroot lattice in the sense of \cite[1.A]{BMM93}. Let $U(\Gamma)=\mathrm{Uch}({\Bbb G_0})$ in the sense of \cite[1.26]{BMM93} and proof. Recall that $\mathrm{Uch}({\Bbb G_0})$ is defined in terms of pairs associated to Lusztig families and relevant subgroups of $W(\Gamma)$. By \cite[1.27]{BMM93}, this is an $\Aut(\Gamma)$-set and there is a bijection $
\tau_{\bH ,F}  \colon \cE (\bH^F,1)\xrightarrow{\ \ \sim\ \ } U(\Gamma)$ satisfying (2) for all \emph{split} (${F_\Gamma}=\Id_{\Gamma}$) groups $(\bH ,F)$ by \cite[1.28]{BMM93}. The bijection is denoted by $\gamma^{\bH^F}\leftarrow\!\shortmid \gamma$ in \cite[1.26]{BMM93}. 

Let us look at cases where $\Gamma$ is still irreducible but ${F_\Gamma}$ is not necessarily $\Id_\Gamma$.
 Here we have to diverge slightly from \cite[1.B]{BMM93} since for the adjoint generic group $\Bbb G$ corresponding to $(\Gamma ,{F_\Gamma})$, the set Uch$(\Bbb G)$ does not identify readily with Uch$(\Bbb G_0)^{F_\Gamma}$. Given the equivariance already verified in the split case, we must ensure that 
 \begin{equation}\label{UniFix}
 \text{  $|\cE (\bH^F,1)|=|\cE (\bH^{F_0},1)^F|    $ }
 \end{equation} for $F_0\colon \bH\to\bH$ a split Frobenius morphism commuting with $F$. This has been proved by Shoji \cite[2.2]{Shoji85}, \cite[3.2]{Shoji87}, see also \cite[p. 24]{BMM93}. (Remark~\ref{CombiUni} below gives a combinatorial proof of this fact.) 
 
 Thanks to what has been seen about the split case above, this means $|\cE (\bH^F,1)|=|U(\Gamma)^{F_\Gamma}|$ and therefore one may define $\tau_{\bH ,F}  \colon \cE (\bH^F,1)\xrightarrow{\ \ \sim\ \ } U(\Gamma)^{F_\Gamma}$ arbitrarily when ${F_\Gamma}$ is non-trivial. To check (2) in that case, note first that automorphisms of simple groups of twisted Lie types act trivially on unipotent characters (see \cite[p. 59]{L88}, \cite[3.7]{MaH0}). So (2) simply asserts that ${f_\Gamma}$ acts trivially on $U(\Gamma)^{F_\Gamma}$ whenever ${f_\Gamma}\in\Aut(\Gamma)$ commutes with the non-trivial ${F_\Gamma}$. This is actually the case since then ${f_\Gamma}\in \spann<{F_\Gamma}>$ because $\Aut(\Gamma)$ is either dihedral of order 6 (type $\tD_4$) or of order $\leq 2$ (other types).

When $\Gamma$ is not connected, let $U(\Gamma):=\Pi_iU(\Gamma_i)$ where the $\Gamma_i$'s are the connected components of $\Gamma$. We have $\Aut(\Gamma)=\Pi_\omega \Aut(\Gamma_\omega)\wr\Sym_{|\omega|}$ where $\omega$ ranges over the isomorphism classes in $\{\Gamma_i  \}_i$ and $\Gamma_\omega$ is the corresponding type. Moreover $\bH=\bH_{\mathrm{ad}}=\Pi_i \bH_i$ a direct product along the components of $\Gamma$ and any bijective endomorphism of $\bH$ permutes the $\bH_i$'s (see for instance \cite[p. 700]{CS13}). So we can assume that there is a single $\omega$, i.e. the $\Gamma_i$'s are all isomorphic. If $F\colon \bH \to \bH$ is a Frobenius map for a reductive group of type $\Gamma$, then $\bH_{\mathrm{ad}}^F = \Pi_o \bH_{i_o}^{F^{|o|}}$ where $o$ ranges over the orbits of $F$ permuting the $i$'s and $i_o\in o$. The definition of $\tau_{\bH ,F}$ satisfying (1) then derives from the irreducible case above. Concerning (2), let us note that for any $\si\in \Aut(\Gamma)$ commuting with ${F_\Gamma}$ one defines easily an algebraic automorphism $\si_\bH\colon \bH_{\mathrm{ad}}\to \bH_{\mathrm{ad}}$ inducing $\sigma$ on $\Gamma$ and satisfying (2). Now, if $f\colon \bH_{\mathrm{ad}}\to \bH_{\mathrm{ad}}$ is as in (2), one may compose it with $\si _{\bH_{\mathrm{}}}^{-1}$ for $\si=f_\Gamma$ the element of $\Aut(\Gamma)$ induced by $f$, so that the checking of (2) now reduces to an $f$ that preserves each $\bH_i$. The irreducible case already checked then gives our claim.
\end{proof}

\medskip

\begin{rem}\label{CombiUni} It is possible to give a purely combinatorial proof of equation (\ref{UniFix}) using only the known action of automorphisms of finite groups of Lie type on unipotent characters (see \cite[3.7]{MaH0}). We use the combinatorics described in \cite[\S~13.8]{Ca85}. 
	
	For the automorphism of order 2 of type $\tD$, one has to define a rank-preserving bijection between non-degenerate symbols of defect in $4\ZZ$ on one hand, and symbols of defect in $2+4\ZZ$, the latter a parametrizing set for unipotent characters of groups $^2\tD_l(q)$, on the other hand. Such a bijection is for instance $\Lambda\mapsto \Lambda ^*$ where the symbol $\Lambda ^*$ is obtained from  
	$\Lambda= \begin{pmatrix} T'\\T''\end{pmatrix}$ by moving the biggest element of $(T'\cup T'')\setminus (T'\cap T'')$ from the set where it belongs to the other one. 
	
	For the automorphism of order 3 of $\tD_4$, there are 8 symbols of rank 4 and defect in $4\ZZ$ that are fixed under that automorphism, and this is indeed the number of unipotent characters for type $^3\tD_4$.
	
	In the case of types $^2\tA$ and $^2\tE_6$, it is well-known that unipotent characters are in bijection with the ones of the corresponding non-twisted group. 
\end{rem}

\begin{cor}\label{F_1F_2} For any pair of commuting Frobenius endomorphisms $F_1,F_2\colon \bC\to \bC$ of a reductive group $\bC$ over $\FF$, there is a bijection $$\cE (\bC^{F_1}, 1)^{F_2}\to \cE (\bC^{F_2}, 1)^{F_1}$$ which is equivariant for algebraic automorphisms $\bC\to\bC$ commuting with both $F_1$ and $F_2$.
\end{cor}  

\begin{proof} Let us denote by $\Gamma$ the isomorphism type of the root system of $\bC$ and by $\phi_i$ the automorphism of $\Gamma$ induced by $F_i$ for $i=1,2$. 
Proposition~\ref{UGamma} gives an $\Aut(\Gamma)$-set $U(\Gamma)$ and maps $\tau_{\bC ,F_i}\colon \cE (\bC^{F_i},1)\xrightarrow{\ \sim\ }U(\Gamma)^{\phi_i}$ satisfying the equivariance property (2) of Proposition~\ref{UGamma}. The restriction of $\tau_{\bC,F_2}$ to $\cE(\bC^{F_2},1)^{F_1}$ gives a bijection $$\tau_1\colon\cE(\bC^{F_2},1)^{F_1}\xrightarrow{\ \sim\ }U(\Gamma)^{\spann<\phi_1\, ,\,\phi_2>}$$
thanks to condition \ref{UGamma}(2) applied to $f=F_1$. The bijection $\tau_1$ also satisfies the same commutation property \ref{UGamma}(2) for automorphisms $f$ of $\bC$ commuting with both $F_1$ and $F_2$. Interchanging the roles of $F_1$ and $F_2$, we also obtain $\tau_2\colon\cE(\bC^{F_1},1)^{F_2}\xrightarrow{\ \sim\ }U(\Gamma)^{\spann<\phi_1\, ,\,\phi_2>}$. Now the bijection $\tau_1^{-1}\tau_2$ satisfies our claim. 
\end{proof}

\subsection{A descent theorem for invariant characters.} 

The following theorem deals with the right hand side of Equation (\ref{Irrdelga}). This is again a descent equality relating the number for some group $\GG^F$ to a similar number for a smaller group $\GG^{F_1}$ where $F$ is a power of $F_1$. Shintani maps again play a role through the use of Shoji's theorems in the proof of Proposition~\ref{UGamma}.

\begin{thm} \label{thm_Char_count}
Assume $\GG=\bG$ is simple of type $\tB$, $\tC$ or $\tE$. Let $\GG\leq \w\GG$ be a regular embedding. Let $F_1\colon \wbG\to\wbG$ be a Frobenius morphism defining it over $\FF_{q_1}$. Let $m\geq 1$, $q:=q_1^m$ and let $F=F_1^m$ defining $\wbG$ over $\FF_q$.
Assume $F_1$ and $ F$ coincide on $\Z(\bG)$. 

Then 
$$ | \Irr (\GG^F)^{\spann <\wbG^F ,\, F_1>}|=|\Irr(\GG^{F_1})^{\wbG^{F_1}}|.$$ 
\end{thm}

\begin{proof}
The image of $\w\GG^F$ in $\Out(\GG^F)$ does not depend on the choice of the regular embedding since it also coincides with the image of $\GG_{\mathrm{ad}}^F=\w\GG^F/\Z(\w\GG)^F$. As seen in Sect.\,2.A, one can assume that $\bZ:= \Z (\w\GG)$ is a torus of dimension 1.
 
 Let us form the semi-direct product $\wbG^F\rtimes\spann<\si>$ where $\si$ has order $m$ and acts by $F_1$. Let us show that  $\wbG^F \spann<\si>/\GG^F$ has trivial Schur multiplier. When $F$ is split, then $F_1$ is also split by the hypothesis on action on the center and $\wbG^F \spann<\si>/\GG^F\cong \FF_{q^{}}^\times\rtimes \spann<\si>$ has trivial Schur multiplier by \cite[14.1]{IMN}. When $F$ is not split then the type is $\tE_6$ and both $F$ and $F_1$ act on $\Z(\bG)$ by inversion. On $\wbG/\GG\cong \FF^\times$, $F_1$ acts by $t\mapsto t^{-q_1}$ and we have $-q=(-q_1)^m$. Then $\wbG^F \spann<\si>/\GG^F\cong \spann<\zeta_0>\rtimes \spann<\si>$ where $\zeta_0$ is an element of order $q+1$ in $\FF^\times$ and $\si(\zeta_0)=\zeta_0^{-q_1}$. The fixed point subgroup $\spann<\zeta_0>^{ \spann<\si>}$ is the subgroup of order $q_1+1$, generated by $\zeta_0^{q+1\over q_1+1}$. The latter is clearly $(\zeta_0\si)^m$ in $\spann<\zeta_0>\rtimes \spann<\si>$ since $-q=(-q_1)^m$. Therefore  $\spann<\zeta_0>\rtimes \spann<\si>= CB$ where $B= \spann<\zeta_0\si>$ and $C= \spann<\zeta_0>\unlhd CB$, to which we can apply \cite[14.3]{IMN} since $\Z(CB)\cap C=\spann<\zeta_0^{q+1\over q_1+1}>\leq B$. We get that $\spann<\zeta_0>\rtimes \spann<\si>$ and therefore $\wbG^F\spann < F_1>/\GG^F$ has trivial Schur multiplier in all cases. 
 
 From the above, we conclude that the characters of $\GG^F$ invariant under $(\wbG^F/\GG^F)\rtimes\spann<F_1>$ do extend to $\w\GG^F\rtimes\spann<F_1>$. Clifford theory then implies that 
$$ | \Irr (\GG^F)^{\spann <\wbG^F ,\, F_1>}|=|\Irr (\wbG^F/\GG^F)^{F_1}|^{-1}|\Irr '(\wbG^F)^{F_1}|$$ where $\Irr '(\wbG^F)$ is the subset of $\Irr (\wbG^F)$ of characters restricting irreducibly to $\GG^F$. By a trivial application of Brauer's permutation lemma \cite[6.32]{Isa} applied to the abelian group $\wbG^F/\GG^F$, we have $|\Irr (\wbG^F/\GG^F)^{F_1}|=| (\wbG^F/\GG^F)^{F_1}|$ which in turn is by Lang's theorem $|((\wbG/\GG)^F)^{F_1}|=|\wbG^{F_1}/\GG^{F_1}|$. So indeed
\begin{align}\label{eq1}
| \Irr (\GG^F)^{\spann <\wbG^F ,\, F_1>}|&=|\wbG^{F_1}/\GG^{F_1}|^{-1}|\Irr '(\wbG^F)^{F_1}|.
\end{align}
Since our claim amounts to replacing $(F_1,F)$ with $(F_1,F_1)$ while keeping our number $| \Irr (\GG^F)^{\spann <\wbG^F ,\, F_1>}|$ unchanged, it now reduces to showing that $$|\Irr '(\wbG^F)^{F_1}|=|\Irr '(\wbG^{F_1})|.$$

Denote $\bH =\wbG^*$. As in the proof of \cite[2.3]{CS17C} the characters of $\wbG^F$ are parametrized by $\bH^F$-conjugacy classes of pairs $(s,\la)$ with $s\in \bH_{\mathrm{ss}}^F$ and $\la\in \cE(\Cy_\bH(s)^F,1)$. Moreover, this parametrization is $F_1$-equivariant. Since centralizers of semi-simple elements of $\bH$ are connected by \cite[13.15(ii)]{DM}, there are well-known bijections $$
\bH_{\mathrm{ss}}^F/{\sim}_{\bH^F}\longleftrightarrow \bH_{\mathrm{ss}}^F/{\sim}_{\bH}\longleftrightarrow (\bH_{\mathrm{ss}}/{\sim}_{\bH})^F
$$ (see \cite[4.3.6]{G}). They are $F_1$-equivariant and hence $$
(\bH_{\mathrm{ss}}^F/{\sim}_{\bH^F})^{F_1}\xleftarrow{\ \sim\ } \bH_{\mathrm{ss}}^{F_1}/{\sim}_{\bH^{F_1}}$$ by the obvious map. Therefore
$F_1$-stable $\bH^F$-classes of pairs $(s,\la)$ correspond to $\bH^{F_1}$-classes of pairs $(s,\la)$ with $s\in \bH_{\mathrm{ss}}^{F_1}$ and $\la\in \cE(\Cy_\bH(s)^F,1)^{F_1}$.

Recall that by the simply-connectedness of $\GG$, $\GG^*$ is adjoint and therefore can be thought of as $\bH/\Z (\bH)$. When $s\in \bH$, one denotes by $A(s)$ the group of components of $\Cy_{\GG^*}(s\Z(\bH))$. Recall that it identifies $F_1$-equivariantly with a subgroup of $\Irr(\Z(\GG))$, see \cite[13.14(iii)]{DM}. Via the parametrization of $\Irr(\wbG^F)$ above, the subset $\Irr'(\wbG^F)$ of characters restricting irreducibly to $\GG^F$ corresponds to pairs $(s,\la)$ with $A(s)^F_\la =\{ 1\}$ (\cite[5.1]{L88}, see proof of \cite[2.3]{CS17C}). By the assumption that $F_1$ acts the same as $F$ on $\Z(\bG)$, one has $A(s)^{F}=A(s)^{F_1}$, therefore $\Irr '(\wbG^F)^{F_1}$ corresponds to the $\bH^{F_1}$-classes of pairs $(s,\la)$ with $s\in \bH_{\mathrm{ss}}^{F_1}$ and $\la\in \cE(\Cy_\bH(s)^F,1)^{F_1}$ such that $A(s)^{F_1}_\la =\{ 1\}$.

Now applying Corollary~\ref{F_1F_2} with $\bC =\Cy_{\bH} (s)$ and $(F_1,F_2)$ being the present $(F,F_1)$, for a fixed semi-simple $s\in {\bH}^{F_1}$, we get a bijection $$\cE(\Cy_{\bH} (s)^F,1)^{F_1}\ni\la\mapsto\la '\in \cE(\Cy_{\bH} (s)^{F_1},1)$$ that is equivariant for the action of $\Cy_{\GG^*} (s\Z(\bH))^{F_1}$. So it sends any $\la$ such that $A(s)^{F_1}_\la =\{ 1\}$ to
$\la '$ with $A(s)^{F_1}_{\la '} =\{ 1\}$. 

We now have the equality $$|\Irr '(\wbG^F)^{F_1}|=\sum_{s\in \bH_{\mathrm{ss}}^{F_1}}|\Cy_\bH(s)^{F_1}|^{-1}|\{\la '\in \cE(\Cy_{\bH} (s)^{F_1},1)\mid A(s)^{F_1}_{\la '} =\{ 1\}  \}|.$$
This implies $|\Irr '(\wbG^F)^{F_1}|=|\Irr '(\wbG^{F_1})^{F_1}|$ and therefore $ | \Irr (\GG^F)^{\spann <\wbG^F ,\, F_1>}|=|\Irr(\GG^{F_1})^{\wbG^{F_1}}|$ thanks to \eqref{eq1}. 
\end{proof}

%

\subsection{Graph automorphisms and invariant characters. }

We study a situation related to the extendibility condition in \Ainfty\ from Definition~\ref{A(d)}. Our main application is to groups with connected center, see Corollary~\ref{gammaF_0stab}. Again, the Shintani norm map is used in a crucial way to relate various subspaces of central functions.

\begin{thm}\label{GtildeExt} Let $\w\GG$ be a reductive group defined over a finite field with associated Frobenius map $F_0\colon \w\GG\to\w\GG$. Let $\gamma$ an algebraic automorphism of $\w\GG$ such that $\gamma^2=\Id$ and $\gamma F_0=F_0\gamma$. Let $m\geq 1$ and denote $F=F_0^m$. Assume 
	that any power $F_1=F_0^{m'}$ with $m'\mid m$ satisfies \begin{equation}\label{IrrF1ga}
	\text{  $ |\Irr(\wbG^F)^{\spann<F_1\, ,\, \gamma>}|=|\Irr(\wbG^{F_1})^{\spann<\gamma>}|.   $ }
	\end{equation}
	
	Let $E'$ be the subgroup of $\Aut(\w\GG^F)$ generated by the restrictions of $F_0$ and $\gamma$.
	Then maximal extendibility holds with respect to the inclusion $\wbG^F\unlhd \wbG^F\rtimes E'$.
\end{thm}

We will apply this through the following

\begin{cor}\label{gammaF_0stab} Let $\GG=\bG$ and a pair $\bT\leq\bB$ as in Sect.\,2.A above. We assume that there is a graph automorphism $\gamma$ of order 2  and a Frobenius morphism $F_0$ taking $\xx_\al(t)$ to $\xx_{\gamma(\al)}(t)$, resp. to $\xx_{\al}(t^p)$, for any $t\in \FF$ and $\al\in\pm\Delta$. Recall that there is a regular embedding $\GG\leq\w\GG$ where both $F_0$ and $\gamma$ extend and commute, with $F_0$ defining $\w\GG$ over $\FF_p$.
	
	Let $E'\leq\Aut(\w\GG^F)$ be as above generated by the restrictions of $F_0$ and $\gamma$.
	Then any element of $\Irr (\wbG^F)$ extends to its stabilizer in $\wbG^F\rtimes E'$.
\end{cor}

\begin{proof} Assuming Theorem~\ref{GtildeExt}, it clearly suffices to check that assumption (\ref{IrrF1ga}) is satisfied. Denote $\bH=\wbG^*$ with regard to a maximal torus $\bT^*$. Then arguing as in the proof of Theorem~\ref{thm_Char_count}, the Jordan decomposition of characters in the group $\wbG$ with connected center gives a bijection $$\Irr (\wbG^F)^{F_1}\longleftrightarrow (\coprod_{s\in \bH^{F_1}_{\text {ss}}}\cE(\Cy_\bH(s)^F,1)^{F_1})/{\sim_{\bH^{F_1}}}.$$ Thanks to \cite[3.1]{CS13} this can be assumed to be $\gamma$, $\gamma^*$-equivariant where $\gamma^*\in\Aut(\bH)$ is defined dually to $\gamma$. In the case $(F,F_1)=(F_1,F_1)$, we get a bijection
	$$\Irr (\wbG^{F_1})\longleftrightarrow (\coprod_{s\in \bH^{F_1}_{\text {ss}}}\cE(\Cy_\bH(s)^{F_1},1)/{\sim_{\bH^{F_1}}}.$$ The family of maps $$\Big(\cE(\Cy_\bH(s)^F,1)^{F_1}\lra \cE(\Cy_\bH(s)^{F_1},1)\Big)_{s\in \bH^{F_1}_{\text {ss}}}$$ provided by Corollary~\ref{F_1F_2} commutes with the $\bH^{F_1}$ and $\gamma^*$ actions. So we get a bijection $$\Irr (\wbG^F)^{F_1}\longleftrightarrow \Irr (\wbG^{F_1})$$ that is $\gamma$-equivariant. This clearly implies our claim. We then get the statement by applying Theorem~\ref{GtildeExt}.
\end{proof}

\medskip\noindent{\it Proof of Theorem~\ref{GtildeExt}.} Let $\chi\in\Irr(\w\GG^F)$.
  If $\gamma\not\in E'_\chi$ , then $E'_\chi$ is cyclic and $\chi$ extends to $\wbG^{F}E'_\chi$ by \cite[11.22]{Isa}. We now assume $\gamma\in E'_\chi$ and therefore $E'_\chi =\spann<F_0>_\chi\times \spann<\gamma>$. It is easy to see that the subgroups of $\spann<F_0>$ in $\Aut(\w\GG^F)$ are all of the form $\spann<F_0^{m'}>$ for $m'\mid m$, the latter being the unique subgroup of $\spann<F_0>$ with order ${m\over m'}$. So it suffices to show that, for $F_1$ as in (\ref{IrrF1ga}), any element of $\Irr(\wbG ^F)^{\spann<F_1\, ,\,\gamma>}$ extends to $\wbG ^F\rtimes{\spann<F_1\, ,\,\gamma>}$.
	 
	Let us consider the space $C_1={\Bbb C}\Irr(\wbG^{F_1}) $ of central functions on $\wbG^{F_1}$. Note that $\gamma$ is an automorphism of $\wbG^{F_1}$, hence acts on $C_1$. Since $\gamma$ permutes the basis $\Irr (\wbG^{F_1})$ we have $\Tr(\gamma ,C_1)=|\Irr(\wbG^{F_1})^{\spann<\gamma>} |$ and by (\ref{IrrF1ga})
	 \begin{equation}\label{Trga}
	 \text{  $  \Tr(\gamma ,C_1)=|\Irr(\wbG^F)^{\spann<F_1\, ,\,\gamma>}|.  $ }
	 \end{equation}
	Let us form the semi-direct product $\wbG^{F}\rtimes\spann<F_1>$. 	We fix an extension map \begin{align*}\Lambda\ \colon\ \Irr(\wbG ^F)^{F_1}&\ \longrightarrow\ \Irr(\wbG^{F}\rtimes\spann<F_1>)\cr \chi&\ \longmapsto\Lambda(\chi ) ,\end{align*} giving an extension to each $F_1$-invariant character of $\wbG^F$, which is always possible since $\wbG^{F}\rtimes\spann<F_1>/\wbG^{F}$ is cyclic \cite[11.22]{Isa}.
	
	The Shintani norm map $\Norm_{F/F_1}^{(\w\GG)}\colon \wbG^F /\sim_{F_1}\to {\wbG^{F_1}}/\sim_{F}$ can be composed with multiplication by $F_1$ on the domain. This gives a bijection  $\Norm '\colon \wbG^F{F_1}/{\sim_{\wbG^F}}\to {\wbG^{F_1}}/\sim_{\wbG^{F_1}}$ as explained in Sect.\,\ref{sec4A}.A. The map $\Norm '$ is defined by $y^{-1}F_1(y)F_1\mapsto F(y)y^{-1}$ for $y\in\w\GG$ with $y^{-1}F_1(y)\in\wbG^F$. Composition with $\Norm '$ induces a $\Bbb C$-linear isomorphism $$C_2\xrightarrow{\ \sim\ } C_1$$ where $C_2$ is the space of ${\sim_{\wbG^F}}$-class functions on ${{\wbG^F}}F_1$. This isomorphism commutes with the action of $\gamma$ as can be seen on the definition of $\Norm_{}'$.

	By Proposition~\ref{onxY}, a basis of $C_2$ is given by the family $\{  \restr{\Lambda(\chi)}|{\wbG^F F_1} \mid \chi\in\Irr(\wbG^F)^{F_1}\}$. For $\{\chi ,\chi^\gamma\}$ an orbit of $\gamma$ on $\Irr(\wbG^F)^{F_1}$, denote $$V_{\{\chi ,\chi^\gamma\} }:={\Bbb C}\restr\Lambda(\chi)|{\wbG^F F_1}+{\Bbb C}\restr\Lambda(\chi^\gamma)|{\wbG^F F_1}\subseteq C_2.$$ They form a decomposition of $C_2$ as a direct sum with each $V_{\{\chi ,\chi^\gamma\} }$ of dimension $|{\{\chi ,\chi^\gamma\} }|$. 
	
	Each $V_{\{\chi ,\chi^\gamma\} }$ is $\gamma$-stable. Indeed, for any $\chi\in\Irr(\wbG^F)$, $\Lambda(\chi)^\gamma$ being another extension of $\chi^\gamma$ is of the form $\Lambda(\chi^\gamma)\epsilon$ where $\epsilon$ is a linear character of $\wbG^{F}\rtimes\spann<F_1>/\wbG^{F}$ and consequently $\Big(\restr\Lambda(\chi)|{\wbG^F F_1}\Big)^\gamma =\restr\Lambda(\chi)^\gamma|{\wbG^F F_1}=\epsilon(F_1)\restr\Lambda(\chi^\gamma)|{\wbG^F F_1}\in \CC \restr\Lambda(\chi^\gamma)|{\wbG^F F_1}$. 
	
	The action of $\gamma$ on $C_2$ with regard to this decomposition is as follows.  \begin{enumerate}
		\item If $\chi^\gamma\neq \chi$, then $\restr\Lambda(\chi)|{\wbG^F F_1}$ and $\restr\Lambda(\chi^\gamma)|{\wbG^F F_1}$ are linearly independent by Proposition~\ref{onxY}. Moreover we have just seen that $(\restr\Lambda(\chi)|{\wbG^F F_1})^\gamma $ and $\restr\Lambda(\chi^\gamma)|{\wbG^F F_1}$ generate the same line, so we can take $(\restr\Lambda(\chi)|{\wbG^F F_1})^\gamma $ and $\restr\Lambda(\chi)|{\wbG^F F_1}$ as basis of $V_{\{\chi ,\chi^\gamma\} }$. Then it is clear that $\gamma$ has trace 0 on $V_{\{\chi ,\chi^\gamma\} }$. 
		\item	If $\chi^\gamma=\chi$ and $\Lambda(\chi)^\gamma=\Lambda(\chi)$, then $V_{\{\chi \} }$ is a line where the action of $\gamma$ has trace 1.
		\item If $\chi^\gamma=\chi$ and $\Lambda(\chi)^\gamma\neq\Lambda(\chi)$, then the endomorphism induced by $\gamma$ on the line $V_{\{\chi \} }$ is $-1$, since of order 2. There the trace is $-1$.
	\end{enumerate}
	This implies that $\Tr(\gamma ,C_2)=m_b-m_c$
	where $m_b$ is the number of $\chi\in\Irr(\wbG ^F)^{\spann<F_1\, ,\,\gamma>}$ with $\gamma$-invariant extensions  to $\spann<\wbG ^F,\,F_1>$ and $m_c$ is the number of $\chi\in\Irr(\wbG ^F)^{\spann<F_1\, ,\,\gamma>}$ that do not have any $\gamma$-invariant extension (note that an element of $\Irr(\wbG ^F)^{\spann<F_1\, ,\,\gamma>}$ has no extension to $\spann<\wbG ^F,F_1>$ that is $\gamma$-invariant or all extensions are).

	On the other hand $$\Tr(\gamma ,C_2)=\Tr(\gamma ,C_1)=|\Irr(\wbG ^F)^{\spann<F_1\, ,\,\gamma>}|=m_b+m_c$$ by (\ref{Trga}) above. This forces $m_c=0$, and therefore any element of $\Irr(\wbG ^F)^{\spann<F_1\, ,\,\gamma>}$ extends to $\wbG ^F\rtimes{\spann<F_1\, ,\,\gamma>}$ as required. \qed

\subsection{Stabilizers and extendibility of characters in types $\tB$, $\tC$ and $\tE$}

We now combine Theorem~\ref{pr_Cla_Sh} and Theorem~\ref{thm_Char_count} to obtain property \Ainfty\ of stabilizers in the outer automorphism groups introduced in  Definition~\ref{A(d)}. This clearly implies Theorem B from our Introduction. Type $\tC$, already known from \cite{CS17C}, is not included in the statement below but the same proof applies. To deal with type $\tE_6$ an additional extendibility property is required which will use Corollary~\ref{gammaF_0stab} above.

\begin{thm}\label{gloBCE} {}
Assumption \Ainfty\ of Definition~\ref{A(d)} is satisfied by any $\bG$ of type $\tB$, $\tE_6$ or $\tE_7$ and any Frobenius endomorphism. 
\end{thm}

The following lemma helps reducing the stabilizer part of \Ainfty\ to an abelian 2 or 3-group of outer automorphisms.

\begin{lem}\label{CDr} Let $C\rtimes D$ be a semi-direct product of finite groups where $|C|=r\in\{2,3\}$ and $D$ is abelian. Then a subgroup $X\leq CD$ is $C$-conjugate to some $C'D'$ with $C'\leq C$ and $D'\leq D$ if $X_r$ has the same property in the abelian group $(CD)_r=C\times D_r$.
\end{lem}

\begin{proof} The case $r=2$ is trivial since then $CD$ is abelian. Assume $r=3$. Since $\Aut(C)$ has order 2, one easily reduces to the case where $D=D_3\times D_2$. By assumption we can write $X_3=C'D'$ with $C'\leq C$, $D'\leq D_3$. Now a Sylow 2-subgroup $X'$ of $X$ satisfies $X=X_3X'$ while $^cX'\leq D_2$ for some $c\in CD$. One may assume $c\in C$ and we indeed get $^cX={}^cX_3.{}^cX'=C'(D'.{}^cX')$ since $CD_3$ is abelian. \end{proof}

\medskip\noindent{\it Proof of Theorem~\ref{gloBCE}}. In the types considered $\Z(\bG)$ is cyclic, so there exists a regular embedding $\bG\leq \wbG=\bG\Z(\wbG)$ such that $\Z(\wbG)$ is a torus of rank 1 (see Sect.\,2.A). We abbreviate $\GG:=\bG$. Continuing with the setting and notations of Sect.\,2, one considers a Frobenius endomorphism $F_0\colon \wbG\to\wbG$ defining $\wbG$ over $\FF_p$ and sending $\xx_{\al}(t)$ to $\xx_{\al}(t^p)$ for any $\al\in\Phi$, $t\in \FF$. Assume that $F\colon \w\GG\to\w\GG$ is of the form $F_0^m$ or $\gamma F_0^m$ for some $m\geq 1$ and $\gamma$ is an automorphism of order 2 such that $\restr\gamma|{\GG}$ is associated with the graph automorphism of $\Delta$ in type $\tE_6$. We recall $E$ from Definition~\ref{DefE}, that is $E=\spann<\restr F_0|{\GG^F}, \restr \gamma|{\GG^F}>$ in type $\tE_6$, $E=\spann<\restr F_0|{\GG^F} >$ in all other types considered. Let $Z:=\wbG^F/(\GG^F \Z(\wbG)^F)$ seen as a subgroup of $\Out(\GG^F)$.
	
Let us now check the first part of \Ainfty, i.e. the fact that for any $\chi\in\Irr(\GG^F)$ the stabilizer $(ZE)_\chi$ is $Z$-conjugate to a subgroup of type $Z'E'$ for $Z'\leq Z$, $E'\leq E$.

 Lemma~\ref{CDr} shows that it suffices to study the stabilizers of characters of $\GG^F$ in $Z\times E_r$ for $r=|Z|\in\{2,3\} $. We have to prove that if $\chi\in\Irr(\GG^F)$, $z\in Z$, $e\in E_r$ , then an equality $\chi^{ze}=\chi$ holds only if $\chi^z=\chi^e=\chi$. This is equivalent to $\Irr(\GG^F)^{\spann<ze>} =\Irr(\GG^F)^{\spann<z\, ,\, e>} $ or also 
\begin{equation}\label{FiPt}
\text{  $|\Irr(\GG^F)^{\spann<ze>}| =|\Irr(\GG^F)^{\spann<z\, ,\, e>}| $ for any $z\in Z$ and $e\in E_r$.}
\end{equation}
Note that any element of $E_r$ is an $r'$-power of some $\restr F_1|{\GG^F}$
where $F_1\colon\GG\to\GG$ is a Frobenius endomorphism such that $F$ is a power of $F_1$ and both act the same on $\Z(\GG)^F$. Namely when $F=F_0^m$ then $F_1=F_0^{m_{r'}a}$ for $a\geq 1$ a divisor of $m_r$ provide generators for all subgroups of $E_r$ , while for $F=\gamma F_0^m$ one can consider the elements $\gamma F_0^{m_{3'}a}$ for $a\geq 1$ a divisor of $m_3$. Any change $(z,e)\mapsto (z^k,e^k)$ for some $r'$-integer $k$ leaves both $\spann<ze>$ and $\spann<z,e>$ unchanged. Accordingly we may assume that $e$ in (\ref{FiPt}) is $\restr F_1|{\GG^F}$ with $F_1$ satisfying the hypothesis of Theorem~\ref{thm_Char_count}. We may also assume that $z\not=1$, otherwise (\ref{FiPt}) is trivial. So we assume that the action of $z$ is conjugation by $t\in\wbG^F$ such that $\wbG^F=\spann<\GG^F ,t>$. Theorem~\ref{thm_Char_count} then implies that \begin{equation}\label{FiPt2}
\text{  $|\Irr(\GG^F)^{\spann<z\, ,\,e>}|=| \Irr(\GG_{\mathrm{ }}^{F_1})^{\wbG^{F_1}}| $.}
\end{equation}
Let $t_1\in {\wbG^{F_1}}$ such that its $\sim_{F}$-class is the image by $\Norm_{F/F_1}^{(\wbG)}$ of the $\sim_{F_1}$-class of $t$ (see Theorem~\ref{not_Norm}). Then Theorem~\ref{pr_Cla_Sh} implies that $\wbG^{F_1}=\spann<\GG^{F_1},t_1>$ and therefore \begin{equation}\label{FiPt3}
\text{  $| \Irr(\GG_{\mathrm{ }}^{F_1})^{\wbG^{F_1}}|=| \Irr(\GG_{\mathrm{ }}^{F_1})^{\spann<{t_1}>}| $.}
\end{equation}
On the other hand (\ref{Fixpoints1}) tells us that \begin{equation}\label{FiPt4}
\text{  $| \Irr(\GG_{\mathrm{ }}^{F})^{\spann<ze>}|=| \Irr(\GG_{\mathrm{ }}^{F_1})^{\spann<{t_1}>}| $.}
\end{equation}

Combining those three equalities gives our claim (\ref{FiPt}).

We now turn to the extendibility part of the statement \Ainfty\ of Definition~\ref{A(d)}. We will abbreviate $G:=\GG^F$ and $\w G:=\w\GG^F$. We have to prove that any element of $\Irr(G)$ has a $\w G$-conjugate whose stabilizer in $ZE$ has the structure just proven and that it extends to its stabilizer in $GE$. When $E$ is cyclic, this is trivial by \cite[11.22]{Isa}. So the only case we have to check is the one of non-twisted type $\tE_6$ with $F=F_0^m$ and we must then indeed show that any $\chi\in\Irr(G)$ extends to its stabilizer in $GE$. As in the proof of Theorem~\ref{GtildeExt}, one may assume that the non-trivial graph automorphism $\gamma$ is in $ E_\chi$. Let $F_1=F_0^{m_{2'}}$ be the power of $F_0$ such that $\gamma$ and $F_1$ generate a Sylow $2$-subgroup of $E_\chi$. The only $\gamma$-invariant irreducible character of $\Z(G)$ is the trivial one (see \cite[1.15.2(c)]{GLS3}) so $\Z(G)\leq \ker(\chi)$. Let $\wh \chi$ be the extension of $\chi$ to $G\Z(\wG)$ whose kernel contains $\Z(\wG)$. By definition $\wh \chi$ is $\spann<F_1,\gamma>$-invariant. 

Assume that $\chi$ is $\wG$-invariant. Then there are $|\Z(G)|$ different extensions of $\chi$ in $\Irr(\wG\mid \wh \chi)$. The group $\spann<F_1,\gamma>$ is a $2$-group acting on this set with $3$ elements. Hence $\spann<F_1,\gamma>$ fixes some element $\w\chi \in\Irr(\wG\mid \chi)$. 
By Corollary~\ref{gammaF_0stab} this character $\widetilde\chi$ extends to $\wG\spann<F_1,\gamma>$, and this defines us an extension of $\chi$ to $G\spann<F_1,\gamma>$. According to \cite[11.31]{Isa} this implies that $\chi$ extends to $GE_\chi$, since all Sylow subgroups of $E$ of odd order are cyclic. 

Assume $\chi$ is not $\wG$-invariant. The set $\Irr(\wG\mid \wh \chi)$ consists of just one element $\psi $ obtained by inducing $\wh\chi$ from ${G\Z(\wG)}$ to $\wG$. Accordingly this character is $\spann<F_1,\gamma>$-invariant and has degree $3\chi(1)$. Applying Corollary~\ref{gammaF_0stab} again, we have an extension $\w\psi$ to $\wG\spann<F_1,\gamma>$. This in turn has a Clifford correspondent $\psi_0$ in $\Irr((\wG\spann<F_1,\gamma>)_\chi\mid \chi)$ of degree $\chi(1)$ since $|\wG/\wG_\chi |=3$. Because of its degree, $\psi_0$ is an extension of $\chi$. Again this proves that $\chi$ extends to $GE_\chi$. 
\qed

\section{Towards the local conditions $\mathrm{A}(d)$ and $\mathrm{B}(d)$}

The rest of the paper is essentially devoted to checking conditions $\mathrm{A}(d)$ and $\mathrm{B}(d)$ from Definition~\ref{A(d)} for $d\geq 1$ and groups of types $\tB$ and $\tE$, thus completing the proof of Theorem~\ref{thmA} via the criterion given in Theorem~\ref{Glo+Loc}. In the present section we restate and slightly refine two theorems from \cite{CS17C} that list various requirements to check in a splitting of the tasks that has been followed already for types $\tA$ and $\tC$. In the notation of Sect.\,2.A, where $\bS$ is a Sylow $d$-torus of $(\bG ,F)$, those requirements concern subgroups of the relative Weyl group $\norm{\bG}{\bS}^F/\cent{\bG}{\bS}^F$, but also the normal inclusion $\cent{\bG}{\bS}^F\unlhd \norm{\bG}{\bS}^F$ and related character extendibility questions. We single out the case were $\cent{\bG}{\bS}$ is a torus (regular $d$'s) which needs the longest proof even though the citerion is a bit simpler, see Theorem~\ref{thm_4.4} below. The next two sections will adress the cases of type $\tB$ and $\tE$ respectively.

We remain in the setting of Sect.\,2.A, with $\bG$ a simply-connected simple algebraic group over $\FF$ with pair $\bT\leq \bB$ and root system $\Phi$ of basis $\Delta$. Let $W:=\NNN_\bG(\bT)/\bT$ and $\rho\colon \NNN_\bG(\bT)\to W$ the canonical epimorphism. 
Recall that we denote by $\xx_\al (t)$ ($t\in \FF$ and $\al\in\Phi$) the elements of the root subgroups. The group $\bG$ can be presented by those $\xx_\al (t)$'s subject to the Steinberg relations as in \cite[Thm.~1.12.1]{GLS3}. We also use the elements $\n_\al(t)\in\NNN_\bG(\bT)$ and $\h_\al(t)\in \bT$ for $t\in\FF^\times$ defined in Sect.\,2.A.
Let $F_0\colon\bG\to\bG$ be defined by $F_0(\xx_{\al} (t))=\xx_{\al} (t^p)$ whenever $\al\in\Phi$, $t\in\FF$. Recall $\Aut(\Delta)$ and the group $\Gamma_{\bT\leq \bB}$ of graph automorphisms defined for $\gamma\in\Aut(\Delta)$ by $\xx_\al(t)\mapsto \xx_{\gamma(\al)}(t)$ for any $t\in\FF$, $\al\in\pm\Delta$.

The Frobenius endomorphism we consider for $\bG$ is of the form $F=\gamma_0 F_0^m$ for some $m\geq 1$ and $\gamma_0\in\Gamma_{\bT\leq \bB}$. Let $\phi_0$ be the automorphism of $W$ corresponding to $\gamma_0$. Recall from Definition~\ref{DefE} that $E$ denotes the subgroup of $\Aut(\GG_{\mathrm{sc}}^F)$ generated by the restrictions to $\GG_{\mathrm{sc}}^F$ of $F_0$ and $\Gamma_{\bT\leq \bB}$.

Let us recall Tits' extended Weyl group and our notation for automorphisms.

\begin{defi}\label{VhatE} Let
	$V:= \left\langle \nn_\al(1) \mid  \al \in \Phi \right\rangle \leq \NNN_\bG(\bT)$
	and $H := \bT \cap V$. 
	
	For $e :=|\Aut(\Delta)|\  |V |$ we 
	let $\wh E := \Cy_{em}\times \Aut(\Delta)$ where the first term acts on 
	$\GG_{\mathrm{sc}}^{F_0^{ em}}$ by $F_0$ and $\Aut(\Delta)$ as seen above. 
	
	We 
	write $\wh F_0$ for the element of $\wh E$ acting by $F_0$ and $\wh F$ for 
	$\gamma_0\wh F_0^m$. 
\end{defi}

As in Sect.\,2.A we fix a regular embedding $\bG\leq\w\GG$, such that $F_0$, $\Aut(\Delta)$ and $F$ act also on $\w\GG$. Let $\w\bT =\bT.\Z (\w\GG)$.

Recall that whenever $A$ acts on a set $M$ we denote by $A_{M'}$ the stabilizer of $M'$ in $A$, where $M'\subset M$. 

The conditions \AA d and \BB d of Definition~\ref{A(d)} are checked through the following result which refines slightly Theorem 4.3 of \cite{CS17C}. Recall $d\geq 1$ is a positive integer. 

\begin{thm}\label{thm_loc_gen}
	\label{thm_loc_genii}  Assume there exists an element $v\in V$ such that for the Sylow $d$-torus $\bS$ of $(\bT,vF)$ the groups $N:=\norm \bG \bS ^{vF}$, $\w N:=\norm {\w\GG} \bS ^{vF}$,  $\wh N:=(\Cent_{\GG_{\mathrm{sc}}^{F_0^{em}}\wh E}(v\wh F))_{\bS}$, $C:=\cent \bG \bS ^{vF}$ and $\w C:=\cent {\w\GG} \bS^{vF}$ satisfy the following conditions: 
	\begin{asslist}
		\item \label{thm_loc_gen_i}\label{5_3_i}
		$\bS$ is a Sylow $d$-torus of $(\bG,vF)$.
		\item\label{thm_loc_genii1}
		There exists some set $\calT\subseteq \Irr(C)$, such that 
		\begin{enumerate}[label=(ii.\arabic*)]
			\item \label{5_3_ii_1}
			$\wN_\xi= \w C_\xi N_\xi$ for every $\xi\in \calT$,
			\item \label{5_3_ii_2}
			$(\wN\wh N )_{\xi^N} =\wN_{\xi^N} \wh N_{\xi^N}$ for every $\xi\in \calT$, and 
			\item $\calT$ is an $\wh N$-stable $\w C$-transversal of $\Irr(C)$.
		\end{enumerate}
		
		\item \label{thm_loc_genii2}
		There exists an extension map $\Lambda$ with respect to $C\unlhd N$ such that
		\begin{enumerate}[label=(iii.\arabic*)]
			\item 
			$\Lambda$ is $\wh N$-equivariant. 
			\item \label{5_3_iii_2}
			If $E$ is non-cyclic, then every character $\xi\in\Irr(C)$ has an extension $\wh \xi\in\Irr(\wh N_\xi)$ with $\restr{\wh \xi}|{N_\xi} =\Lambda(\xi)$ and $v\wh F\in\ker(\wh\xi)$. 
		\end{enumerate}
		\item \label{thm_loc_genii3}
		Let $W_d:=N/C$ and $\wh W_d:=\wh N/C$. For $\xi\in \Irr(C)$ and $\w\xi\in\Irr(\w C_\xi\mid \xi)$ let $W_{\w\xi}:=N_{\w\xi}/C$, $W_{\xi}:=N_{\xi}/C$, $K:=\NNN_{W_d}(W_\xi,W_{\w\xi})$ and $\wh K:=\NNN_{\wh W_d}(W_\xi,W_{\w\xi})$. Then there exists for every $\eta_0\in\Irr(W_{\w\xi})$ some $\eta\in\Irr(W_{\xi}\mid \eta_0)$ such that 
		\begin{enumerate}[label=({iv}.\arabic*), ref=\thethm(iv.\arabic*)]
			\item $\eta$ is $\wh K_{\eta_0}$-invariant.
			\item \label{5_3_iv_2}
			If $E$ is non-cyclic, then $\eta$ extends to some $\wh \eta\in \Irr(\wh K_{\eta})$ with $v\wh F\in\ker(\wh\eta)$.
		\end{enumerate} 
	\item \label{5_3v} If $\bG$ is of type $\tD$ maximal extendibility holds with respect to $N\unlhd \w N$.
	\end{asslist}
	Then conditions \AA d and  \BB d from Definition~\ref{A(d)} are satisfied by $\bG$.
\end{thm}

\begin{proof}
	The assumptions are the same as in Theorem~4.3 of \cite{CS17C}, except for (ii.3) which is here strengthened. The cited theorem shows that \AA d is satisfied. 
	
	Let's show that \BB d is also satisfied. By assumption \ref{5_3v} maximal extendibility holds with respect to $N\unlhd\w N$ since, in types $\neq \tD $, $\Z(\bG)$ is cyclic which allows to take a regular embedding such that $\w\GG/\bG$ is of dimension 1 (see Sect.\,2.A) and $\w N/N$ is therefore cyclic. It remains to construct an extension map $\w\Lambda$ with respect to $\w C\unlhd\w N$ that is compatible with the action of $\Irr(\w N/N)$ by multiplication and is $\wh N $-equivariant.
	Let $\chi\in\Irr(\w C)$ and $\chi_0\in \calT\cap \Irr(C\mid \chi)$. Then $\chi =\w\chi_0^{\w C}$ for some extension $\w\chi_0$ of $\chi_0$ to $\w C_{\chi_0}$. According to \cite[4.1]{S10b} there exists a common extension $\psi$ of $\restr{\Lambda(\chi_0)}|{N_{\w\chi_0}}$ and $\w\chi_0$ to $N_{\w\chi_0} \w C_{\chi_0}$. By this definition $\psi^{\w C {N_{\w\chi_0}}}$ is an extension of $\w\chi_0^{\w C}$. 
	By the equation in (ii.1) this character is an extension of $\chi$ to $\w N_{\chi}$ since $\w N_\chi=\w C(\w N)_{\w \chi_0}\leq \w C(\w N)_{ \chi_0}=\w C(\w C_{\chi_0} N_{ \chi_0})=\w C N_{\chi_0}$. Accordingly we can define $\w \Lambda(\chi)$ as $\psi^{\w C {N_{\w\chi_0}}}$.
	
	Since $\calT$ is $\wh N$-stable and $\Lambda$ is $\wh N \w C$-equivariant, $\w\Lambda$ is clearly $\wh N$-equivariant.  By construction it is also compatible with multiplication by linear characters of $\w N/N$. 
\end{proof}

The cases where the Sylow $d$-torus $\bS$ is such that $\cent{\bG}{\bS}$ is a torus correspond to the so-called \textit{regular numbers $d$ for the pair $(\bG ,F)$}, in the sense of \cite[p. 107]{Br}. Then the above criterion for \AA d and \BB d somehow simplifies into the following (see \cite[4.4]{CS17C}).

\begin{thm}\label{thm_4.4}
	Let $d\geq 1$ be a regular number for $(\bG,F)$, i.e. the centralizer of a Sylow $d$-torus is a (maximal) torus. Assume that there exists an element $v\in V$ (see Definition~\ref{VhatE}) such that, denoting by $\bS$ the Sylow $d$-torus of $(\bT ,vF)$ and $\bN:=\norm{\bG}{\bS}$, the following properties hold:
	\begin{asslist}
		\item \label{thm_loc_geni1}
		$\rho(v)\phi_0$ is a $\zeta$-regular element of $W\phi_0$ in the sense of Springer (see \cite[Def.~2.5]{S10a} or \cite[Ch. 5]{Br}) for some primitive $d$-th root of unity $\zeta\in\CC$.
		\item \label{thm_loc_geni2}
		$\rho(V_d)=W_d$ with $V_d:=V^{vF}$ and $W_d:=\Cent_W(\rho(v)\phi_0)$.
		\item \label{thm_loc_geni3} \label{5_4_iii} Let $\wh N:=\Cent_{\bN^{F_0^{em}} \wh E}(v\wh F)_{}$ and  $\wh V_d:=(V\wh E)\cap\wh N$ (see Definition~\ref{VhatE}).
		There exists an extension map $\Lambda_0$ with respect to $H_d:=H^{vF}\unlhd V_d$ such that:
		\begin{enumerate}[label=({iii}.\arabic*)]
			\item 
			$\Lambda_0$ is $\wh V_d$-equivariant.
			\item \label{5_4_iii_2}
			If $E$ is non-cyclic, $\Lambda_0(\la)$ extends for every $\la\in\Irr(H_d)$ to some $\wh \la\in\Irr((\wh V_d)_\la)$ with $v\wh F\in \ker(\wh\la)$.
		\end{enumerate}
		\item \label{thm_loc_geni4}
		Let $C:=\bT^{vF}$, $\w C:=\w\bT^{vF}$, $N:=\bN^{vF}$, $W_d:=N/C$ and $\wh W_d:=\wh N/C$. 
		For $\xi\in \Irr(C)$ and $\w\xi\in\Irr(\w C\mid \xi)$ let $W_{\w\xi}:=N_{\w\xi}/C$, $W_{\xi}:=N_{\xi}/C$, $K:=\NNN_{W_d}(W_\xi,W_{\w\xi})$ and $\wh K:=\NNN_{\wh W_d}(W_\xi,W_{\w\xi})$. Then there exists for every $\eta_0\in\Irr(W_{\w\xi})$ some $\eta\in\Irr(W_{\xi}\mid \eta_0)$ such that 
		\begin{enumerate}[label=({iv}.\arabic*)]
			\item $\eta$ is $\wh K_{\eta_0}$-invariant.
			\item \label{5_4_iv_2}
			If $E$ is non-cyclic, $\eta$ extends to some $\wh \eta\in \Irr((\wh W_d)_{\eta})$ with $v\wh F\in\ker(\wh\eta)$.
	\end{enumerate} 
\item If $\bG$ is of type $\tD$ maximal extendibility holds with respect to $N\unlhd \w N$.\end{asslist} 
	Then conditions  $\mathrm{A}(d)$ and  $\mathrm{B}(d)$ of Definition~\ref{A(d)} are satisfied.
\end{thm}

\begin{proof}
	For condition \AA d, this is Theorem~4.4 of \cite{CS17C}. As explained in the proof of Theorem~\ref{thm_loc_gen}, maximal extendibility with respect to the inclusion $N\unlhd\w N$ is ensured in all types. There remains to check the part of condition \BB d about the inclusion $\w C\unlhd\w N$. We have $\w N =\w C N$ since $ \norm {\w\GG} \bS = \norm {\bG} \bS  \cent {\w\GG} \bS $ and the intersection $\norm {\bG} \bS \cap \cent {\w\GG} \bS $ is connected being $\cent {\bG} \bS$. So we can argue as in the proof of Theorem~\ref{thm_loc_gen} with $\Irr(C)$ replacing $\calT$, since in addition $\w C$ is abelian.
\end{proof}



\section{The local conditions for type $\tB$}\label{sec_locB}
In this section we verify the local conditions \AA d and \BB d from Definition \ref{A(d)} via the criteria given before in the cases where $\bG$ is of type $\tB$.

In the following  $\bG$, $\Phi$, $\w\GG$, $F$, $q=p^m$ and $E$ are defined as in Sect.\,2.A. In this part we assume that $\Phi$ is of type $\tB_l$ (so that $\GG_{\mathrm{sc}}^F={\mathrm{Spin}}_{2l+1}(p^m)$ though this matrix representation won't be used).
Our aim is to verify assumptions \AA d and \BB d of Definition~\ref{A(d)}, thus establishing Theorem \ref{thmA} for that type thanks to Theorem~\ref{Glo+Loc}. 

In type $\tB_l$ the group $E$ is cyclic by definition and hence the extendibility part of condition \AA d is automatically satisfied. We also single out the cases where $\Out(\GG_{\mathrm{sc}}^F)$ is cyclic.  

\begin{lem}\label{simplifications}
	Condition \AA d holds for any $d\geq 1$ if $l=2$ or $p=2$ or $q=p^{m}$ for an odd integer $m$.
\end{lem}

\begin{proof}
	For $l=2$ the group $\bG$ is isomorphic to a simply-connected simple group of type $\tC_2$ and hence the statements \AA d and \BB d follow from 5.1 and 6.1 of \cite{CS17C}, respectively. For even $p$ the group $\wbG$ can be taken to be $\bG$ since $\Z(\bG)$ is trivial. This implies the statement in that case. If $m$ is odd the group $E$ has odd order and then the claim follows from the fact that in the notation of Definition~\ref{A(d)}, $\wbG^F/(\GG_{\mathrm{sc}}^F \cent{O}{\GG_{\mathrm{sc}}^F})$ is the Sylow $2$-subgroup of the cyclic group $O/ (\GG_{\mathrm{sc}}^F \cent{O}{\GG_{\mathrm{sc}}^F})$. 
\end{proof}
The above statement justifies that in the following we assume $8\mid(q-1)$ and $l\geq 3$. Here is more notation specific to type $\tB_l$. 

\begin{notation}\label{n_typB}
	The root system $\Phi$ of $\bG$ and its system of simple roots $\Delta =\{\al_1,\ldots,\al_l\}$ are given as follows (see for instance \cite[1.8.8]{GLS3} with a slightly different convention). Letting $(e_1,\dots ,e_l)$ be the canonical orthonormal basis of $\mathbb{R}^l$ with its euclidean scalar product, one takes $\al_1= e_1$, $\al_2=e_2-e_1$, $\dots$ , $\al_l=e_l-e_{l-1}$. We identify the Weyl group $W$ of $\bG$, a Coxeter group of type $\tB_l$ with the subgroup of bijections $\si$ on $\{ \pm 1, \ldots \pm l\} $ such that $\si(-i)=-\si(i)$ for all $1\leq i\leq l$, see also \cite[\S 5]{S10a}. This group is denoted by $\Sym_{\pm l}$. 
	Via the natural identification of $\Sym_l$ with a quotient of $\Sym_{\pm l}$ and the identification of $\Sym_{\pm l}$ with $W$, the map $\rho:\bN\ra W$ induces an epimorphism $\ov\rho\colon V\longrightarrow \Sym_l$ (see Definition~\ref{VhatE}). In $\Sym_{\pm l}$ the set
	$$ \left \{ 
	\sigma \in \Sym_{\pm l} \, \, \big| |\{1,\cdots, l\} \cap \sigma^{-1}(\{-1,\cdots, -l\})| \text{ is even } \right \} $$
	forms a normal subgroup with index $2$, naturally isomorphic to a Coxeter group of type $\tD_l$. We denote by $W_{\tD}$ the associated subgroup of $W$ and $\bN_c:=\rho^{-1}(W_{\tD})$. 
\end{notation}

\subsection{The regular case: Construction of a Sylow $d$-torus}\label{sec10A}
We first consider the case where $d$ is a regular number for $W$, that is a divisor of $2l$, see \cite[Tab.~1]{S10a}. This is also equivalent to $\cent{\bG}{\bS}$ being a torus for any Sylow $d$-torus $\bS$, see \cite[2.5, 2.6]{S09} and we will establish condition \AA d by use of Theorem~\ref{thm_4.4}.

As in \cite[\S 4]{CS17C} we choose a Sylow $d$-twist in the sense of \cite[Def.~3.1]{S09} using $$v_0:=\n_{\al_1}(1)\cdots \n_{\al_l}(1)\in V$$ (see Definition~\ref{VhatE}). Note that 
$\rho(v_0)=(1,2,\ldots,l,-1,-2,\ldots, -l).$ 
Let $$v:=(v_0)^{\frac{2l}{d}}$$ and let $d_0$ be the length of the $\ov\rho (v)$-orbits and $a$ the number of orbits. Then 
\begin{align}  a={l\over d_0}\ \text{ and } \ 
d_0&=\begin{cases} d& \text{ if } 2\nmid d,\\
\frac d 2& \text{ if } 2\mid d.\\
\end{cases} \label{defd0}\end{align}
Recall $\bN_c=\rho^{-1}(W_\tD)$ and denote $N_c:=\bN_c^{vF}$.
\begin{lem}\label{lem106}
	We have $v\in N_c$ if and only if $2\mid a $ or $2\nmid d$. 
\end{lem}
\begin{proof}
	This follows from the characterization of the elements of $W_\tD$ as elements of $\Sym_{\pm l}$.
\end{proof}

\begin{lem}\label{lemrhoVd_B}
	The element $\rho(v)$ is a $\zeta$-regular element of $W$ for any primitive $d$-th root of unity $\zeta\in \CC$ and 
	\begin{align}\label{rhoVd_B}
	\rho(V_d)=W_d,
	\end{align}
	where $V_d:=V^{vF}$ and $W_d:=\cent W v$. In particular assumptions (i) and (ii) of Theorem~\ref{thm_4.4} hold for $v$.
\end{lem}
\begin{proof}
	The group $W$ is at the same time a Coxeter group of type $\tC_l$. The element $\rho(v)$ was proved to be a $d$-regular element for $W$, see \cite[5.A]{CS17C}. The arguments there show  $\rho(V_d)=W_d$ as a consequence of an analogous property of the braid groups. The same considerations apply since the braid groups coincide for type $\tB_l$ and $\tC_l$. 
\end{proof}

\subsection{The regular case: Maximal extendibility with respect to $H_d\unlhd V_d$}

In the next step we verify that for the element $v$ the groups $H_d:=V^{vF}\cap \bT$ and $V_d=V^{vF}$ satisfy assumption (iii) of Theorem~\ref{thm_4.4}. Note that the groups $H_d$ and $V_d$ in type $\tB_l$ differ from what they are in type $\tC_l$. 

\begin{thm}\label{vgoodB} 
	Maximal extendibility holds with respect to $H_d\unlhd V_d$. As a consequence, assumption (iii) of Theorem~\ref{thm_4.4} holds. 
\end{thm}

As in the proof of \cite[Thm.~5.3]{CS17C} we first analyze the structure of $H_d$ and $V_d$ before we extend the characters of $H_d$ to their inertia groups in $V_d$. 

\textbf{The group $H_d$.} Let $\oo v:=\ov\rho (v)$ and $\cO_k$ be the $\oo v$-orbit on $\{1,\dots ,l\}$ containing $k$. Let $\varpi\in\FF$ be a fourth root of unity and $$h_0:=\h_{e_1}(-1).$$ For $1\leq k\leq {a}$ let $$h_k:=\prod_{i\in \cO_k}\h_{e_i}(\varpi).$$ By Chevalley relations the following equalities hold:
\begin{align}
h_k ^2= \h_{e_1}(-1)^{|\cO_k|}=h_0^{|\cO_k|}\label{eq10.1}
&\text{ and }\\
h_k^v= \begin{cases} h_k & \text{if } 2\nmid d\\
h_0h_k& \text{otherwise.} \end{cases}&
\end{align}

Note that for any $h\in H$ there exist elements $t_i\in \spann<\varpi>$ such that $h=\prod_{i=1}^l \h_{e_i}(t_i)$ and $(\prod t_i)^2=1$. 
One shows that $h_k\notin H_d$ and $h_k h_{k+1}\in H_d$ for every $1\leq k <a$. Hence $H_d$ is an elementary abelian $2$-group of rank $a$, namely
$$H_d:=\cent Hv=\spann<h_0> \times \spann<h_1h_2> \times\cdots\times \spann<h_{a-1}h_{a}>.$$ 

\textbf{The group $V_d$ and some of its elements.}
Let $\ov c_1 \in \cent W{\rho(v)}$ be defined as in the proof of Theorem 5.3 of \cite{CS17C}: If $2\mid d$ let $$\ov c_1:=
(1,a+1,\ldots, a (d_0-1)+1, -1, -a-1,\ldots, -a (d_0-1)-1 )\in \cent W {\rho(v)}.$$ 
This is the cycle of $\rho(v)$ containing $1$. If $2\nmid d$ let 
\begin{align*}
\ov c_1':=
&(\,\,\,1,&\,2a+1,&\,\,\,\, 4a+1,& \ldots, \,(d-1)a+1,& -a-1,&\ldots, -(d-2)a-1 )&\\
&(- 1,&- 2a-1,&-4a-1,& \ldots, -(d-1)a-1,& a+1,&\ldots, (d-2)a+1 )&
\in \cent W {\rho(v)}.
\end{align*} 
Let $\ov c_1:=\ov c_1'\prod_{i=0}^{d-1 } (ia+1,-ia-1) \in \cent W {\rho(v)}$. Note that $\ov c_1'$ is the pair of cycles of $\rho(v)$ containing $\pm 1$. 	

Because $\rho(V_d)=W_d$ there exists some $c_1\in V_d$ with $\rho(c_1)=\overline{c}_1$. Let $\Phi_1:=\Phi\cap {\spann<\pm e_i\mid i\in \cO_1> }$ and $V_{\cO_1}:=\spann< \n_\al (\pm 1)\mid \al\in \Phi_1>$. In the next step we construct $c_1\in(V_{\cO_1})^{vF}$.
By definition $\n_\al(\pm 1)^2\in H$ and hence $c_1$ can be chosen to be contained in $\bT V_{\cO_1}$. 
Let $H_{1}:=\spann< \h_\al (\varpi)\mid \al\in \Phi_1>$, 
$\Phi_1':=\Phi\cap {\spann<\pm e_i\mid i\notin \cO_1> }$, $H_{1'}:=\spann< \h_\al (\varpi)\mid \al\in \Phi_1'>$ and $V_{\cO_1}:=\spann< \n_\al (\varpi)\mid \al\in \Phi_1>$. Then $c_1\in H_{1'}V_{\cO_1}$. Since $V_{\cO_1}\cap H_{1'}=\spann<h_{e_1}(-1)>$ by Chevalley relations we see that $c_1=c_1' h$ for some $c_1'\in V_{\cO_1}$ and $h\in H_{1'}$. Let $J\subseteq \{1,\ldots , l\} \setminus \cO_1$ with $h=\prod_{i\in J} h_{e_i}(\varpi)$. The equality $c_1^v=c_1$ implies $[h,v]=[c_1',v]\in \{\h_{e_1}(\pm 1)\} $. Since $c_1\in V$ either $c_1',h\in V$ or $c_1',h\notin V$. As $[h,v]\in \{\h_{e_1}(\pm 1)\} $ we see that $J$ is $\oo{ v}$-stable, hence a union of orbits $\cO_k$. If $d_0=d$ this implies $[h,v]=1$ and if $2\nmid d_0$, $[h,v]=1$ and $h\in H$ are equivalent. Going through all possible cases one sees that either $c_1'$ or $c_1'h_1$ is contained in $V_{\cO_1}\cap V_d=V^{vF}$. In the following we denote this element by $c_1$. 
The group $ (V_{\cO_1}\cap V_d)\cent {H_1}v$ is generated by $h_1$, $h_0$ and $c_1$, i.e., 
$ (V_{\cO_1}\cap V_d)\cent {H_1}v=\spann <h_1,h_0,c_1>$.

Recall $v_k:=\n_{\al_k}(1)\in \bG$. Like in the proof of Theorem \cite[5.3]{CS17C} let $p_k:=\prod_{i=0}^{d_0-1}{v_k}^{v^i}$ for $1\leq k \leq (a-1)$. The considerations from there show 
$$
\rho(p_k)(\cO_k)=\cO_{k+1} \text{ and }\rho(p_k)(\cO_{k+1})=\cO_{k},$$ 
as well as $p_k\in V_d$ and $p_k \in V_{\cO_k\cup\cO_{k+1}}$, where $V_{\cO_k\cup\cO_{k+1}}:=\spann< \n_\al (\pm 1)\mid \al\in \Phi_{k,k+1}>$ with $\Phi_{k,k+1}:=\Phi\cap {\spann<\pm e_i\mid i\in \cO_k\cup \cO_{k+1}> }$. The square of $p_k$ satisfies 
\begin{align}\label{eq10.pk2}
\begin{split}
p_k^2&=\prod_{i=0}^{d_0-1}({v_k}^{v^i})^2=
\prod_{i=0}^{d_0-1} \h_{e_{\oo{v}^i(k)}-e_{\oo{v}^i(k+1)}} (-1) \\
&=\prod_{j\in \cO_k}\h_{2e_j}(\varpi)\left (\prod_{j\in \cO_{k+1}}\h_{2e_j}(-\varpi)\right )=
h_k h_{k+1} h_0^{|\cO_k|}\in H_d. 
\end{split}
\end{align}
Note that since the elements $v_1,\dots, v_{a-1}$ satisfy the braid relations this applies also to $p_k$. For $2\leq k \leq a$ let $c_k:=(c_1)^{p_1 \cdots p_{k-1}}$. 

From their definitions one shows that these elements satisfy the following equations: 
For every $0\leq i\leq a$, $1\leq j \leq a$ and $1\leq k <a$ we have 
\begin{align}
h_i^{c_j}&=\begin{cases} h_i &\text{if } i\neq j,\\
h_i h_0 & \text{if } i=j,	\end{cases} \label{eq10.3hicj}
\end{align}
and also 
\begin{align}\label{eq10.4}
h_i^{p_k}&=\begin{cases} h_i &\text{if } i\notin \{k,k+1\},\\
h_{i+1} & \text{if } i=k,\\
h_{i-1} & \text{if } i=k+1.	\end{cases}
\end{align}
Further by the Chevalley relations 
\begin{align}\label{eq10.5} [c_i,c_j]&=h_0 	\end{align} 
for every $1\leq i,j\leq a$, see also \cite[2.1.7(c), proof of Lem.~10.1.5]{S07}.
Since $\rho(V_d)=\Spann<\ov c_i|1\leq i\leq a>\rtimes \Sym_a$ we see that the elements $c_i$ ($1\leq i \leq a$) together with the elements $p_k$ ($1\leq k < a$) generate the group $V_d$. 

Those elements now play a crucial role in describing the Clifford theory for $H_d\unlhd V_d$.
\begin{proof}[Proof of Theorem \ref{vgoodB}]
	Let $\la\in \Irr(H_d)$, $\wh H_d:=\spann <h_1>H_d$ and $\la_1\in\Irr(\wh H_d\mid \la)$. Note that by Equation \eqref{eq10.1} the group $\spann <h_0,h_1>$ is a cyclic group of order $4$ for odd $|\calO_1|=d_0$ and a Klein 4-group for even $d_0$. In the following we extend $\la$ to $V_{d,\la}$. We construct the extension in those two cases differently. Further this construction also depends on the value of $\la(h_0)$. 
	
	\medskip
	
	For completeness note that for $a=1$ the group $H_d$ is a cyclic group of order $2$ and $V_d=\spann<c_1, h_0>$ is abelian by \eqref{eq10.3hicj}. This implies the statement in that case. 
	
	\medskip
	
	We consider now the case where $a\geq 2$ and $\la(h_0)=-1$. Then $\wh H_d$ is the central product of the groups $\spann <h_0,h_i>$ ($0\leq i \leq a$).
	If $2\nmid d_0$, the group $\spann <h_0,h_1>$ is a cyclic group of order $4$ and by \eqref{eq10.3hicj} the elements $c_i$ ($1\leq i \leq a$) act on the groups $\spann <h_0,h_i>$ by $h_0\mapsto h_0$ and $h_i\mapsto h_ih_0$. 
	Also if $2\mid d$, the group $\spann <h_0,h_1>$ is a Klein 4-group and by \eqref{eq10.3hicj} the elements $c_i$ ($1\leq i \leq a$) act on the groups $\spann <h_0,h_i>$ by inversion. Hence $\la_1$ is $\Spann<c_i|1\leq i \leq a>$-conjugate to a character $\la'_1$ with 
	$$\la'_1(h_1)=\cdots =\la'_1(h_a).$$
	After replacing $\la$ and hence $\la_1$ by a suitable $V_d$-conjugate we can assume that $\la_1$ satisfies already this equation. 
	
	From equations \eqref{eq10.3hicj} and \eqref{eq10.4} we see that 
	\begin{equation}\label{eqVdla1}
	SC'= \wh H_d V_{d,\la_1}
	\end{equation}
	where $$S:=\wh H_d \Spann <p_i | 1\leq i <a > \und C':=\wh H_d \Spann<c_i^2| 1\leq i< a>.$$ 
	Let $c_0:=c_1\cdots c_a $. Clearly $V_{d,\la_1}\leq V_{d,\la}$. Recall $|\wh H_d: H_d|=2$. This implies $|V_{d,\la}:V_{d,\la_1}|\leq 2$. According to \eqref{eq10.3hicj} we see that $\la_1^{c_0}(h_i)=\la_1(h_0h_i)=-\la_1(h_i)$ for every $1\leq i \leq a$ and hence $\la_1^{c_0}$ is another extension of $\la$ to $\wh H_d$. This proves $c_0\in V_{d,\la}$ and hence 
	\begin{align}\label{VlamndainWB}
	V_{d,\la}&= V_{d,\la_1}\spann<c_0> .
	\end{align}
	Recall $\bN_c:=\rho^{-1}(W_{\tD})$. We have \begin{equation}\label{Vdla1inNc}
	V_{d,\la_1}\leq \bN_c\ .
	\end{equation}
	
	Note that $C'$ is abelian and $\la_1$ extends to some $\wh \la_1 \in \Irr(C')$ with $\wh \la_1(c_i^2c_{i+1}^{-2})=1$. By that definition $\wh \la_1$ is $S$-invariant. We see that $S\cap H_d=\Spann<p_k^2|1\leq k <a>$. 
	According to \eqref{eq10.pk2} \begin{equation}\label{lapk2}
	\la(p_k^ 2)=\la(h_{k})\la(h_{k+1}) \la(h_0)^{d_0}= \la(h_{k+1})^2 (-1)^{d_0},
	\end{equation} since $\la(h_k)=\la(h_{k+1})$. If $2\nmid d_0$, then $\la(h_k)=\la(h_{k+1})$ is a primitive fourth root of unity and hence 
	$$\la(p_k^ 2)=\la(h_k)^2 (-1)^{d_0}= (-1) (-1)=1.$$ 
	If $2\mid d_0$ we see that $\la(h_k)=\la(h_{k+1})\in \{\pm 1\}$, and hence 
	$$\la(p_k^ 2)=\la(h_k)^2 (-1)^{d_0}= 1 \cdot 1=1.$$ 
	Altogether this implies that $C'S/\ker(\la)$ has a semidirect product structure and hence $\wh \la_1$ extends to $C'S$. Note that $[C'S\cap V_d,v]=\{ 1\}$ and even in general $ [C'S\cap V_d,c_0]=\{ 1\}$.	This implies that $\restr \wh \la_1|{C'S \cap V_d}$ is $c_0$-invariant and extends to $V_\la$. 
	
	\medskip

	It remains to consider the case where $\la(h_0)=1$. Then $c_i^{d_0}\in \spann <h_0>$ since $\cent{\spann <h_0,h_i>}{\spann <c_i>}=\spann <h_0>$ by the Chevalley relations. 
	After some $V_d$-conjugation there exists an extension $\la_1\in\Irr(\wh H_d)$ of $\la$ and some $1\leq k_0\leq a$ such that $\la_1(h_i)=1$ for $1\leq i \leq k_0$ and $\la_1(h_i)=-1$ for $i>k_0$. In this case $V_{d,\la_1}=  C'S$ with $S:=\wh H_d \Spann <p_i| 1\leq i <a \text{ and } i \neq k_0>$ and $C':=\wh H_d \Spann<c_i| 1\leq i\leq a>$.
	Again the group $C'/\spann <h_0>$ is abelian and there exists some extension $\wh \la_1$ to $C'$ of $\la_1$ such that $\wh \la_1(c_i c_{i+1}^{-1})=1$. Then $\wh \la_1$ is $S$-invariant and one can extend the character to some $\tilde \la_1\in \Irr(C'S)$. 
	Now let $x\in V_{d,\la}\setminus V_{d,\la_1}$. If this element exists, $x^2\in V_{d,\la_1}$, since $x$ sends $\la_1$ to another extension of $\la$. 
	Conjugation with $x$ satisfies $\spann <h_0,h_i>^x=\spann <h_0,h_{i'}>$ for some $i$ with $1\leq i \leq k_0$ and $i'$ with $i' > k_0$. Accordingly 
	\begin{equation}\label{orhox}
	\o\rho(x) (\{1,\ldots,k_0\} )=\{k_0+1,\ldots,a\} .
	\end{equation}
	Without loss of generality we can assume $x\in S$ with $x^2\in H_d$. Since $c_i^x\in \{c_1,\ldots , c_a\}$ we see that $\restr \tilde \la_1|{C'S\cap V_d}$ is $x$-invariant. Hence $\restr\tilde \la_1|{C'S\cap V_d}$ is $x$-invariant and $\la$ extends to $V_{d,\la}$. 
	
	In all possible cases the characters of $H_d$ extend to their inertia group in $V_d$.
	
	According to Remark 4.5 of \cite{CS17C} the above statement proves that in type $\tB_l$ the assumption (iii) of Theorem~\ref{thm_4.4} is satisfied with our chosen element $v$. 
\end{proof}

For later use we state the following observations from the above proof explicitly. 
\begin{prop}\label{prop108}
	Assume $2\mid d$, recall $\wh H_d:=\spann <h_1>H_d$. Let $\lambda\in \Irr(H_d)$ and $\wh \lambda\in\Irr(\wh H_d\mid \lambda)$. Then 
	\begin{enumerate}
		\item \label{prop108a}
		If $\lambda(h_0)=-1$ and $V_{d,\lambda}\not \leq N_c$ , then $2\nmid a$. Then $V_{d,\lambda}\cap N_c=V_{\wh \lambda}$.
		\item \label{prop108b}
		If $\lambda(h_0)=1$ and $V_{d,\lambda}\neq V_{d,\wh \lambda}$ , then $2\mid a$.
	\end{enumerate}
\end{prop}
\begin{proof}
	By \eqref{VlamndainWB} and \eqref{Vdla1inNc} we have that $V_{d,\lambda}=V_{d,\la_1}\spann<c_0>$ and $V_{d,\la_1}\leq \bN_c$ , where $c_0=c_1\cdots c_a$. Note that 
	$$ c_0\in \begin{cases}
	\bN_c&\text{if } 2\mid a ,\\
	V_d\setminus	\bN_c &\otw.
	\end{cases}$$
	This implies statement (a). 
	
	For the proof of (b) recall that an element $x\in V_{d,\la}\setminus V_{d,\wh \la}$ satisfies Equation \eqref{orhox}. Such an element can only exist if $a=2k$ and hence $2\mid a$. 
\end{proof}

\subsection{The regular case: Character theory of the relative inertia groups}\label{subsec10D}
Recall $C:=\bT^{vF}$ and $\wt C:=(\wt \bT)^{vF}$. In the following let $W_\xi:=N_\xi/C$ and $W_{\wt\xi}:=N_{\wt \xi}/C$ with $\xi\in\Irr(C)$ and $\wt\xi\in\Irr(\w C)$ (often an extension of $\xi$). 
 
\begin{defi}
	If $Y\unlhd X$ and $\chi\in \Irr(Y)$ we call the group $X_\chi/Y$ the \textit{relative inertia group (of $\chi$ in $X$)}. 
\end{defi}

In this sense assumption (iv) of Theorem~\ref{thm_4.4} considers the character theory of the relative inertia groups for characters $\xi\in\Irr(C)$ in $N$. In particular the characters of the relative inertia groups $N_\xi/C$ are compared with those of $N_{\wt\xi}/C$ where $\wt\xi\in\Irr(\w C\mid \xi)$. In view of describing the groups we introduce first some notation (similar to the one in \cite[5.4]{CS17C}). 

In order to describe the action of $N$ on $C$, respectively $\wt C$, we introduce the groups $C_0$ and $\calG$ together with an action of $\calG$ on $C_0$.

\begin{notation}[The action of $\calG$ on $C_0$]
	For $2\nmid d_0$ we make the following definitions. Let $\epsilon=(-1)^{\frac d {d_0}}$, $C_0:=\Cy_{2(q^{d_0} + \epsilon)}$, the cyclic group of
	order $2(q^{d_0} + \epsilon)$ and $C_{0,2}$ its subgroup of order $2$. 
	The automorphism of $C_0$ given by $\zeta \mapsto \zeta^{q^{d_0}+\epsilon} \zeta^q$ has order $d$. Hence let $\Cy_d\times \Cy_2$ act on $C_0$ by letting a generator of the first factor act by $\zeta \mapsto \zeta^{q^{d_0}+\epsilon} \zeta^q$ and letting the second act by inversion. (Note that this action is actually well-defined.)
	Let $\mathcal G\leq \Cy_d\times \Cy_2$ be defined by 
	$$\mathcal G:= \begin{cases} \Cy_d\times \Cy_2,& \text{if }d=d_0 ,\\
	\Cy_d\times\{ 1\},& \text{if } d=2d_0.	\end{cases}$$
	
	For $2\mid d_0$ and hence $d=2d_0$ we set $\epsilon:=-1$. 
	Let $C_0:=\Cy_2\times \Cy_{q^{d_0}+1}$ be identified with a subgroup of
	$\FF_{q^{2d}}^\times \times \FF_{q^{2d}}^\times$, and $C_{0,1}=\Cy_2\times\{ 1\}$. The map $(a,b) \mapsto (ab^{\frac{q^{d_0}+1} 2} , b^q)$
	defines a group automorphism of $C_0$ of order $d$. Hence there is a well-defined action of $\calG:=\Cy_d$ on $C_0$ thus defined. 
\end{notation}

\begin{rem}[Determining the group $W_{\w\xi}$ ]\label{remwhC}
	Let $\xi\in\Irr(C)$. By definition $\w\bT$ is the central product of a $\bT$ and $\bZ(\wbG)$ over $\bZ(\bG)=\spann<h_0>$. The projections of $\wt C=\wbT^{vF}$ onto the components are $$\wh C:=\{x\in\bT\mid (vF)(x)\in x\bZ(\bG)\}\und \wh Z:=\{x\in\bZ(\wbG)\mid (vF)(x)\in x\bZ(\bG)\}.$$ 
	Clearly $\xi$ extends to some $\wt \xi\in \Irr(\wh C)$ and $\wt \xi$ extends to $\wh C \wh Z$. This extended character can be written as $\wh \xi.\lambda$ for some $\wh\xi\in\Irr(\wh C)$ and $\la\in \Irr(\wh Z)$. Since $W_d:=N/C$ acts trivially on $\wh Z$ and normalizes $\wh C$, the stabilizers of $\wt\xi$, $\wh\xi.\la$ and $\wh \xi$ coincide, i.e., $W_{\w\xi}=W_{\wh \xi}$.
\end{rem}

We will use the groups introduced above to give an analogue of Proposition~5.5 of
\cite{CS17C}. 

\paragraph{The group $\wh C$.} 
Like in \cite[\S\S~5,9]{S10a} we see that 
$$\bT=\bT_1 \cdots \bT_a,$$
where $\bT_k:=\Spann<\hh_{e_i}(t)|  i\in \cO_k, t\in \FF^\times>$ and $ \bT_1 \cap (\bT_2 \cdots \bT_a)$ is $\spann<h_0>$, see Lemma 9.3 of \cite{S10a}. 
Hence $\wh C$ is the central product of the groups $\wh T_k=\{ t\in \bT_k \mid\,\, (vF)(t)\in t \spann<h_0> \}$. There exists an isomorphism $\Xi_k: \wh T_k\ra C_0$ that maps $\spann<h_0>$ to $C_{0,2}$. This implies that there exists an isomorphism \begin{equation}\label{whC}
\Xi: \wh C \lra C_0 \times_{C_{0,2}} \cdots \times_{C_{0,2}}  C_0,
\end{equation} 
where the latter group is the $a$-fold central product of $C_0$ over $C_{0,2}$.

\paragraph{The group $C$.}
Let $\nu$ be the linear character of $C_0$ of order $2$ trivial on the first factor in case $2\mid d_0$, and let $\mu\in \Irr(\wh C)$ the character of order $2$ that corresponds to $\nu.\cdots .\nu$ via $\Xi$. Then $\mu$ corresponds to the composite of the map $\wh C \ra \spann<h_0>$ given by $ t \mapsto (vF)(t)t^{-1}$ and the faithful character of $\spann<h_0>$. Hence $$C=\ker(\mu).$$ 

\paragraph{The characters of $C$ and $\wh C$.}
We choose a $\calG$-transversal $\calR$ of $\Irr(C_0)$. For later purposes we choose $\calR$ such that $\zeta\nu\in\calR$ whenever $\zeta\in \calR$ and $\zeta \nu$ lie in different $\calG$-orbits. Since, via $\Xi_k$ the action of $\spann <c_k> \wh T_k/\wh T_k$ and the action of $\calG$ on $C_0$ coincide, every $\wh \xi\in \Irr(\wh C)$ is $W_d$-conjugate to a character of the form 
\begin{align}\label{IrrwhC}
\xi_1. \cdots.\xi_a \text{ with }\xi_k\in\calR 
\end{align}
(more precisely $\xi_k\in\Irr(\wh T_k)$ corresponds to $\calR$ via $\Xi_k$). 

Since $\wh C$ is abelian $\Irr(C)=\{ \restr \wh \xi| C \mid \wh \xi \in \Irr(\wh C)\}$. Because of 
$|\wh C:C|=2$ every $\xi\in \Irr(C)$ has two extensions $\wh \xi$ and $\wh \xi \mu$ to $\wh C$.

\begin{rem}[The relative inertia groups $W_\xi$ and $W_{\wh \xi}$ for $\xi\in\Irr(C)$ and $\wh \xi\in\Irr(\wh C\mid \xi)$]\label{par1112}
	Assume that $\wh \xi\in\Irr(\wh C\mid \xi)$ is as in (\ref{IrrwhC}). Then by the particular choice of $\calR$ we see that via the identification of $\rho(N)$ with $\calG\wr \Sym_a$ the relative inertia group $W_{\wh \xi}=N_{\wh \xi}/C$ satisfies
	$$W_{\wh\xi}= \prod_{\zeta\in\calR} \calG_\zeta \wr \Sym_{I_\zeta} 
	\text{ with } I_\zeta:=\{i\mid \xi_i=\zeta\}$$ 
	for $\zeta\in\calR$ and $\wh \xi=\xi_1. \cdots . \xi_a$ in (\ref{whC}). 
	Recall that because of $W_{\w\xi}=W_{\wh \xi}$ this also determines $W_{\w \xi}$.
	
	For $\xi:=\restr \wh \xi|C$ the group $W_\xi$ has $W_{\wh \xi}$ as normal subgroup of index at most $2$. For $n\in W_\xi\setminus W_{\wh \xi}$ we have $\wh \xi^ n=\wh \xi \mu$. Such an element exists if and only if for every $\zeta\in\calR$ 
	the equation $|I_\zeta|=|I_{\zeta\nu}|$ holds, whenever $\zeta\nu\in \calR$. 
\end{rem}
\medskip

For the proofs in the non-regular case it is useful to observe the following inclusions. 
Recall $N_c:=\bN_c^{vF}$.
\begin{lem}\label{lemdodd}
	Assume $2\nmid d$. Let $\zeta\in \Irr(C)$ with $\zeta(h_0)=-1$. Then $N_{\zeta} \leq N_c$.\end{lem}
\begin{proof}
	Recall that we assume $8|(q-1)$ following Lemma \ref{simplifications}. This implies that $h_k\in C$ whenever $2\nmid d$. Any $x\in N\setminus N_c$  satisfies $[(\prod_{i=1}^l \hh_{e_i}(\varpi)),x]= h_0$ by the Chevalley relations. Since $\zeta(h_0)=-1$ this implies the inclusion $N_\zeta\leq N_c$.
\end{proof}

\begin{prop}\label{lemWxi}
	Assume $2\mid d$. Let $\zeta\in\Irr(C)$, and $\wh \zeta\in\Irr(\wh C\mid \zeta)$. 
	\begin{enumerate}
		\item \label{prop108a_2}
		If $\zeta(h_0)=-1$ and $N_{\zeta}\not \leq N_c$ , then $2\nmid a$ \label{lemWxia}
		and $N_{\zeta}\cap N_c=N_{\wh \zeta}$. 
		\item \label{lemWxib}\label{eqlemWxi}
		If $\zeta(h_0)=1$ and $N_{\zeta} \neq N_{\wh \zeta}$, then $a$ is even. 
	\end{enumerate}
\end{prop}
\begin{proof}
	In (a),  $\la:=\restr \zeta|{H_d}$ satisfies $N_\zeta\leq V_{d,\lambda}$ and hence $V_{d,\lambda}\not\leq N_c$. By Proposition \ref{prop108a} this implies $2\mid a$.
	
	For proving the required equation $N_{\zeta}\cap N_c=N_{\wh \zeta}$ stated in (a) we compute for more general $\zeta$ the groups $N_\zeta$ and $N_{\wh \zeta}$. 
	
	If $2\nmid o(\zeta)$, then $\wh \zeta$ can be chosen to satisfy $2\nmid o(\wh \zeta)$. The other extension of $\zeta$ to $\wh C$ has even order. Hence $\wh \zeta$ is $N_\zeta$-invariant. This implies $W_{\wh \zeta}=W_\zeta$ in this case. 
	
	Assume now that $\zeta\in\Irr(C)$ satisfies $\zeta(h_0)=-1$ and that $o(\zeta)$ is a power of $2$. Note that  $|C|= (q^{d_0}+1)^a$ because of $2\mid d$. In this case $H_d$ and $\wh H_d$ are the Sylow $2$-subgroups of $C$ and $\wh C$ since we assume $4\mid (q-1)$ by Lemma \ref{simplifications}. The characters $\lambda=\restr \zeta| {H_d}$ and $\wh \la:=\restr \wh \zeta| {\wh H_d}$ satisfy 
	\begin{equation}\label{eqneq}
	CV_{d,\lambda}= N_{\zeta} \und CV_{d,\lambda_1}= N_{\wh \zeta}. 
	\end{equation}
	The equality $V_{d,\la}\cap N_c=V_{\wh\la}$ from Proposition \ref{prop108a}
	implies 
	\begin{equation}
	N_{\wh \zeta}=CV_{d,\wh \la}= C (V_{d,\la}\cap N_c)=CV_{d,\la}\cap N_c=N_{\zeta}\cap N_c.
	\end{equation}
	
	Any $\zeta\in\Irr(C)$ with $\zeta(h_0)=-1$ can be written uniquely as a product $\zeta_2 \zeta_{2'}$ where $\zeta_2,\zeta_{2'}\in\Irr(C)$ such that the order of $\zeta_2$ is a power of $2$ and $2\nmid o(\zeta_{2'})$. We denote by $\wh \zeta_2,\wh\zeta_{2'}\in\Irr(\wh C)$ extensions of $\zeta_{2}$ and $\zeta_{2'}$ with $2\nmid o(\wh \zeta_{2'})$. Then $N_{\zeta}=N_{\zeta_2} \cap N_{\zeta_{2'}}$ and $N_{\wh \zeta}=N_{\wh \zeta_2} \cap N_{\wh \zeta_{2'}}$. By the above considerations we get $N_{\zeta_{2'}}=N_{\wh\zeta_{2'}}$ and 
	\begin{equation}
	N_{\wh \zeta_2}=N_{\zeta_2}\cap N_c.
	\end{equation}
	Together this implies the required equality $N_{\wh \zeta}=N_{\zeta}\cap N_c$.
	
	For the proof of part (b) let $\zeta\in\Irr(C)$ and $\wh \zeta\in \Irr(\wh C\mid \zeta)$ with $\zeta(h_0)=1$ and $N_\zeta\neq N_{\wh \zeta}$. Again $\zeta=\zeta_2 \zeta_{2'}$ where $\zeta_2,\zeta_{2'}\in\Irr(C)$ such that the order of $\zeta_2$ is a power of $2$ and $2\nmid o(\zeta_{2'})$. We denote again by $\wh \zeta_2,\wh\zeta_{2'}\in\Irr(\wh C)$ extensions of $\zeta_{2}$ and $\zeta_{2'}$ with $2\nmid o(\wh \zeta_{2'})$. Then $N_{\zeta}=N_{\zeta_2} \cap N_{\zeta_{2'}}$ and $N_{\wh \zeta}=N_{\wh \zeta_2} \cap N_{\wh \zeta_{2'}}$. By the above considerations $N_{\zeta_{2'}}=N_{\wh\zeta_{2'}}$. Hence $N_\zeta\neq N_{\wh \zeta}$ implies 
	$V_{d,\wh \la}\neq  V_{d,\la}$ where $\lambda=\restr \zeta| {H_d}$ and $\wh \la:=\restr \wh \zeta| {\wh H_d}$ following the equations \eqref{eqneq}. Then Proposition \ref{prop108b} implies $2\mid a$ as required. 
\end{proof}
For later use we add another comment on the value of a certain commutator. 
\begin{rem}\label{remNzetaNwtzeta}
	For $\wh C$ from Remark \ref{remwhC}, every $t\in \wh C$ and $n\in N$ the commutator satisfies $[t,n]\in C$. If $\zeta\in\Irr(C)$, $\wh \zeta\in\Irr(\wh C\mid \zeta)$ and $n\in N_\zeta$, then the value of $\zeta([t,n])$ is given by 
	\begin{equation*}
	\zeta([t,n])=\begin{cases}\,\,\,\,	1& \text{if } n\in N_{\wh \zeta},\\
	-1& \otw.	\end{cases} \end{equation*}
	This follows from straightforward calculations.
\end{rem}


For verifying assumption (iv) of Theorem~\ref{thm_4.4} about the Clifford theory for $W_{\wh \xi}\unlhd \NNN_{W_d}(W_{\wh \xi})$ we prove a more general
statement on extendibility. 
\begin{prop} \label{prop_maxext_wreathprod}
	Let $A$ be an abelian group and $n\geq 1$. Let $X_0\leq A^n$ such that $X_0$ is the direct product of its $n$ projections.
	Let $K$ be a Young subgroup $K$ of $\Sym_n$ normalizing $X_0$ in the action of $\Sym_n$ on $A^n$ by permutation. Denote $X:=X_0\rtimes K\leq A\wr \Sym_n$. Then maximal extendibility holds with respect to 
	$$X\unlhd \NNN_{A\wr \Sym_n}(X).$$
	
\end{prop}
\begin{proof} Without loss of generality we assume $X_0=A_1^{n_1}\times\cdots  \times A_r^{n_r}$ for pairwise distinct subgroups $A_i\leq A$ with $n_i\geq 1$ and $n_1+\cdots  +{n_r}=n$. The group $K$ also splits analogously as a direct product of Young subgroups $K_i$ of $\Sym_{n_i}$. Then $\NNN_{A\wr \Sym_n}(X)$ is the direct product of groups $\NNN_{
		A\wr \Sym_{n_i}}(X)=\NNN_{A\wr \Sym_{n_i}}(A_i\wr K_i)$. So our claim clearly reduces to $r=1$, which we now assume.
	
	Group-theoretic considerations prove that $((a_1,\ldots,a_n),\si)$ is an element of $\NNN_{A\wr \Sym_n}(X)$ if and only if $\si\in\NNN_{\Sym_n}(K)$ and $ a_iA_1=a_jA_1$ whenever $i,j$ are in the same $K$-orbit. 
	Let $\chi\in \Irr(X)$. 
	By the action of the elements of $\NNN_{A\wr \Sym_n}(X)$ on $X$ we see that $\NNN_{A\wr \Sym_n}(X)_\chi \leq A^n\rtimes \NNN_{\Sym_n}(K)_\chi$. 
	Now we construct an extension of $\chi$ to $A^n\rtimes \NNN_{\Sym_n}(K)_\chi$ , although the latter may not normalize $X$. 
	
	Let $\la\in\Irr({A_1^n}\mid \chi)$ and $\la_i\in\Irr(A_1)$ such that $\la=\la_1\times \cdots \times \la_n$. There exists some extension $\wh\la\in\Irr(A^n)$ of $\la$ such that $(\Sym_n)_{\la}=(\Sym_n)_{\wh\la}$ , since $A$ is abelian. It is a well-known fact that $\la$ has a unique extension $\phi\in\Irr((A_1)^n\rtimes	K_{\la})$ such that $\restr\phi|{K_\la}$ has only non-negative integral values, see for instance \cite[33.1]{B06}. Analogously $\wh \la$ has an extension $\wh\phi\in\Irr(A^n\rtimes (\Sym_n)_{\la})$ such that $\restr\wh\phi|{K_\la}$ has only non-negative integral values. Because of the uniqueness, $\wh \phi$ is an extension of $\phi$. 
	
	By Clifford theory $\chi=(\phi\eta)^{X}$ for some $\eta\in\Irr(K_\la)$. By the uniqueness of $\phi$ we see that $\NNN_{\Sym_n}(K)_\chi=K \NNN_{\Sym_n}(K,K_\la)_{\phi\eta}=K \NNN_{\Sym_n}(K,K_\la)_{\eta} $. 
	The character $(\wh \phi \eta)^{A^n X}\in \Irr(A^n X)$ is an extension
	of $\chi$. 
	
	The group $K_\la$ is a Young subgroup and $\wh K_1:= \NNN_{\Sym_n}(K,K_\la)_\eta$ is
	a direct product of wreath products with summands of $K_\la$ as base subgroups. Hence $\eta$ has an extension
	$\w\eta\in\Irr(\wh K_1)$. Then the character $(\wh\phi_{A^n\rtimes\wh K}
	\w\eta)^{A^n\rtimes (\wh K_1 K)}$ is an extension of $\chi$. Since
	$\NNN_{A^n\rtimes \Sym_n}(X)_\chi\leq A^n\rtimes \NNN_{\Sym_n}(K)_\chi=A^n\rtimes (\wh K_1 K)$ this
	proves our claim.
\end{proof}

The character theory for the groups $W_\xi$ and $W_{\w\xi }$ satisfies the following. 
\begin{prop}\label{prop_Wxi_regB}
	For every $\xi\in\Irr(C)$, $\wh \xi\in\Irr(\w C\mid \xi)$, $K:=\NNN_{W_d}(W_\xi ,W_{\w\xi })$ and $\eta_0\in\Irr(W_{\w\xi })$, any $\eta\in\Irr(W_\xi \mid \eta_0)$ is $K_{\eta_0}$-invariant.
\end{prop}
\begin{proof}
	The groups $W_d$ is a wreath product of the form $\calG\wr\Sym_a$ while
	the structure of  $W_{\wt \xi}$ is given in Remark \ref{par1112}. Since $W_{\wt \xi}$ satisfies the requirements for $X$ in Proposition~\ref{prop_maxext_wreathprod}, any $\eta_0\in \Irr(W_{\wt \xi})$ extends to its inertia group in $\NNN_{W_d}(W_{\wt \xi})$. In particular there exists some extension to $W_{\xi,\eta_0}$ that extends to $K_{\eta_0}$. Since $|W_\xi:W_{\wh \xi}|\leq 2$ all characters in $\Irr(W_\xi\mid \eta_0)$ are $K_{\eta_0}$-invariant. 	
\end{proof}

\begin{cor}\label{Adreg}
Properties \AA d and \BB d from Definition~\ref{A(d)} hold for $\GG_{\mathrm{sc}}^F$ of type $\tB_l$ when the centralizer of the Sylow $d$-torus is a torus and $q$ is a square.
\end{cor}
\begin{proof} We apply Theorem~\ref{thm_4.4}. Assumptions (i) and (ii) are given by Lemma~\ref{lemrhoVd_B}. Assumptions (iii) and (iv) are ensured by Theorem~\ref{vgoodB} and the above Proposition~\ref{prop_Wxi_regB}, knowing that $E$ is cyclic in this type and acts trivially on $W_d$. This gives our claim. 
\end{proof}

\subsection{Considerations for the non-regular case}
In the next step we extend the above results to arbitrary integers $d\geq 1$ and complete the proof of properties \AA d and \BB d from Definition~\ref{A(d)} by checking all assumptions of Theorem~\ref{thm_loc_gen} above.

We keep the same notation (see Definition~\ref{VhatE}).
Recall that $m$ is the integer such that $p^m=q$, $e=|V|$ and $\wh E=\Cy_{em}$. We define an action of $\wh E$ on $\GG_{\mathrm{sc}}^{F_0^{em}}$ by letting the generating element $\wh F_0$ act on 
$\GG_{\mathrm{sc}}^{F_0^{em}}$ by $F_0$. Let $\wh F:=\wh F_0^m$.

Again we will fix below an element $v\in V$. We then denote by $\bS$ a Sylow $d$-torus of $(\bT, vF)$, its centralizer $\bC= \Cent_{\bG}(\bS)$ and the finite groups
\begin{align}
\begin{split}
C:= \bC^{vF}=\Cent_{\bG}(\bS)^{vF}, &\quad \wt C:= \Cent_{\wbG}(\bS)^{vF}, \\
N:= \NNN_{\bG}(\bS)^{vF},& \quad \wt N:= \NNN_{\wbG}(\bS)^{vF}, \quad \text{ and } \quad  \wh N:=\Cent_{(\bG \wh E)^{F_0^{em}}}(v\wh F)_\bS
\end{split}
\label{defCNwhN} \end{align}
Recall that the assumptions of Theorem~\ref{thm_loc_gen} essentially require to verify the following:
\begin{asslist}
	\item $\bS$ is also a Sylow $d$-torus of $(\bG,vF)$;
	\item a $\w C$-transversal $\calT\subseteq \Irr(C)$ with certain properties can be chosen;
	\item there exists some $\wh N$-equivariant extension map with respect to $C\unlhd N$ for $\calT$;
	\item for every $\xi\in\calT$ and $\wt\xi\in\Irr(\wt C\mid \xi)$ every character $\eta_0\in\Irr(W_{\wt\xi})$ has a $K_{\eta_0}$-invariant extension to $W_{\xi,\eta_0}$ , where $W_\xi:=N_\xi/C$, $W_{\wt \xi}:=N_{\wt \xi}/C$, $W_d:=N/C$ and $K:=\NNN_{W_d}(W_\xi,W_{\wt \xi})$. 
\end{asslist}

 The element $v\in V$ is chosen as follows. The positive integer $d$ determines $d_0$ as in \eqref{defd0}. Let $l'$ be the maximal multiple of $d_0$ with $l'\leq l$. For $a:=\frac {l'}{d_0}$ let 
$$v :=\big(\n_{\al_1}(1)\cdots \n_{\al_{l'}}(1) \big)^ a \in V.$$ 
According to \cite[3.2]{S10b} the Sylow $d$-torus $\bS$ of $(\bT,vF)$ is one of $(\bG,vF)$ and hence assumption (i) of Theorem~\ref{thm_loc_gen} holds.

Since we study the non-regular case, we assume that $l'<l$.

In order to describe the structure of the groups from \eqref{defCNwhN}
we use the following root systems and groups: Let 
\begin{equation}\label{defPhi1}
\Phi_1:=\Phi\cap \Spann<e_i| 1\leq i\leq l'>\und \Phi_2:=\Phi\cap \Spann<e_i| l'< i\leq l>.
\end{equation}
Additionally let 
\begin{equation}\label{T1G2}
\bT_1:=\Spann<\hh_{\al}(k)| \al\in\Phi_1, k \in \FF^\times > \und \GG_2:=\Spann<\xx_\al(t)| \al\in\Phi_2 , t\in\FF >.
\end{equation} 

Then, see \cite[2.2]{S10b}, $\bC =\bT_1\GG_2$\ , a central product over $ \Z(\GG_2)=\Z(\bG)=\spann<h_0>$. This last point can be seen by the fact that $\GG_2$ is a semi-simple group of type $\tB_{l-l'}$ with center generated by $\h_{e_{l'+1}}(-1)$ but this equals $h_0=\h_{e_{1}}(-1)$ since they are conjugate but the latter is central.  

\begin{lem}
	\begin{enumerate}
		\item \label{centC} $\Cent_{\wt C}(C) C=\wt C$.
		\item Any $t\in C\setminus (T_1G_2)$ induces a diagonal automorphism on $G_2$. 
		\item Every $\xi\in\Irr(C)$ satisfies $$\wt N_\xi=\wt C _\xi N_\xi \quad \und \quad (\wt N \wh N)_{\xi^N}=\wt N _{\xi^N}\wh N_{\xi^N},$$
		in particular $\calT:=\Irr(C)$ satisfies Assumption (ii) of Theorem~\ref{thm_loc_gen}.
	\end{enumerate}
\end{lem}
\begin{proof} (a) Since $\Z(\w \bC)^{vF}\leq \cent{\w C}{C}$, it suffices to prove that $\Z(\w \bC)^{vF} \bC^{vF}=\w \bC^{vF}$. We have $\w\bC =\Z(\w \bC)\bC$, so by a classical consequence of Lang's theorem, it suffices to prove that $\Z(\w \bC)\cap\bC$ is connected. We have indeed $\Z(\w \bC)\cap\bC = \Z(\bC)= \bT_1$ since $\bC=\bT_1\GG_2$ is a central product where $\Z(\GG_2)=\spann<h_0>\leq \bT_1$ as recalled above.
	
Part (b) is clear from the fact that $\GG_2={[\bC ,\bC]}$.

 (c)	Part (a) implies $\wt C_\xi=\wt C$ for every $\xi\in\Irr(C)$. 
	Since $\wt N=\wt C N$ (see proof of Theorem~\ref{thm_4.4}) we see that $\wt N_\xi=\wt C N_\xi$. By part (a), $\wt N$ stabilizes $\xi^N$ and hence 
	\begin{equation*}
	(\wt N \wh N)_{\xi^N}=\wt N \wh N_{\xi^N}=\wt N _{\xi^N}\wh N_{\xi^N}.  
	\end{equation*}	
	All requirements in (ii) of Theorem~\ref{thm_loc_gen} are then satisfied for $\calT =\Irr(C)$.
\end{proof}

We keep $d$ a non-regular number of $(\bG,F)$ and recall from Lemma~\ref{simplifications} that in order to check condition \AA d, one may assume the following. \medskip

\centerline{
	\textbf{Assumption:}  $\Out(\GG_{\mathrm{sc}}^F)$ is non-cyclic, in particular $2\mid |E|$ and hence $8\mid~(q-1)$. }
 
 \medskip

 Before going into the character theoretic details we have to clarify the structure of $C$ and $N$.

\paragraph{A Sylow $d$-torus $\bS$ and associated groups.}

 We denote by $W_{\tD_{l'}}$ the subgroup of $W$ that is generated by the reflections along the long roots of $\Phi_1$. Then $|W_{\Phi_1}:W_{\tD_{l'}}|=2$ and $\rho(W_{\tD_{l'}})$ is the subgroup of $\Sym_{\pm l'}$ containing all permutations of $\Sym_{\pm l'}$ with an even number of positive integers mapped to negative ones, see also Notation~\ref{n_typB}. 

The next proposition focuses on the structure of the groups involved in assumption (iii) of Theorem~\ref{thm_loc_gen}. We use in a decisive way condition \Ainfty\ for $G_2$ , a group of type $\tB$ of rank smaller than the one of $\bG$, which holds thanks to Theorem~\ref{gloBCE}. 

\begin{prop}\label{prop_1013}
	For the root systems $\Phi_1$ from \eqref{defPhi1} let 
	\begin{equation}\label{defN1Nc}
	\bN_1:=\Spann<\nn_{\al}(k)| \al\in\Phi_1, k \in \FF^\times >,\,
	\bN_c:= \rho^{-1}(W_{\tD_{l'}})\cap \bN_1,\,
	N_1:=\bN_1^{vF} \und 
	N_c:= \bN_c^{vF}.
	\end{equation}
	Then these groups satisfy the following:
	\begin{enumerate}
		\item \label{prop_strucNc}
		$N=N_1C$ and $N_c=\cent{N_1}{\bGzwei}$.
		\item \label{prop103b}
		Let $x_1\in N_1\setminus N_c$ , $\varpi\in\FF$ a primitive fourth root of unity, $h:=\prod_{i=l'+1}^l \hh_{e_i} (\varpi)\in (\bT\cap\bGzwei)^F$.
		Then \begin{equation}\label{1017b}
		\cent{N}{G_2}=\spann<N_c,x_1h>\und  N= C \cent N {G_2}.
		\end{equation}
		\item \label{prop103c}
		Let $G_2:=\GG_2^{vF}$. Then $G_2=\GG_2^{h'F}$ where 
		\begin{equation}\label{defh'}
		h':=\begin{cases}1_{\bGzwei} & \text{ if } v\in N_c,\\
		h & \otw.\end{cases}.
		\end{equation}
		\item Let $g\in \bT\cap\bGzwei$ with $gF(g)^{-1}=h'$. Then  
		$$ \iota_2: \GG_2^F \rtimes \spann<\wh F_0> \rightarrow \GG_2^{h'F} {\spann <t''\wh F_0>}\text{ with }y \mapsto {}^{g^{}}y$$ is an isomorphism
		with $\iota_2(\wh F)=h'\wh F$, where $t'':=gF_0(g)^{-1}\in \bT\cap \bGzwei$ normalizes $\GG_2^{h'F}$. 
		\item 
		Let $T_1:=\bT_1^{vF}$ and $C_0:=T_1G_2$. There exists some $t'\in C$ such that $t't''\in \cent{\bT}{G_2}$. It satisfies 		
		\begin{equation}\label{deft'}
		t' \in \begin{cases}
		C_0 & \text{if } [\wh F_0,h']=1_\bGzwei,\\
		C\setminus C_0 &\otw .\end{cases}
		\end{equation}
		\item \label{existTransv}
		Let $\theta \in \Irr(G_2)$ then 
		$$(N\wh E)_\theta= C_{\theta} \left (N_1\spann< t' \wh F_0>\right )_{\theta}=C_{\theta} N_1\spann< t' \wh F_0>_{\theta},$$ 
		for $t'\in C$ defined above.
	\end{enumerate}
\end{prop}
\begin{proof}
	Parts (a) and (b) follow from straightforward computations, see \cite[10.2.5, 10.2.6]{S07} and \cite[Lem.~6.3]{S10b}. Note that the hypothesis $4\mid (q-1)$ implies $h\in \GG_2^{vF}$. 
	
	Part (c) is then a direct consequence of part (a).
	
	The existence of $g$ is clear from the connectedness of $\bT\cap\GG_2$. The map $\iota_2$ is clearly an isomorphism with $^g(\GG_2^F)=\GG_2^{gF(g^{-1})F}$ and $\iota_2(\wh F_0)=[g,\wh F_0]\wh F_0$.
	
Note that $F_0(h)h^{-1}=\prod_{i=l'+1}^l \hh_{e_i} ((-1)^{p-1\over 2})\in \Z(\bG)$ and therefore 
 $z:= F_0(h')h'{}^{-1} \in \Z(\bG)$. Denote $s'':=g^{-1}  F_0(g)=\iota_2^{-1}(t'')$.
	We have 
	$F(s'')s''^{-1}=(F(t''^{h'})t''^{-1})^{-1}$ and
	\begin{align*}
	F(s'')s''^{-1}&= F(g^{-1} ) F_0( F(g) ) F_0(g)^{-1} g\\
	&=F(g^{-1}) F_0(h'^{-1}) g=  F(g^{-1}) z h'^{-1} g=
	z F(g)^{-1} (F(g) g^{-1}) g=z.
	\end{align*}
	Accordingly $s''$ induces a diagonal automorphism on $\GG_2^F$ corresponding to $z$ by the relevant Lang map, and $t''$ induces a diagonal automorphism on $\GG_2^{h'F}$ corresponding to $z$.
	
For (e), observe that since $\bC $ has connected center (see proof of Proposition~\ref{prop_1013}(a)), $\bC^{vF}$, and therefore $\bT^{vF}$ induces all diagonal automorphisms of $\GG_2^{vF}=[\bC,\bC]^{vF}$. This ensures the existence of $t'$. The  claim that $t'\in C_0$ if and only if $z=1$ comes from the fact just proved that $t''$, hence $t'$, induces a diagonal outer automorphism on $\GG_2^{h'F}$ corresponding to $z$, while $t'\in C_0$ if and only if the induced automorphism is interior.

	Concerning (f), one has $N\wh E=C\cent{N}{G_2}\wh E$ by (\ref{1017b}). We can content ourselves with checking $(C \spann<t'\wh F_0>)_\theta =C_\theta\spann<t'\wh F_0>_\theta $. Applying $\iota_2^{-1}$ this also reads $$(C^g \spann<\wh F_0>)_{\theta '}= C^g_{\theta '} \spann<\wh F_0>_{\theta '}$$
for any $\theta '\in\Irr(\GG_2^F)$. Observe that $\Phi_2$ is a root system of type $\tB_{l-l'}$ if $l-l'\geq 2$, and of type $\tA_1$ otherwise. The group $\GG_2^{F}$ is a quasi-simple group of type $\tB_{l-l'}$ or isomorphic to $\SL_2(q)$, respectively, on which $F_0$ and $F$ act as described in the setting of Sect.\,2.A. On that group, $C^g$ acts by diagonal automorphisms, since $C$ does so on $G_2={}^g\GG_2^{F}$. Our statement then follows from Theorem~\ref{gloBCE} and \cite[Thm.~4.1]{CS17A}.
\end{proof}

The next proposition ensures assumption~(iii) of Theorem~\ref{thm_loc_gen}. Recall the definitions of $C$, $N$, and $\wh N$ given in \eqref{defCNwhN}.
\begin{prop}\label{propextmap} 
	There exists an $\wh N$-equivariant extension map for $C\unlhd N$. 
\end{prop}

\begin{lem}\label{lem_extmap_Bnonreg_neu}
	Let $\wh F_0'\in \wh E$, $n\in N_1$ and $\zeta\in\Irr(T_1)$ such that $\zeta^{n\wh F_0'}=\zeta$. Then $\zeta$ has an $n\wh F_0'$-invariant extension $\wt \zeta\in\Irr(N_{1,\zeta})$. 
	In particular $\wt\zeta([n\wh F_0',y])=1$ for every $y\in N_{1,\zeta}$. 
\end{lem}
\begin{proof}
	Note that $\bS$ is a Sylow $d$-torus of $(\mathbf G_1,vF)$ with $\bT_1=\cent{\mathbf G_1}{\bS}$ and $N_1:=\norm{\mathbf G_1}{\bS}$, see \cite[Lem.~2.2]{S10b}. In this case $d$ is a regular number for $(\mathbf G_1,vF)$. By the proof of Theorem 4.4 in \cite{CS17C} together with Theorem \ref{vgoodB} we see that there exists an $N_1\rtimes \wh E$-equivariant extension map $\Lambda_1$ with respect to $T_1\unlhd N_1$. This proves that $\zeta$ has an $\spann< n \widehat F'_0 >$-invariant extension $\wt\zeta\in\Irr(N_{1,\zeta})$.
	Since $\wt \zeta$ is a linear character this implies $\wt\zeta([n\wh F_0',y])=1$. 
\end{proof}

\begin{proof}[Proof of Proposition \ref{propextmap}] For the proof it is sufficient to construct for every $\xi\in \Irr(C)$ some $(N\wh E)_\xi$-invariant extension $\wt \xi$ to $N_\xi$. 
	
	Recall $T_1:=\bT_1^{vF}$, $G_2:=\GG_2^{vF}$, and $C_0:=T_1G_2$. 
	
	Let $\xi\in\Irr(C)$. Since $C_0$ is the central product of $T_1$ and $G_2$ any character in $\Irr(C_0\mid \xi)$ can be written as  $\zeta.\theta \in\Irr(C_0)$ with $\zeta\in\Irr(T_1)$ and $\theta\in\Irr(G_2)$. Let $\xi_0\in \Irr(C_{\zeta.\theta} \mid \zeta.\theta)$ with $\xi_0^C=\xi$. Since $C_0\unlhd N\wh E$ we see that $(N\wh E)_{\xi}= C (N\wh E)_{\xi_0}\leq C (N\wh E)_{\theta} $.
	By Proposition~\ref{prop103c} the stabilizer $(N\wh E)_{\xi_0}$ satisfies $ N_{\xi_0}= C_{\zeta .\theta} \Cent_N (G_2)_{\xi_0}$. This group has the normal subgroup $C_0 \Cent_N (G_2)_{\xi_0}$ of index $\leq 2$. 
	
	Assume there exists some extension  $\wt \psi\in \Irr(\cent N {G_2}_{\xi_0})$ of $\zeta$. Then $\wt \psi.\theta\in\Irr(\Cent_N (G_2)_{\xi_0} G_2)$ is a well-defined extension of $\zeta.\theta$. According to \cite[Lem.~4.1]{S10b} there exists a common extension $\wt \xi_0$ of $\wt\psi.\theta$ and $\xi_0$. The character $\w \xi:=\wt\xi_0^{N_\xi}$ is then an extension of $\xi$. 
	This extension is $(N\wh E)_\xi$-invariant, if $\wt \xi_0$ is $(N\wh E)_{\xi_0}$-invariant or equivalently $\wt\psi\in\Irr(\cent{N}{G_2}_{\xi_0})$ is $(N\wh E)_{\xi_0}$-invariant. Hence for the proof it is sufficient to construct some $(N\wh E)_{\xi_0}$-invariant extension $\wt \psi$ of $\zeta$ to $\cent{N}{G_2}_{\xi_0}$.
	
	Let $h'$ and $t'$  be as in \eqref{defh'} and \eqref{deft'}. 
	Then Proposition~\ref{existTransv} implies 
	$$(N\wh E)_\theta = C_\theta \,\, (N_1,\spann<t'\wh F_0>)_{\theta}.$$ 
	Let $n\in N_1$, $t\wh F_0'\in \spann<t' \wh F_0>$ such that 
	$$(\xi_0)^{nt\wh F_0'}=\xi_0.$$
	
	In order to ensure that an extension $\wt \psi$ of $\zeta$ to $\cent{N}{G_2}_{\xi_0}$ is $(N\wh E)_{\xi_0}$-invariant it is hence sufficient to verify that 
	\begin{equation}\label{equivwtpsi}
	\wt \psi^{nt\wh F'_0}=\wt \psi. 
	\end{equation}
	
	Let $\wt \zeta$ be the extension of $\zeta$ to $N_{1,\zeta}$ from Lemma \ref{lem_extmap_Bnonreg_neu} and $\psi:=\restr \wt \zeta|{\cent{N_1}{G_2}}$.
	Let $\mathrm i\in\CC $ be a primitive fourth root of unity. 
	If $N_{1,\zeta}\not \leq N_c$ there exists some $x\in N_{1,\zeta}\setminus N_c$ , $\cent{N}{G_2}_{\zeta.\theta}= \spann<N_{c,\zeta},xh>$, and an extension $\wt \psi$ of $\psi$ to $\cent{N_1}{G_2}_{\zeta}$ is determined by 
	$$ \wt \psi(xh):=\begin{cases}
	\zeta(x) & \text{if } h_0\in\ker(\zeta),\\
	\zeta(x) \mathrm i & \otw. 
	\end{cases}$$
	Note that $\wt\psi((xh)^2)=\wt \psi(x^2)\wt \psi(h_0)$ and hence $\wt \psi$ is a well-defined linear character.
	
	It remains to verify Equation \eqref{equivwtpsi} for $\wt\psi$. 
	According to Lemma \ref{lem_extmap_Bnonreg_neu} we have $ \psi^{t\wh F'_0n}= \psi$.
	This proves  \eqref{equivwtpsi} whenever $N_{1,\zeta}\leq N_c$. 
	In the following we assume $N_{1,\zeta}\not\leq N_c$. Then Equation \eqref{equivwtpsi} holds whenever
	$$\wt\psi([t \wh F_0' n , xh ])=1
	\text{ or equivalently }
	\wt\psi([ \wh F_0' , h ]) \wt \psi([t,x ])=1, $$ 
	since $\wt \psi$ is a linear character. The verification of this statement requires to consider various cases. 
	For those considerations one uses the abelian group $\wh T_1:=\{ t\in \bT_1\mid (vF_0)(t) \in t\spann<h_0> \}\geq T_1$ and choose an extension $\wh \zeta\in\Irr(\wh T_1)$ of $\zeta$. %
	
	First we consider \textit{the case $2\nmid d$}. By Lemma~\ref{lemdodd} we can assume that $\zeta(h_0)=1$ using our assumption $N_{1,\zeta}\not\leq N_c$.  This implies $\wt \psi([\wh F_0',h])=1$. On the other hand because of $2\nmid d$ we have $v\in N_c$ by Lemma~\ref{lem106},  $h'=1_{\bGzwei}$ by \eqref{defh'} and $t'\in C_0$ by  \eqref{deft'}. This implies $t\in C_0$ and $\wt \psi([t,x])=1$. Altogether we have $\wt\psi([ \wh F_0' , h ]) \wt \psi([t,x ])=1$. 
	
	\textit{The next case is $2\mid d$ and $\zeta(h_0)=-1$}. Proposition~\ref{prop108a_2} implies that $a$ is odd and the element $x\in N_{1,\zeta}\setminus N_c$ satisfies 
	\begin{equation}
	\wt \psi([t,x])=\begin{cases}1& t\in C_0,\\
	-1& \otw, \end{cases}
	\end{equation}
	according to Remark~\ref{remNzetaNwtzeta} since $x\notin N_{\wh \zeta}$ according to Proposition~\ref{prop108a_2}.
	Because of $2\nmid a$ we have $v\notin N_c$ by Lemma~\ref{lem106} and $h'=h$ by \eqref{defh'}. This implies  $[\wh F'_0,h]=[\wh F'_0,h']$. 
	
	If $[\wh F'_0,h]=h_0$ we have $\wt \psi ([\wh F_0',h])=-1$,  $t\notin C_0$ by \eqref{deft'} and hence $\wt \psi([t,x])=-1$. 
	If $[\wh F'_0,h]=1_\bGzwei$ we have $\wt \psi ([\wh F_0',h])=1$, $t\in C_0$ by \eqref{deft'} and hence $\wt \psi([t,x])=1$. In both cases $\wt\psi([ \wh F_0' , h ]) \wt \psi([t,x ])=1$.
	
	Now we consider the remaining \textit{case where $2\mid d$ and $\zeta(h_0)=1$}. Accordingly $\wt \psi([\wh F'_0,h])=1$. If $N_{\zeta}=N_{\wh \zeta}$ , we have  $\wt \psi([t,x])=1$, as required. 
	If $N_{\zeta}\neq N_{\wh \zeta}$ , we know $2\mid a$ from Proposition~\ref{lemWxib}, hence $v\in N_c$ by Lemma~\ref{lem106} and $t\in C_0$ by \eqref{defh'} and \eqref{deft'}. Again 
	$\wt \psi([t,x])=1$, and hence $\wt\psi([ \wh F_0' , h ]) \wt \psi([t,x ])=1$. 
	This finishes the proof. 
\end{proof}

It remains to determine the relative inertia groups and analyse their characters. 

\begin{prop}\label{Wxi}
	Let $\xi\in\Irr(C)$ above some $\zeta.\theta$ where $\zeta \in \Irr(T_1)$ and $\theta\in\Irr(G_2)$. Let $\wt \xi\in\Irr(\wt C \mid \xi)$. We write $W_\xi:=N_\xi/C$ and $W_{\wt \xi}:=N_{\wt \xi}/C$.
	Let $\wh T_1:=\{x\in \bT_1 \mid (vF)(x)\in x\spann<h_0>  \}$,  $\wh \zeta\in\Irr(\wh T_1\mid \zeta)$ and $W_{\wh \zeta}:=N_{1,\wh \zeta}/T_1$.  Then:
	\begin{enumerate}
		\item $W_\xi=W_{\wt \xi}$ , if $\restr \xi|{T_1G_2}=\zeta.\theta$. 
		\item $W_\xi=W_{\zeta}$ and $W_{\wt \xi}=W_{\wh \zeta}$ , if $\restr \xi|{T_1G_2}\neq \zeta.\theta$.
	\end{enumerate}
\end{prop}
\begin{proof}
	As before we denote by $\xi_0\in\Irr(C_{\theta.\xi}\mid \xi.\theta)$ the Clifford correspondent of $\xi$. From the proof of Proposition~\ref{propextmap} we know $N_\xi=N_{1,\xi}C=N_{1,\xi_0}C$ with $N_{1,\xi_0}\leq N_{1,\zeta} $ , in particular $N_{1,\xi_0}= N_{1,\zeta} $ , if $\xi_0=\zeta.\theta$ and equivalently $\restr \xi|{T_1G_2}\neq \zeta.\theta$. This implies in particular 
	$W_\xi=W_{\zeta}$ , if $\restr \xi|{T_1G_2}\neq \zeta.\theta$. 
	
	By Proposition~\ref{prop_strucNc}, $\xi$ has an extension $\wh \xi$ to the group $\wh C :=\{ c\in \bC \mid  (vF)(c) \in c\spann<h_0>\}$. By the considerations given in the proof of Proposition~\ref{prop_strucNc} we see that $N_{\wh \xi}=N_{\wt \xi}$. The group $\wh C$ is the central product of the group $\wh T_1:=\{x\in \bT_1 \mid (vF)(x)\in x\spann<h_0>  \}$ and the group $\wh G_2:=\{x\in \GG_2 \mid (vF)(x)\in x\spann<h_0>  \}$. The character $\wh \xi=\wh\zeta.\wh\theta$ with $\wh \zeta\in\Irr(\wh T_1\mid \zeta)$ and $\wh \theta\in\Irr(\wh G_2\mid \theta)$. Since according to Proposition~\ref{prop103b}, $N_1$ induces on $\wh G_2$ and $G_2$ inner automorphisms the character $\wh \theta$ is $N_1$-invariant.
	We see that $N_{\wh \xi}=C N_{1,\wh \xi}$ and $N_{1,\wh \xi}=N_{1,\wh \zeta}$. This implies $W_{\wt \xi}=W_{\wh \zeta}$.
	
	Let us now assume that $\restr \xi|{T_1G_2}=\zeta.\theta$ and hence $\xi_0\neq \zeta.\theta$. Then there exist two extensions $\wh \zeta.\wh \theta$ and $\wh\zeta'.\wh \theta'$ of $\xi$ with $\wh\zeta\neq \wh\zeta'$ and $\wh \theta\neq \wh\theta'$. Any element in $N_{1,\xi}\setminus N_{1,\wh\xi}$ has to map $\wh \zeta.\wh \theta$ to $\wh \zeta'.\wh\theta'$. Since $\wh \theta'$ is $N_1$-stable we see that $N_{1,\xi_0}=N_{1,\wh \zeta}$ in that case. 
	
	If $\restr \xi|{C_0}=\zeta.\theta$, the group $N_{1,\xi_0}$ coincides with $N_{1,\zeta}$ , $W_\xi=W_{\zeta}$ and $W_{\wt \xi}=W_{\wt\zeta}$. 
\end{proof} 

\subsection{Proof of Theorem \ref{thmA} for type $\tB_l$.}

We first finish checking assumption (iv) of Theorem~\ref{thm_loc_gen} in the above non-regular case. For any $\xi\in\Irr(C)$, $\w\xi\in\Irr(\w C\mid\xi)$, the groups $W_\xi$ and $W_{\wt\xi}$ coincide with relative inertia groups occurring in the regular case for $\zeta$ and $\wh \zeta$. Their characters were already studied in Proposition~\ref{prop_Wxi_regB}. It was proved that every $\eta_0\in\Irr(W_{\wt\xi})$ has an $\NNN_{W_d}(W_{\wt\xi}, W_\xi)_{\eta_0}$-invariant extension to $(W_\xi)_{\eta_0}$. Along with Proposition~\ref{Wxi}, this establishes assumption (iv) of Theorem~\ref{thm_loc_gen}. Assumptions (i)-(iii) are given by Lemma~\ref{centC} and Proposition~\ref{propextmap}. Theorem~\ref{thm_loc_gen} then implies conditions \AA d and \BB d in the non-regular case while Corollary~\ref{Adreg} gives them in the regular case. This applies in the setting we have worked with until now, which assumes that $q$ is a square and therefore $E$ has even order. Let us assume now that $E$ has odd order. Then \AA d holds by Lemma~\ref{simplifications}. In order to get \BB d one has to study extendibility of characters in the inclusion  $\w C\unlhd \w N$.

Let $\theta\in\Irr(\wt C)$. Recall that $N_c\leq \Cent_{N_1}(G_2)$. Let $\psi\in\Irr(\restr \theta|{T_1})$.
Let $\Lambda_1$ be an $N_1E$-equivariant extension map \wrt $T_1\unlhd N_1$. Then $\restr {\Lambda_1(\psi)}|{N_{c,\theta}}$ defines an $(N\wh E)_\theta$-invariant extension $\wt \theta$ of $\theta$ to $N_{c,\theta} \wt C$. 

Note that $\wt C N_c$ is a normal subgroup of $\wt C N_1$ with index $2$. Accordingly $|  \wt C N_{1,\theta}:  \wt C N_{c,\theta}|\leq 2$. The group $\spann<vF>$ acts trivially on $\wt CN_1$ by definition. This implies that the group ${N\wh E}_{\theta}/(\spann<v\wh F> N_{\theta} )$ acts on the set of extensions of $\wt \theta$ to $\wt C N_{1,\theta}$. There are at most $2$ extensions of $\wt \theta$ to $\wt C N_{1,\theta}$. Since $E$ is of odd order, this group is again of odd order and there is an extension of $\wt \theta$ to $\wt C N_{1,\theta}$ that is ${N\wh E}_{\theta}$-invariant by Glauberman's lemma \cite[13.8]{Isa}.

We now have \AA d and \BB d in all cases. Since \Ainfty\ is ensured by Theorem~\ref{gloBCE},
 Theorem~\ref{Glo+Loc} then implies that simple groups of type $\tB$ satisfy the inductive McKay condition for all primes. \qed

\section{The local conditions for types $\tE$}\label{sec_locE}

We now prove conditions \AA d and \BB d ($d\geq 1$) from Definition~\ref{A(d)} for $\GG_{\mathrm{sc}}^F$ of types $\tE_6$, $^2\tE_6$ and $\tE_7$. Let us recall the order of $\GG_{\mathrm{sc}}^F$ as a polynomial in $q$, the prime power associated with $F$ (see Sect.\,2.A). We have  $|\GG_{\mathrm{sc}}^F|=q^NP(q)$ with $N\geq 1$, $P$ as follows, and where $\Phi_d$ denotes the $d$-th cyclotomic polynomial (see \cite[Table 24.1]{MT}). For the information about regular numbers, see \cite[Tables 1,2,8]{Springer}.

For  $\GG_{\mathrm{sc}}^F$ of type $\tE_6$ one has $P=\Phi_{1}^{6}\Phi_{2}^{4}\Phi_{3}^{3}\Phi_{4}^{2}\Phi_{5} \Phi_{6}^{2}\Phi_{8} \Phi_{9}  \Phi_{12}  . $  All $d$'s appearing are regular except 5.

For  $\GG_{\mathrm{sc}}^F$ of type $^2\tE_6$ one has $P=\Phi_{1}^{4}\Phi_{2}^{6}\Phi_{3}^{2}\Phi_{4}^{2}\Phi_{6}^{3}\Phi_{8} \Phi_{10}  \Phi_{12}  \Phi_{18}  .$ All $d$'s appearing are regular except 10.

For  $\GG_{\mathrm{sc}}^F$ of type $\tE_7$ one has  $P=\Phi_{1}^{7}\Phi_{2}^{7}\Phi_{3}^{3}\Phi_{4}^{2}\Phi_{5} \Phi_{6}^{3} \Phi_{7} \Phi_{8} \Phi_{9} \Phi_{10}  \Phi_{12} \Phi_{14}  \Phi_{18}  .$ 
All $d$'s appearing are regular except 4, 5, 8, 10, 12.

Our checking of conditions \AA d and \BB d for regular numbers $d$ will use essentially Theorem~\ref{thm_4.4} whose assumptions will be reviewed in the following. 
Since the center of $\GG_{\mathrm{sc}}$ is cyclic one can choose a regular embedding $\bG\leq\w\GG =\bG \Z(\w\GG)$ as in Sect.\,2.A with $\Z(\w\GG)$ a torus of rank 1. Recall that $F_0$ and the graph automorphism in type $\tE_6$ also extend to $\w\GG$.

\subsection{Relative Weyl groups}

Assumption (iv) of Theorem~\ref{thm_4.4} is essentially a problem about characters of complex reflection subgroups of the Weyl groups involved. In our case they are subgroups of the group $G_{26}$ of order $1296$ or $G_8$ of order $96$ (see \cite[Table 3]{BMM93}). The following proposition and its proof are due to G. Malle whom we thank for allowing it to be included in the present paper.  Here, one takes advantage of the knowledge of centralizers of semi-simple elements (see \cite[\S~14.1]{MT} and the references given there). Shephard-Todd's notation for finite complex reflection groups $G(m,p,n)$ and $G_4,\dots ,G_{37}$ is used freely.

One considers the dual groups $\GG^*:=(\bG)^* $, $\w\GG^*$ with corresponding Frobenius $F$ dual to the one of $\bG$. For $d\geq 1$ let $\bS^* \leq  \GG^*$  denote a Sylow $d$-torus of $\GG^* $, with relative Weyl group $W_d^*:=\norm{\GG^*}{\bS^*}^F/\cent{\GG^*}{\bS^*}^F$ in $\GG^*{}^F$. For $s \in  \GG_{\mathrm{ss}}^*{}^F$ centralizing $\bS^*$ we let $W_s$ denote the relative Weyl group of a Sylow $d$-torus of $\w \GG^*$  in $\Cy_{\GG^*} (s)$ and $W_{\w s}$ the relative Weyl group of $\bS^*$ in $\Cy_{\w \GG^*} ({\w s})$, where ${\w s} $ denotes a preimage of $s$ in $\w \GG^*{}^F$. Note that this is isomorphic to the relative Weyl group of $\bS^*$ in $\Cy^\circ_{\GG^*} (s)$ and thus a subgroup of $W_s$ of index $|\Cy_{\GG^*} (s):\Cy^\circ_{\GG^*} (s)|$. Further, let $K_s := N_{W_d^*}(W_s,W_{\w s})$. For $\w \eta  \in \Irr(W_{\w s})$ one considers the following two properties:

\begin{enumerate}[label=($\ddagger$)(\roman*)] %
\item  there exists a $(K_s)_{\w \eta }$-invariant character $\eta  \in  \Irr(W_s\mid \w \eta )$,
\item there exists a character $\eta$  as in (i) which extends to $(K_s)_{\w \eta }$ . 
\end{enumerate}


We denote $G:=\GG_{\mathrm{sc}}^F$ and $G^*=\GG^*{}^F$.

\begin{prop}\label{Gunter}
Let $s \in  G^*$  centralizing a Sylow $d$-torus of $(\GG^*,F) $. \begin{enumerate}[label=(\arabic*)]
\item If $G = \tE_6(q)_{\mathrm{sc}}$  then any $\w \eta  \in \Irr(W_{\w s})$ satisfies ($\ddagger$)(i) and (ii).
\item  If $G = {}^2\tE_6(q)_{\mathrm{sc}}$  then any $\w \eta  \in \Irr(W_{\w s})$ satisfies  ($\ddagger$)(i).
\item If $G = \tE_7(q)_{\mathrm{sc}}$ then any $\w \eta  \in \Irr(W_{\w s}) $ satisfies  ($\ddagger$)(i).
\end{enumerate}
\end{prop}

\noindent{\it Proof of Proposition~\ref{Gunter}}
We start by making some general observations. Both points of property  ($\ddagger$) certainly hold in the following cases:

\begin{enumerate}[label=(\roman*)]
\item when $K_s$ is abelian; 
\item if $W_{\w s} =\{ 1\}$, by taking for $\eta$  the trivial character of $W_s$ ; 
\item if $W_s = (K_s)_{\w \eta}$  by taking any $\eta$  above $\w \eta$ ; and 
\item if $W_{\w s}$ is a direct factor of $K_s$ , since then every character of $W_{\w s}$ is invariant and extends.
\end{enumerate}

Also,  ($\ddagger$)(i) trivially holds if $W_{\w s} = W_s$ with $\eta  = \w \eta $. We deal with the three types of groups in turn. 

First assume that $G = \tE_6(q)_{\mathrm{sc}}$. Here, the relative Weyl groups of the Sylow $d$-tori considered are the primitive complex reflection groups $W_3 = G_{25}$, $W_4 = G_8$ and
$W_6 = G_5$. In Table 1 we have collected the semi-simple elements $s \in  G^*$  (up to conjugacy of centralizers) centralizing some Sylow $d$-torus, together with the groups $W_{\w s}$ , $W_s$ and $K_s$ defined above (but omitting the trivial case $s = 1$ in which $W_{\w s} = W_s = K_s = W_d^*$ and so  ($\ddagger$) trivially holds).

Here, an entry of the form $[n]$ denotes an unspecified group of order $n$. Other notation is similar to the one used in \cite{BMM93}. The third column gives restrictions on $q$ for such elements $s$ to exist.

\begin{table}[h!]
\begin{center}
\label{tab:table1}
\begin{tabular}{|r|c|c|c|ccc|c|} %
\hline
\textrm{No.} & $\Cy_{G^*}(s)$ & $q$& $d$& $W_{\w s}$& $W_{s}$& $K_s$& \textrm{case}\\
\hline
1 & $^3\tD_4(q).\Phi_3.3$ &$\equiv 1(3) $&$3,6 $&$G_4 $&$G_4\times 3 $&$G_4\times 3 $&$ (iii)$\\
2 & $\tA_2(q)^3.3$ &$\equiv 1(3) $&$3 $&$3^3 $&$3\wr 3 $&$3\wr S_3 $&$ $\\
3 & $\Phi_3^3.3$ &$\equiv 1(3) $&$3 $&$1 $&$ 3 $&$G_{25} $&$ (ii)$\\
\hline
4  & $\tA_2(q)^2.\Phi_3$ &$  $&$3 $&$  $&$3^2 $&$3^2.2 $&$ $\\
5 & $^3\tD_4(q).\Phi_3$ &$  $&$3,6 $&$  $&$G_4$&$G_4\times 3 $&$(iv) $\\
6 & $\tA_2(q).\Phi_3^2$ &$  $&$3 $&$  $&$3 $&$3^2.2 $&$ $\\
7 & $\Phi_3^3$ &$  $&$3 $&$ $&$1 $&$G_{25} $&$(ii) $\\
\hline
8 & $\tD_4(q).\Phi_1^2.3$ &$\equiv 1(3) $&$4 $&$G(4,2,2) $&$G(4,2,2).3 $&$G_{8} $&$ $\\
\hline
9 & $\tD_5(q).\Phi_1$ &$  $&$4 $&$  $&$G(4,1,2) $&$G(4,1,2) $&$(iii) $\\
10 & $\tD_4(q).\Phi_1^2$ &$ $&$4 $&$  $&$G(4,2,2) $&$G_{8} $&$ $\\
11 & $\tA_1(q^2)^2.\Phi_1^2$ &$\mathrm{odd} $&$4 $&$  $&$2\times 2 $&$4\times 4 $&$(i) $\\
12 & $^2\tA_3(q).\tA_1(q^2).\Phi_1$ &$\mathrm{odd}  $&$4 $&$  $&$4\times 2  $&$4\times 4 $&$(i) $\\
13 & $^2\tA_3(q).\Phi_1.\Phi_4$ &$  $&$4 $&$  $&$4 $&$4\times 4 $&$(i) $\\
14 & $\tA_1(q^2).\Phi_1^2.\Phi_4$ &$  $&$4 $&$  $&$2 $&$4\times 4 $&$(i) $\\
15 & $\Phi_1^2.\Phi_4^2$ &$  $&$4 $&$  $&$1 $&$G_{8} $&$(ii) $\\ \hline
16 & $\Phi_3.\Phi_6^2.3$ &$\equiv 1(3) $&$6 $&$ 1 $&$ 3 $&$[72] $&$(ii) $\\ \hline
17 & $\tA_2(q^2).\Phi_6$ &$  $&$6 $&$  $&$ 3 $&$6\times 3 $&$(i) $\\
18 & $^2\tA_2(q).\Phi_3.\Phi_6$ &$  $&$6 $&$  $&$ 3 $&$6\times 3 $&$(i) $\\
19 & $\Phi_3.\Phi_6^2$ &$  $&$6 $&$ $&$1 $&$G_{5} $&$(ii) $\\\hline
\end{tabular}
\caption{Relative Weyl groups in $\tE_6(q)_{\mathrm{sc}}$}
\end{center}
\end{table}

The observations at the beginning of the proof only leave the five Cases 2, 4, 6, 8 and 10 to be dealt with. In Cases 4 and 10 property  ($\ddagger$)(i) holds since $ W_{\w s} = W_s $ and (ii) is satisfied as $K_s/W_s$ has cyclic Sylow subgroups. Finally, for Cases 2, 6 and 8, explicit computations using Chevie show the claim. For example, in Case 2, $W_{\w s} = W_s$ is a parabolic subgroup of $W_3 = G_{25}$ of rank 1 and $K_s$ is its centralizer in $W_3$. Chevie shows that all characters of $W_s$ extend to their inertia group in $K_s$.

Now assume that $G = {}^2\tE_6(q)_{\mathrm{sc}}$. Here the occurring relative Weyl groups and normalizers are exactly as in $\tE_6(q)$, except that the cases of $d = 3$ and $d = 6$ have to be interchanged and their centralizers be replaced by their Ennola duals. Thus the claim follows from our previous arguments.

Finally, let $G = \tE_7(q)_{\mathrm{sc}}$. Recall that we only claim property  ($\ddagger$)(i). This is certainly satisfied if $W_{\w s} = W_s$. So we can restrict ourselves to looking at elements $s$ with disconnected centralizer. In particular $q$ is odd. The relevant cases for $d = 6$ are obtained from the ones for $d = 3$ by Ennola duality. This does not change the relative Weyl groups, so it suffices to consider $d = 3$. The occurring series are collected in Table 2, again omitting the case $s = 1$.

\begin{table}[h!]
\begin{center}
\label{tab:table1}
\begin{tabular}{|r|c|c|ccc|c|} 
\hline
\textrm{No.} & $\Cy_{G^*}(s)$ & $q$&  $W_{\w s}$& $W_{s}$& $K_s$& \textrm{case}\\
\hline
1 & $\tE_6(q).\Phi_1.2$ &$  $&$G_{25} $&$G_{26} $&$G_{26} $&$ (iii)$\\
2 & $\tA_2(q)^3.\Phi_1.2$ &$\equiv 1(6) $&$3^3 $&$3^3.2 $&$3^3.2^2 $&$ $\\
3 & $\tA_2(q)^2.\Phi_1.\Phi_3^3.2$ &$ $&$3 ^2$&$3^2.2 $&$ [108] $&$ $\\
\hline
4  & $\tA_2(q).\Phi_1\Phi_3^2.2$ &$  $&$3 $&$3.2 $&$6\times 6$&$(i) $\\
5 & $\Phi_1\Phi_3^2.2$ &$  $&$1  $&$2$&$[144] $&$(ii) $\\\hline
\end{tabular}
\caption{Relative Weyl groups in $\tE_7(q)_{\mathrm{sc}}$ , $q$ odd, $d=3$}
\end{center}
\end{table}

The Cases 1, 4 and 5 are clear by our general remarks. Again by explicit computation all characters of $W_s$ extend to $K_s$ in Cases 2 and 3. 

Finally, we consider $d = 4$ for $G= \tE_7(q)_{\mathrm{sc}}$. The number $d = 4$ is not regular for the Weyl group of type $\tE_7$, and the centralizer of a Sylow 4-torus has type $\Phi_4^2.\tA_1(q)^3$. Instead of discussing all the possibilities for disconnected centralizers of suitable semi-simple elements, we take another approach. Note that the subgroup $W_{\w s}$ is a subgroup of $W_4$ generated by reflections. Now it is easy to list all subgroups of $W_4$ generated by reflections; up to conjugation there are nine non-trivial ones, isomorphic to 

$$\Cy_2, \Cy_2 \times \Cy_2, \Cy_4, \Cy_2 \times \Cy_4, \Cy_4 \times \Cy_4, \G(2,1,2), \G(4,2,2), \G(4,1,2), W_4 = G_8.$$ Six of these have a non-abelian normalizer in $W_4$, namely the two abelian groups $\Cy_2 \times  \Cy_2$, $\Cy_4 \times \Cy_4$, and of course all of the non-abelian ones. Now $\G(4,1,2)$ and $ W_4$ are self-normalizing, so there is nothing to show. Furthermore, $\Cy_4\times \Cy_4$ has index 2 in its normalizer, so again this case need not be considered. We are thus left with three possibilities for $W_{\w s}$ , up to conjugation, namely $$\Cy_2 \times \Cy_2,\ \G(2,1,2),\ \text{and}\ \G(4,2,2).$$ The only overgroup of $\G(4,2,2)$ with index 2 is G$(4,1,2)$ which is self-normalizing, so this does not give rise to a case to consider. Now let's consider $W_{\w s} = \G(2,1,2)$. Here, $N := N_{W_4}(W_{\w s})$ is a Sylow 2-subgroup of $W_4$. Two of the linear characters and the character of degree 2 extend to $N$, while the other two linear characters have an inertia subgroup of index 2, so again we are done. Finally, assume that $W_{\w s} = \Cy_2 \times  \Cy_2$. In this case $N = N_{W_4}(W_{\w s})$ has order 32, and there is a unique intermediate normal subgroup $N_1$ containing $W_{\w s}$ of index 2. Two of the four linear characters of $W_{\w s} $ extend to $N$; the other two have $N_1$ as inertia group, so we are done again. \qed

\subsection{The regular case - Extended Weyl group} We apply the above to verify the necessary statements from Theorem~\ref{thm_4.4} about characters of the local subgroups in cases where $d$ is regular for $(\bG ,F)$.

\begin{prop}\label{ExcReg}
Let $d$ be a regular number for $(\bG , F)$ where $\bG$ has type $\tE_6$ or $\tE_7$ and $F\colon\bG\to\bG$ is a Frobenius endomorphism. Then \AA d and \BB d (see Definition~\ref{A(d)}) are satisfied by $\GG_{\mathrm{sc}}^F$.
\end{prop}

\begin{proof}
We may assume $d\geq 3$ thanks to \cite[\S~3]{MS16}. The remaining cases of regular $d\geq 3$
will be checked by applying Theorem~\ref{thm_4.4}. 

The assumptions (i), (ii) and (iii.1) of Theorem~\ref{thm_4.4} are ensured by Theorem D and Lemmas 8.1 and 8.2 of \cite{S09}. Note that while those statements apply to the normal inclusion $C\unlhd N$, they derive in  fact from the existence of an extension map (denoted there by $\Lambda '$) for the inclusion $H_d\unlhd V_d$  with the equivariance property.

Assumption (iii.2) is to be proved in the case of untwisted type $\tE_6$. Let us write $\wh E =\spann<\wh F_0>\times \spann<\gamma>$ in the notation of Sect.\,4. Let $\ww \in V$ be defined as in \cite[3.2(b)]{S09} so that $\gamma$ acts by conjugation by $\ww$ on $V$. Then $\wh V_d = V_dB$ for the abelian group $B:=\spann<\wh F_0>\spann<\ww\gamma>$. We see that $[B,V_d]=\{ 1\}$ as subgroups of $\GG_{\mathrm{sc}}^{F_0^{em}}\rtimes \wh E$. 

Let $\la\in \Irr(H_d)$ and $\Lambda_0(\la)\in\Irr ((V_d)_\la)$, this extension being ensured by what has been said before. We build an extension of $\Lambda_0(\la)$ to $(\wh V_d)_\la =( V_d)_\la B$ having $v\wh F$ in its kernel. Note that indeed $v$ centralizes $H_d$ hence stabilizes both $\la$ and $\Lambda_0(\la)$. In the abelian group $B\spann<v>$, the order of $\wh F$ is a divisor of the order of $v$ by definition of $\wh E$ in Sect.~4. So there is an extension $\mu$ of the linear character $\restr{\Lambda_0(\la)}|{( V_d)_\la\cap (B\spann<v>)}$ to $B\spann<v>$ such that $\mu (\wh F)=\Lambda_0(\la)(v)^{-1}$. Now since $\mu$ and $\Lambda_0(\la)$ coincide on $( V_d)_\la\cap (B\spann<v>)$, there is a common extension $\w \mu$ to the central product $( V_d)_\la . (B\spann<v>)$ such that $\w\mu (xy)=\Lambda_0(\la)(x)\mu(y)$ for any $x\in ( V_d)_\la$ , $y\in B\spann<v>$ (see for instance \cite[4.2]{S09}). Then $\w\mu$ is a character as required in (iii.2) of Theorem~\ref{thm_4.4}.

We explain how Proposition~\ref{Gunter} implies assumption (iv) of Theorem~\ref{thm_4.4}. For integers $d\geq 3$ such that $\Phi_d$ is present in the order of $\GG_{\mathrm{sc}}^F $ with power 1, one has that $\bS^F$ is cyclic \cite[25.7]{MT}. On the other hand, $\cent{\bG}{\bS}^F=\cent{\bG}{\bS^F}^F$ (by \cite[13.16(i), 13.17(ii)]{CE04} with regard to a prime dividing $\Phi_d(q)$)  so that $W_d=\norm{\bG}{\bS}^F/\cent{\bG}{\bS}^F$ is abelian since injecting in $\Aut(\bS^F)$. Then assumption~\ref{thm_4.4}(iv) is empty. So there only remain the integers $d\in \{ 3,4,6  \}$ for types $\tE_6$, $^2\tE_6$, $d\in\{3,6\}$ in type $\tE_7$. Note that Proposition~\ref{Gunter} is phrased using duality.
The duality between $\bG$ and $\GG^*$ provides a bijection $(\bT ,\theta)\leftrightarrow (\bT^*,s)$ between $\GG_{\mathrm{sc}}^F$-conjugacy classes of pairs where $\bT$ is a maximal $F$-stable torus of $\bG$ and $\theta\in \Irr (\bT^F)$ and on the other side $\GG^*{}^F$-conjugacy classes of pairs where $\bT^*$ is a maximal $F$-stable torus of $\GG^*$ and $s\in\bT^*{ }^F $ (see \cite[9.A]{B06}). This satisfies $\norm{\bG}{\bT,\theta }^F/\bT^F\cong \norm{\cent{\GG^*}{s}}{\bT^*}^F/\bT^*{}^F$ thanks to the description of centralizers in terms of roots. This is also compatible with a regular embedding $\GG_{\mathrm{sc}}\leq \w\GG$ and its dual $\GG_{\mathrm{ad}}^*
\twoheadleftarrow \w\GG^*$ (see \cite[9.C]{B06}). In our case of a character $\xi\in\Irr (\bC^F)$ where $\bC $ is the centralizer of a Sylow $d$-torus and $\bC$ is a maximal torus, the pair $(\bC ,\xi)$ relates with $(\bC^*{}, s)$ where $\bC^*$ centralizes a Sylow $d$-torus (since $\bC$ and $\bC^*$ must have the same polynomial order, being in $F$-duality). By duality $W_{\w\xi}\leq W_{\xi}\leq W_d:=\norm{\bG}{\bC}^F/\bC^F$ corresponds to $W_{\w s}\leq W_{s}\leq W_d^*:=\norm{\GG^*}{\bC^*}^F/\bC^*{}^F$, the groups studied in Proposition~\ref{Gunter} (see also \cite[Proof of 16.2]{B06}). Proposition~\ref{Gunter} tells us that, with the notations of \ref{thm_4.4}(iv), for any $\eta_0\in\Irr(W_{\w\xi})$ there is 

- a $K_{\eta_0}$-invariant $\eta\in\Irr(W_\xi \mid  \eta_0)$, and

- when $\GG_{\mathrm{sc}}^F$ is of type $\tE_6$, some $\eta$ as above extends to $K_{\eta_0}$. 

In types $^2\tE_6$ or $\tE_7$, assumption~\ref{thm_4.4}(iv) asks for some $\wh K_{\eta_0}$-invariant $\eta\in\Irr(W_\xi \mid  \eta_0)$. We know that $\wh F_0\in \wh N$ and it acts trivially on $W$, hence also on $W_d$ and therefore fixes $\eta$. In type $\tE_7$ this generates $\wh E$ so it is enough to get our claim. In type $^2\tE_6$, each element of $\wh E$ acts by a power of $\wh F_0$ on $W_d$ since $\gamma_0\wh F_0^m$ acts by $F$. So again $K_{\eta_0}$ acts as $\wh K_{\eta_0}$ on $W_\xi$.

In type $\tE_6$, we have seen above that $\wh V_d=V_d=\spann<\wh F_0>\spann<\ww\gamma>$. Then $\wh W_d=W_d\Cy_{\spann<\ww\gamma>}(v)\spann<\wh F_0>$ which is of the form $W_d\times B'$ with an abelian group $B'$. Then the group  $\wh K_{\eta_0}$ is $ K_{\eta_0}\times B'$ and statement~\ref{thm_4.4}(iv) is a trivial consequence of the two points we have.

All assumptions of Theorem~\ref{thm_4.4} are satisfied, so we get conditions \AA d and \BB d for each regular number $d$.
\end{proof}

\subsection{Proof of Theorem A for types $\tE$}

Let $\ell$ be a prime. We keep $\GG_{\mathrm{sc}}^F$ a finite quasi-simple group of type $\tE_6(q)$, $^2\tE_6(q)$ or $\tE_7(q)$, for $q$ a power of the prime $p$. We must show that the simple group $\GG_{\mathrm{sc}}^F/\Z(\GG_{\mathrm{sc}}^F)$ satisfies the inductive McKay condition of \cite{IMN}. If $\ell =p$, this is ensured by \cite[1.1]{S12}. If $\ell =2$, this is given by \cite[p. 905]{MS16}. So we assume that $\ell$ is an odd prime $\not=p$, dividing $|\GG_{\mathrm{sc}}^F|$ and we let $d$ be the order of $q$ mod $\ell$.

Assume $d\not=4$ if the type is $\tE_7$. If $d$ is a regular number then we have seen in Proposition~\ref{ExcReg} that \AA d and \BB d are satisfied, so Theorem A holds with regards to $\ell$ thanks to Theorem~\ref{gloBCE} and Theorem~\ref{Glo+Loc}. 
If $d$ is not regular, then $d=5$ for type $\tE_6$,  $d=10$ for type $^2\tE_6$, or  $d\in\{4,5,8,10,12\}$ for type $\tE_7$, as  recalled at the beginning of this section. Since $d\not=4$, one sees easily that $\Phi_d$ divides the polynomial order to the power 1 only and no other $\Phi_{d\ell^a}$ with $a\geq 1$ divides it. Then $|\GG_{\mathrm{sc}}^F|_\ell =|\bS^F|_\ell$ for $\bS$ a Sylow $d$-torus of rank 1. So all $\ell$-subgroups of $\GG_{\mathrm{sc}}^F$ are cyclic by \cite[25.7]{MT}. Now \cite[1.1]{KoSp16} implies that all $\ell$-blocks satisfy a block-by-block version of the so-called inductive AM-condition \cite[7.2]{S13}. We then have the inductive AM-condition for all $\ell$-blocks of our quasi-simple group. This implies the inductive McKay condition for that simple group and the prime $\ell$, as is easily seen by comparing \cite[7.2]{S13} and \cite[2.6]{S12}.

\medskip

It remains to check simple groups $\tE_7(q)$ for the odd primes $\ell\not= p$ such that the order of $q$ mod $\ell$ is 4. We prove \AA 4 and \BB 4 in order to apply Theorem~\ref{Glo+Loc}, knowing that \Ainfty\ holds thanks to Theorem~\ref{gloBCE}. For that, we check the assumptions of Theorem~\ref{thm_loc_gen}.

We use the notation of Sect.\,2.A. The simple roots being numbered $\Delta=\{ \al_1, \dots ,\al_7  \}$ as in \cite[Table 9.1]{MT}, one takes $$v:=\nn_{\al_3}(1)\nn_{\al_4}(1)\nn_{\al_2}(1)\nn_{\al_3}(1)\nn_{\al_4}(1)\nn_{\al_5}(1)\in V$$ (see Definition~\ref{VhatE}) and $\bS$ the Sylow 4-torus of $(\bT ,vF)$ as in \cite[\S~7.2]{S09}, \cite[\S~5.3]{S07}.

The requirements (i) and (iii) of Theorem~\ref{thm_loc_gen} are given by \cite[8.2]{S09}. 
We may assume $p$ is odd since otherwise one could take $\bG=\w\GG$ in Theorem~\ref{thm_loc_gen} and the other requirements of Theorem~\ref{thm_loc_gen} would then be trivial.
 
We use the general notations $N$, $\bC$, $C=\bC^{vF}$, $\wh N$, $\w C$, $\w N$ from Theorem~\ref{thm_loc_gen}. Note that $\wh N = N\times \wh E$ where $\wh E$ is cyclic of order $|V|m$ with generator $\wh F_0$ acting by $F_0$ on $\w\GG$.

 One denotes $\bT_1=\spann<\hh_{\al_i}(t)\mid t\in \FF^\times , 2\leq i\leq 5>$, $T_1=\bT_1^{vF}$, $G_2=[\bC ,
\bC]^{vF}$, $C_0=T_1G_2\unlhd C$ with index 4 and $\GG_{2,i}$ ($i= 1,2,3$) three $F$-stable subgroups of $[\bC,\bC]$ defined as in the proof below with isomorphisms $\iota_i: \SL_2(q)\to G_{2,i}:=\GG_{2,i}^F$ 
. We have $G_2:=[\bC,\bC]^F=G_{2,1}\times G_{2,2}\times G_{2,3}$ (see \cite[\S~5.3]{S07}). 

\begin{lem}\label{CtildeG2}
The image of $\w C$ in $\Out(G_2)$ is a $V_4\times\Cy_2$ where $V_4$ is the Klein 4-group and corresponds to the image of $C$. The action of the summand $\Cy_2$ is by a diagonal (non-interior) automorphism on $G_{2,1}$ and interior on $G_{2,2}$ and $G_{2,3}$. On the other hand each non-trivial element of $V_4$ acts by simultaneous diagonal  (non-interior) automorphisms on two of the $G_{2,i}$'s and interior on the third.
\end{lem} 
\begin{proof} The simple roots of $\bC$ with regard to $\bT$ are $\{ \beta_1:=\al_7, \beta_2 :=\al_2 +\al_3 +2\al_4 +2\al_5 +2\al_6 +\al_7 , \beta_3 :=2\al_1+2\al_2 +3\al_3 +4\al_4 +3\al_5 +2\al_6 +\al_7 \}$ (see \cite[5.3.4(a)]{S07}). Letting $\GG_{2,i}:=\spann<\xx_{\pm\beta_i}(\FF)>$ and $G_{2,i}=\GG_{2,i}^F$ , one has $[\bC ,\bC]=\GG_{2,1}\GG_{2,2}\GG_{2,3}$ (central product). Let $z_1= \hh_{\al_2}(-1) \hh_{\al_5}(-1)$, $z_2=\ \hh_{\al_2}(-1) \hh_{\al_3}(-1)$, and $z_3 :=\hh_{\al_2}(-1) \hh_{\al_5}(-1)\hh_{\al_7}(-1) $.  One has $\spann<z_1,z_2> =\bT_1\cap\GG_2\leq \Z(\GG_2)$, $\Z(\bG)=\spann<z_3>$ and $\Z(\bC)=\spann<z_1,z_2,z_3>$ (see \cite[5.3.4(c)]{S07}). Let $t_1, t_2, t_3\in\bT$ such that ${}(vF)(t_i)t_i^{-1}=z_i$.
Then the action of each $t_i$ on $[\bC ,\bC]^{vF}$ can be read off on the corresponding $(vF)(t_i)t_i^{-1}\in \Z(\bC)$. The two first act as said about the factor $V_4$ and $t_1t_3$ as said about the factor $\Cy_2$ since the $z_i$'s can be rewritten as $z_1=\hh_{\beta_2}(-1)\hh_{\beta_3}(-1)$, $z_2=\hh_{\beta_1}(-1)\hh_{\beta_2}(-1)$ and $z_3=\hh_{\beta_1}(-1)\hh_{\beta_2}(-1)\hh_{\beta_3}(-1)$ (use \cite[p. 94]{S07}) thus giving their projections on the $\GG_{2,i}$'s. \end{proof}

We now check assumption (ii) of Theorem~\ref{thm_loc_gen}. Let $\calT ^\circ$ be a $\GL_2(q)$-transversal in $\Irr(\SL_2(q))$ that is stable under field automorphisms on matrix entries. The existence of $\calT^\circ$ is ensured by the fact that $\SL_2(q)$ satisfies condition \Ainfty, see \cite[4.1]{CS17A}. Denote $\calT_i=\iota_i(\calT^\circ)$ for $i=1,2,3$, $\calT:=\Irr(C\mid \calT_1\times\calT_2\times\calT_3 )$. This defines an $N\wh E$-stable set and is a $\wt C$-transversal. 

We get \ref{thm_loc_gen}(ii) once we show that $$(\wh N\w N)_{N.\xi} =\wh N_{N.\xi} \w N_{N.\xi}  $$
for any $\xi\in\calT$.
Let $\xi\in\calT$. The above equation holds if, for every $n\in N$, $e\in \wh N$, $d\in \w N$ with

$$\xi ^{ned}=\xi   \text{  there exist  }n_1, n_2\in N \text{ such that } \xi^{n_1e}=\xi=\xi^{n_2d}.\eqno(*)$$

By the choice of $\calT_0$, the characters of $\calT_1\times\calT_2\times\calT_3$ satisfy (*). 

 Recall that $U_6:=\spann<\nn_{\al_i}\mid 1\leq i\leq 6>^{vF}$ satisfies  $N=U_6C$ (see \cite[5.3.6(b)]{S07}).   
Recall that on $G_2=G_{2,1}\times G_{2,2}\times G_{2,3}$, $\wh E$ acts by $F_0$, and $U_6$ acts by permutation of the summands $G_{2,i}\cong \SL_2(q)$  (see \cite[5.3.6]{S07}).

Observe first that maximal extendibility holds with respect to $C_0\unlhd \w C$. Indeed, denote $\calL(x)=x^{-1}\cdot (vF)(x)$ for $x\in \bG$. We have that $T_1\unlhd  T_1':=\{x\in\bT_1\mid  \calL(x)\in \bT_1\cap\GG_2 \}$ are abelian, $G_{2,i}\unlhd  G'_{2,i}$ with cyclic quotients for $ G'_{2,i} =\{ x\in\GG_{2,i}\mid \calL(x)\in \Z (\GG_{2,i})  \}$. This shows that $ C':=\{x\in \bC\mid \calL(x)\in \Z(\bG)   \}$ is a subgroup of the central product $L':= T'_1 . G'_{2,1}. G'_{2,2}. G'_{2,3}$. We have maximal extendibility with respect to the inclusion $C_0\unlhd  L'$. Now $\w C$ is a subgroup of $L'\Z(\w\GG)^{F^2}$ so we deduce maximal extendibility with respect to $C_0\unlhd \w C$.

Let $\xi_0=\xi_{0,1}\times\xi_{0,2}\times\xi_{0,3} \in \calT_1\times \calT_2\times \calT_3\subseteq \Irr (G_2)$ such that $\xi\in\Irr(C\mid\xi_0)$. Assume first that  $C_{\xi_0}\gneq C_0$. By the action of $C$ on $G_2$ described in Lemma~\ref{CtildeG2}, then at least two of the $\xi_{0,i}$'s are $\GL_2(q)$-invariant. Then, by the action of $\w C$ also described in Lemma~\ref{CtildeG2}, we have $C\w C_{\xi_0}=\w C$. By Clifford theory and the extendibility mentioned above, there is $\w \xi_0\in\Irr (C_{\xi_0}\mid \xi_0)$ with $\w\xi_0^C =\xi$. Then $\w C_\xi =C\w C_{\xi_0} =\w C$ and we get that $\xi^d=\xi^{n'}$ for some $n'\in N$ in (*) above since $\w N=N\w C$. This implies (*) in this case.

Assume $C_{\xi_0}= C_0$. Then $\xi =(\w\xi_0)^C$ where $\w\xi_0$ is an extension of $\xi_0$ to $C_0$. Then the stabilizer of $N.\xi$ in $\w N\wh N$ is the one of $N.\xi_0$ and assumption (ii) of Theorem~\ref{thm_loc_gen} in the form of (*) above comes from the fact that $\xi_0$ itself satisfies (*) as pointed out before.

We now check requirement \ref{thm_loc_gen}(iv). The proof of Proposition~\ref{Gunter} shows in the present case that for every pair $W_{\w s}\unlhd W'$ of subgroups of $W_4$ with $\w s\in\w\bC^*{}^F$ and $|W'/W_{\w s}|\leq 2$, and any $\eta\in\Irr(W_{\w s})$ there exists $\eta'\in\Irr(W'\mid\eta )$ which is fixed under $\norm{W_4}{W_{\w s},W'}_{\eta }$. Since on the other hand $\wh E$ acts trivially on $W$, this implies our claim once we check that the subgroups $(W_4)_{\w\xi}\unlhd (W_4)_{\xi}$ for $\xi\in\Irr( C)$, $\w\xi\in\Irr(\w C_\xi\mid \xi)$ are of the above type $W_{\w s}\unlhd W'$. For the first one, note that $W_4$ acts as $U_6$ by permuting the components of type $\tA_1$. Then \ref{thm_loc_gen}(iv) comes from the fact that if $\w\xi\in\cE (\w C,\w s)$, then one may have $W_{\w \eta}\neq W_{\w s}$ only if $W_{\w s}$ permutes summands of type $\tA_1$ where $\w\eta$ has distinct unipotent correspondent via Jordan decomposition $\cE (\w C,\w s)\leftrightarrow \cE(\cent{\w\bC^*}{\w s}^F,1) $. But then it is easy to find some $\w s'$ with $W_{\w s'}=W_{\w \eta}$. For the second subgroup $W'$ it is clear that $(W_4)_{\w\xi}\unlhd (W_4)_{\xi}$ with index $\leq 2$ since $\restr{\w\xi}|{C}$ is $\xi$ or the sum of $\xi$ and a conjugate. 

We now have all requirements of Theorem~\ref{thm_loc_gen}, so that \AA 4 and \BB 4 are satisfied by any $\bG$ of type $\tE_7$. As said before this finishes our proof of Theorem~\ref{thmA} for finite simple groups of Lie types $\tE_6$, $^2\tE_6$ and $\tE_7$. \qed

\newpage

\end{document}